\documentclass[11pt]{amsart}

\usepackage{amsxtra,amssymb,amsmath,amscd,url,listings,stmaryrd,mathabx}

\usepackage[alphabetic]{amsrefs}
\usepackage[utf8]{inputenc}
\usepackage{eucal}
\usepackage{fullpage}
\usepackage{scrtime}
\usepackage[colorlinks]{hyperref}
\usepackage{datetime}
\usepackage[all]{xy}
\usepackage{graphicx}
\usepackage{tikz-cd}

\shortdate

\renewcommand{\leq}{\leqslant}
\renewcommand{\geq}{\geqslant}
 
\numberwithin{equation}{section}

\newcommand{\uple}[1]{\text{\boldmath${#1}$}}

\def\stacksum#1#2{{\stackrel{{\scriptstyle #1}}
{{\scriptstyle #2}}}}

\newcommand{\CGM}{$\mathrm{CGM}$}
\newcommand{\CGMT}{$\mathrm{NIO}$}

\newcommand{\bfb}{\uple{b}}

\newcommand{\bfR}{\mathbf{R}}

\newcommand{\bfK}{\mathbf{K}}

\newcommand{\Cc}{\mathbf{C}}

\newcommand{\Aa}{\mathbf{A}}

\newcommand{\Zz}{\mathbf{Z}}
\newcommand{\Pp}{\mathbf{P}}
\newcommand{\Rr}{\mathbf{R}}
\newcommand{\Gg}{\mathbf{G}}
\newcommand{\Gm}{\mathbf{G}_{m}}

\newcommand{\Qq}{\mathbf{Q}}
\newcommand{\Fp}{{\mathbf{F}_p}}

\newcommand{\Fq}{{\mathbf{F}_q}}
\newcommand{\Fqt}{{\mathbf{F}^\times_q}}
\newcommand{\Fqd}{{\mathbf{F}_{q^d}}}

\newcommand{\Ff}{\mathbf{F}}

\newcommand{\bFq}{\overline{\Ff}_q}
\newcommand{\bQl}{\overline{\Qq}_{\ell}}

\newcommand{\mmu}{\boldsymbol{\mu}}
\newcommand{\Oc}{\mathcal{O}}
\newcommand{\mcV}{\mathcal{V}}
\newcommand{\mcW}{\mathcal{W}}

\newcommand{\mcN}{\mathcal{N}}

\newcommand{\HYPK}{\mathcal{K}\ell}
\newcommand{\KL}{\mathcal{K}\ell}

\newcommand{\mods}[1]{\,(\mathrm{mod}\,{#1})}

\DeclareMathOperator{\frob}{Fr}

\DeclareMathOperator{\hypk}{Kl}

\newcommand{\ftchi}{f\otimes\chi}
\def\sump{\mathop{\sum \Bigl.^{+}}\limits}

\newcommand{\ra}{\rightarrow}
\newcommand{\lra}{\longrightarrow}

\newcommand{\fleche}[1]{\stackrel{#1}{\lra}}

\DeclareMathOperator{\Spec}{Spec}

\DeclareMathOperator{\rank}{rank}
\DeclareMathOperator{\rk}{rk}

\DeclareMathOperator{\Imag}{Im}
\DeclareMathOperator{\Reel}{Re}

\DeclareMathOperator{\Frob}{\mathrm{Fr}}

\DeclareMathOperator{\Kl}{\mathrm{Kl}}

\DeclareMathOperator{\tr}{\mathrm{Tr}}
\DeclareMathOperator{\nr}{\mathrm{Nr}}
\DeclareMathOperator{\Gal}{Gal}

\DeclareMathOperator{\Tr}{Tr}
\DeclareMathOperator{\Hom}{Hom}
\DeclareMathOperator{\End}{End}

\DeclareMathOperator{\swan}{Swan}
\DeclareMathOperator{\Sing}{Sing}

\DeclareMathOperator{\codim}{codim}

\DeclareMathOperator{\cond}{\mathbf{c}}

\DeclareMathOperator{\dual}{D}

\newcommand{\eps}{\varepsilon}
\renewcommand{\rho}{\varrho}

\DeclareMathOperator{\SL}{SL}

\DeclareMathOperator{\GL}{GL}
\DeclareMathOperator{\PGL}{PGL}

\DeclareMathOperator{\Sp}{Sp}
\DeclareMathOperator{\GSp}{GSp}
\DeclareMathOperator{\SO}{SO}

\newcommand{\demi}{{\textstyle{\frac{1}{2}}}}

\newcommand{\sheaf}[1]{\mathcal{{#1}}}

\newcommand{\ppersp}{\Pi}

\newcommand{\Qbar}{\bar{\Qq}}

\DeclareMathSymbol{\gena}{\mathord}{letters}{"3C}
\DeclareMathSymbol{\genb}{\mathord}{letters}{"3E}

\def\sump{\mathop{\sum \Bigl.^{+}}\limits}

\def\multsum{\mathop{\sum\cdots \sum}\limits}

\def\intc{\frac{1}{2i\pi}\mathop{\int}\limits}

\theoremstyle{plain}
\newtheorem{theorem}{Theorem}[section]
\newtheorem*{theorem*}{Theorem}
\newtheorem{lemma}[theorem]{Lemma}
\newtheorem{corollary}[theorem]{Corollary}

\newtheorem{proposition}[theorem]{Proposition}

\newtheorem*{TI}{(TI)}

\theoremstyle{remark}

\newtheorem*{rem}{Remark}

\theoremstyle{definition}

\newtheorem{definition}[theorem]{Definition}

\newtheorem{example}[theorem]{Example}
\newtheorem{remark}[theorem]{Remark}

\newcommand{\mcM}{\mathcal{M}}
\newcommand{\mcL}{\mathcal{L}}

\newcommand{\mcC}{\mathcal{C}}
\newcommand{\mcF}{\mathcal{F}}
\newcommand{\mcK}{\mathcal{K}}
\newcommand{\mcR}{\mathcal{R}}

\newcommand{\mcG}{\mathcal{G}}
\newcommand{\mcB}{\mathcal{B}}
\newcommand{\mcE}{\mathcal{E}}

\newcommand{\lf}{\lambda_f}

\renewcommand{\geq}{\geqslant}
\renewcommand{\leq}{\leqslant}

\newcommand{\ov}[1]{\overline{#1}}

\newcommand\sumsum{\mathop{\sum\sum}\limits}
\newcommand\sumsumsum{\mathop{\sum\sum\sum}\limits}

\begin{document}

\title[Stratification and averaging]{Stratification and averaging for
  exponential sums : bilinear forms with generalized Kloosterman sums}

\author{Emmanuel Kowalski}
\address{ETH Z\"urich -- D-MATH\\
  R\"amistrasse 101\\
  CH-8092 Z\"urich\\
  Switzerland} \email{kowalski@math.ethz.ch}

\author{Philippe Michel} \address{EPFL/SB/TAN, Station 8, CH-1015
  Lausanne, Switzerland } \email{philippe.michel@epfl.ch}

\author{Will Sawin}
\address{Columbia University, 2990 Broadway, New York, NY, USA 10027}
\email{sawin@math.columbia.edu}

\begin{abstract}
  We introduce a new comparison principle for exponential sums over
  finite fields in order to study ``sum-product'' sheaves that arise
  in the study of general bilinear forms with coefficients given by
  trace functions modulo a prime $q$. When these functions are
  hyper-Kloosterman sums with characters, we succeed in establishing
  cases of this principle that lead to non-trivial bounds below the
  Pólya-Vinogradov range.  This property is proved by a subtle
  interplay between étale cohomology in its algebraic and diophantine
  incarnations. We give a first application of our bilinear estimates
  concerning the first moment of a family of $L$-functions of degree
  $3$.
\end{abstract}

\subjclass[2010]{11T23, 11L05, 11N37, 11N75, 11F66, 14F20, 14D05}

\keywords{generalized Kloosterman sums, Kloosterman sheaves,
  monodromy, Riemann Hypothesis over finite fields, short exponential
  sums, families of $L$-functions, first moment}

\thanks{Ph.\ M.\ and E.\ K.\ were partially supported by a DFG-SNF
  lead agency program grant (grants 200021L\_153647 and
  200020L\_175755). W.S. partially supported by Dr. Max R\"ossler, the
  Walter Haefner Foundation and the ETH Zurich Foundation. Ph.M. and W.S. were partially supported by NSF
Grant No. DMS-1440140. \today\
  \currenttime}

\maketitle 

\setcounter{tocdepth}{1}
\tableofcontents

\section{Introduction}

\subsection{Presentation of the results}

Let $q\geq 1$ be an integer and let $K(\cdot;q)$ be a complex-valued
$q$-periodic arithmetic function. A recurrent problem in analytic
number theory is to evaluate how such functions correlate with other
natural arithmetic functions $f(n)$, where $f$ could be the
characteristic function of an interval, or that of the primes, or the
Fourier coefficients of some automorphic form. When facing such
problems, one is often led to the problem of bounding non-trivially
some bilinear forms
$$
B(K,\uple{\alpha},\uple{\beta})=\sumsum_{m\leq M,n\leq
  N}\alpha_m\beta_n K(mn;q),
$$
where the ranges of the variables $M,N\geq 1$ usually depend on $q$,
and $\uple{\alpha}=(\alpha_m)_{m\leq M}$,
$\uple{\beta}=(\beta_n)_{n\leq N}$ are complex numbers which,
depending on the initial problem, are quite arbitrary.  One of the
main objectives is to improve on the trivial bound
$$
\|K\|_\infty\|\uple\alpha\|_2\|\uple\beta\|_2(MN)^{1/2}
$$ 
for ranges of $M$ and $N$ that are as small as possible compared to
$q$; indeed, this uniformity is often more important than the strength
of the saving compared to the trivial bound. 
\par
A natural benchmark is the \emph{P\'olya-Vinogradov} method, which
often provides non-trivial bounds as long as $M, N\geq
q^{1/2}$. Indeed, obtaining a result below that range is usually
extremely challenging. When the modulus $q$ is composite, a number of
techniques exploiting the possibility of factoring $q$ (starting with
the Chinese Remainder Theorem) become available, and results exist in
fair generality.
\par
In this paper, we will only consider the case where $q$ is a prime,
and when $K$ is a trace function (see~\cite{FKMSurvey} for a
background survey).
 
The landmark result in this setting is the work of Burgess~\cite{Bur},
which provides a non-trivial bound for the sum
$$
\sum_{n\leq N}\chi(n)
$$
when $\chi$ is a non-trivial Dirichlet character modulo $q$ and
$N\geq q^{3/8+\eta}$, for any $\eta>0$. This is therefore well below
P\'olya-Vinogradov range. The ideas of Burgess (especially the ``$+ab$
shifting trick'') combine successfully the multiplicativity of $\chi$
and the (almost) invariance of intervals by additive translations. 

Another twist of Burgess's method was given by the works of Karatsuba
and Vinogradov, Friedlander-Iwaniec \cite{FI} and subsequently
Fouvry-Michel \cite{FoMi} to bound non-trivially the bilinear sums
$B(K,\uple{\alpha},\uple{\beta})$ for various choices of functions $K$
and ranges $M,N$ shorter than $q^{1/2}$. In particular, using some
version of the Sato-Tate equidistribution laws due to Katz
\cite{ESDE}, Fouvry and Michel considered
$$
K(x;q)=e\Bigl(\frac{x^k+ax}q\Bigr),\ k\in\Zz-\{0,1,2\},\
a\in\Fqt,\quad (x,q)=1,
$$
and proved that for any $\delta>0$, there exists $\eta>0$ such that,
\begin{equation}\label{nontrivialbilineargeneral}
\sumsum_{m\leq M,n\leq N}\alpha_m\beta_n K(mn;q)
  \ll \|\uple{\alpha}\|_2\|\uple{\beta}\|_2(MN)^{1/2-\eta}	
\end{equation}
as long as
\begin{equation}\label{eqburgessrange}
  M,N\geq q^\delta\text{ and }MN\geq q^{3/4+\delta}.	
\end{equation}
The condition $MN \geq q^{3/4+ \delta}$ is believed to be a barrier in
this setting analogous to the condition $N > q^{1/4+\delta}$ in the
Burgess bound for short character sums.
 
In our previous paper~\cite{KMS}, motivated by the study of moments of
$L$-functions (especially in our papers with Blomer, Mili\'cevi\'c and
Fouvry~\cite{BFKMM, BFKMMS}), we obtained bounds of type
\eqref{nontrivialbilineargeneral} when $K(\cdot;q)$ is a
hyper-Kloosterman sum, namely
$$
\Kl_k(x;q)=\frac{1}{q^{\frac{k-1}{2}}}
\sum_\stacksum{y_1,\cdots,y_k\in\Fqt}{y_1\cdots y_k=x}
e\Bigl(\frac{y_1+\cdots+y_k}{q}\Bigr)
$$
where $k\geq 2$ is some fixed integer. More precisely, we proved that
\eqref{nontrivialbilineargeneral} holds as long as
$$
M,N\geq q^\delta\text{ and }MN\geq q^{7/8+\delta}
$$ 
for some $\delta>0$. The argument was delicate and quite difficult.

In this second paper, we introduce a new approach that is both more
robust and more powerful. The main complete exponential sum that needs
to be bounded in this general setting is a difference of two
exponential sums, which in previous work was bounded by estimating
separately the main terms on both sides. Here, we show that the two
underlying cohomology groups are equal, hence the main terms cancel,
without explicitly calculating them. To establish the desired
cohomological comparison, we define a stratification of the parameter
space, and show using vanishing cycles that if the result fails at any
point of one of the strata, it fails on the generic point. Using a
variant of Katz's diophantine criterion of irreducibility, this
implies that the original exponential sum estimate fails on average
over the stratum. We check that the strata are defined by equations of
a specific type, which makes the averaged exponential sum estimate
amenable to classical analytic techniques, specifically separation of
variables.

\begin{remark}\label{rm-xu}
  As the referee pointed out to us, a similar stratification strategy
  is present in the paper~\cite{xu} of J. Xu on multiplicative
  character sums, where the key applications are related to
  multi-variable Burgess estimates. The main differences are that in
  Xu's method the stratification is more abstract, whereas for us it
  is explicit, and Xu's method relies on the higher moments of the
  exponential sums, while we use only the first moment.
\end{remark}

Our main application in this paper is the proof of the estimate
\eqref{nontrivialbilineargeneral} in the full range
\eqref{eqburgessrange} for generalized hyper-Kloosterman sums with
character twists, whose definition we now recall.  Let $k\geq 1$ be an
integer, and let $\uple{\chi}=(\chi_1,\ldots,\chi_k)$ be a tuple of
$k$ Dirichlet characters modulo $q$, each of which might be
trivial. The $(k-1)$-dimensional generalized Kloosterman sums
associated to $\uple{\chi}$ are the exponential sums defined for
$x\in\Fqt$ by
$$
\Kl_k(x;\uple{\chi},q)=\frac{1}{q^{\frac{k-1}{2}}}
\sum_\stacksum{y_1,\cdots,y_k\in\Fqt}{y_1\cdots y_k=x}
\chi_1(y_1)\cdots\chi_k(y_k)e\Bigl(\frac{y_1+\cdots+y_k}{q}\Bigr).
$$
The hyper-Kloosterman sums (which correspond to the case $\chi_i=1$)
were introduced by Deligne \cite{sga4h}, and these generalisations
were introduced and studied by Katz in~\cite[Ch. 4]{GKM}. As an
application of the Riemann Hypothesis over finite fields, Deligne and
Katz established the highly non-trivial pointwise bounds
$$
|\Kl_k(x;\uple{\chi},q)|\leq k.
$$
\par
The finer properties of these sums were studied in great depth by Katz
in \cite{GKM} and~\cite{ESDE}. Among other things, Katz proved
equidistribution statements that describe precisely the distribution
of generalized Kloosterman sums inside $\Cc$, at least for most
possible choices of $\uple{\chi}$.

A special case of our main result, Theorem \ref{thmtypeIIprecise}, is
the following:

\begin{theorem}\label{1stthmII}
  Assume that $\uple{\chi}$ has Property \CGMT\ of
  \emph{Definition~\ref{def-cgm}}, for instance all $\chi_i$ are
  trivial.  For any $\delta>0$ there exists $\eta>0$ such that for any
  integer $k\geq 2$, any prime number $q$, and any integers
  $M,N\geq 1$ such that
  $$M,\ N\geq q^\delta,\ MN\geq q^{3/4+\delta}$$
  we have
  $$
  \sumsum_{m\leq M,n\leq N}\alpha_m\beta_n \Kl_k(amn;\uple{\chi},q)
  \ll \|\uple{\alpha}\|_2\|\uple{\beta}\|_2(MN)^{1/2-\eta}
  $$ 
  for any $a\in\Fqt$ and for arbitrary families of complex numbers
  $\uple{\alpha}=(\alpha_m)_{m\leq M}$ and
  $\uple{\beta}=(\beta_n)_{n\leq N}$. The implied constant depends
  only on $\delta $ and $k$.
\end{theorem}

Property \CGMT\ (short for ``Not Induced or Orthogonal'') is an
elementary combinatorial property that we define below in
Section~\ref{sec-notation}; it is easy to check, and it is
``generically'' satisfied in some sense. For instance, the case
$\uple{\chi}=(1,\ldots, 1)$ corresponding to hyper-Kloosterman sums
themselves has \CGMT, and so does $(1,\ldots,1,\chi)$ if $k$ is odd.

The exponent $3/4=2\times3/8$ seem to be a recurring barrier: it
occurs in classical subconvexity estimates for $L$-functions, and more
recently (see \cites{FKM1,FKM2}) when dealing with sums of the shape
$$\sum_\stacksum{p\leq N}{p\ prime}K(p;q)$$
where $p$ ranges over prime numbers, or
$$\sum_{n\leq N}\lf(n)K(n;q)$$
where $K$ is a general trace function modulo $q$ and
$(\lf(n))_{n\leq N}$ are the Hecke eigenvalues of a fixed Hecke
eigenform $f$ (cuspidal or Eisenstein).

For
special bilinear forms, where one of the variables is smooth, i.e., for
$$
B(K,\uple{\alpha},1_N)=\sumsum_{m\leq M,n\leq N}\alpha_m K(mn;q).
$$
the barrier occurs at a shorter range, and we again are able to prove
an estimate that reaches this barrier.

A special case of Theorem \ref{thmtypeIprecise} is:

\begin{theorem}\label{1stthmI} 
  Assume that $\uple{\chi}$ has \CGMT. For any $\delta>0$ there exists
  $\eta>0$ such that for $k\geq 2$ an integer, $q$ a prime and
  $M,N\geq 1$ some integers satisfying
 $$M,N\geq q^\delta,\ MN^2\geq q^{1+\delta}$$ we have
 $$
 \sumsum_{m\leq M,n\leq N}\alpha_m \Kl_k(amn;\uple{\chi},q) \ll
 \|\uple{\alpha}\|_2(MN^2)^{1/2-\eta}
$$ 
for any $a\in\Fqt$ and for any tuple of complex numbers
$\uple{\alpha}=(\alpha_m)_{m\leq M}$, where the implicit constant
depends on $\delta $ and $k$.
\end{theorem}

In particular, for $M=N$, we obtain a non-trivial bound as long as
$$
M=N\geq q^{1/3+\delta}
$$
for some $\delta>0$. If we denote by $d_2(n)$ the classical  divisor
function, we deduce the following result:

\begin{corollary}
  Assume that $\uple{\chi}$ has \CGMT.  For any $\delta>0$, there
  exists $\eta>0$ such that for any integer $k\geq 2$, any prime
  number $q$, and any $N\geq q^{2/3+\delta}$, we have
$$
\sum_{n\leq N}d_2(n)\Kl_k(an;\uple{\chi},q)\ll Nq^{-\eta},
$$
for any $a\in\Fqt$ where the implicit constant depends on $\delta $ and $k$.
\end{corollary}

It is of considerable interest to generalize results like
Theorem~\ref{1stthmII} to other trace functions $K$ modulo $q$. We
believe that the methods in this paper could be applicable when $K$
satisfies suitable big monodromy assumptions, and has the following
property: $K$ belongs to a family $K_{a}$ parameterized by non-trivial
additive characters $x\mapsto e(ax/p)$ of $\Fq$, and this family
satisfies a relation of the type $K_{a^{\mu}}(x)=K(a^{\nu}x)$ for some
fixed non-zero integers $\mu$ and $\nu$. For instance, this holds for
the generalized Kloosterman sums with $\mu=1$, $\nu=k$ when defining
$$
K_a(x)=\frac{1}{q^{\frac{k-1}{2}}}
\sum_\stacksum{y_1,\cdots,y_k\in\Fqt}{y_1\cdots y_k=x}
\chi_1(y_1)\cdots\chi_k(y_k)e\Bigl(\frac{a(y_1+\cdots+y_k)}{q}\Bigr).
$$

\subsection{Applications to moments of $L$-functions}

As with our previous paper \cite{KMS}, Theorems \ref{1stthmII} and
\ref{1stthmI} have applications to the evaluation of moments of
$L$-functions indexed by Dirichlet characters modulo $q$. As a simple
illustration, we will prove in Section~\ref{Zacsec} the following
result, which generalizes some recent work of Zacharias~\cite{Zac}:

\begin{theorem}\label{zacthm} 
  Let $f$ be a primitive holomorphic cusp form of level $1$. For $q$
  prime, let $\xi$ be a non-trivial Dirichlet character modulo
  $q$. There exist an absolute constant $\delta>0$ such that
$$
\frac{1}{q-1}\sum_{\chi\mods
  q}L(\ftchi,1/2)L(\xi\chi,1/2)=1+O_f(q^{-\delta}).
$$
\end{theorem}

\begin{remark} 
  Zacharias established this asymptotic for $\xi=1$ using amongst
  other ingredients the bounds from \cite{KMS} for
  $K(x)=\Kl_3(x;(1,1,1),q)$; he evaluated more generally a mollified
  version of this average, enabling him to establish that, for $q$
  large, there is a positive proportion of $\chi\mods q$ such that
  $L(\ftchi,1/2)$ and $L(\chi,1/2)$ are both non-vanishing. Most
  likely a similar result may be established in our case.
\end{remark}

As in \cite{BFKMM,KMS,BFKMMS}, we also expect that our results will
prove useful to estimate other averages of certain $L$-functions of
degree $3$ and $4$ indexed by Dirichlet characters. For instance, we may consider:
\begin{itemize}
\item The twisted first moment 
$$
\frac{1}{q-1}\sum_{\chi\mods
  q}L(\ftchi,1/2)L(\xi\chi,1/2)\prod_{i}\eps_{\xi_i\chi}^{k_i}$$ where
$\uple{\xi}=(\xi_i)_i$ a tuple of characters of modulus $q$ (possibly
trivial) and $\uple{k}=(k_i)_i$ is a family of integers;
\item The shifted second moment
$$
\frac{1}{q-1}\sum_{\chi\mods q}L(\ftchi,1/2)L(f\otimes\xi\ov\chi,1/2).
$$
\end{itemize}

\subsection{Principle of the stratification and averaging method}

We denote $K(x)=\Kl_k(ax;\uple{\chi},q)$ for a fixed $k$-tuple
$\uple{\chi}$ with Property \CGMT\ and a fixed $a\in\Fqt$.

As in our previous work (and \cites{FI,FoMi}), the proof starts with
an application of the $+ab$-shifting trick of Karatsuba and
Vinogradov. Let us recall that the shifting trick builds on the almost
invariance of an interval under sufficiently small translations. The
interval to be shifted here is that of the $n$ variable (either
directly for Theorem \ref{1stthmI} or after an application of Cauchy's
inequality for Theorem \ref{1stthmII}) and the shift is by product
$+ab$ with $(a,b)\in[A,2A[\times[B,2B[$ for $A,B$ suitable parameters
(such that $AB=N$). As $K(mn;q)$ depends only on the congruence class
of $mn\mods q$ the replacement of $n\leftrightarrow n+ab$ leads to the
following transformations
\begin{gather*}
  mn\mods q\leftrightarrow m(n+ab)=am(\ov an+b)=s(r+b)\mods q,
  \\
  (m_1n,m_2n)\mods q\leftrightarrow (am_1(\ov an+b),am_2(\ov
  an+b))=(s_1(r+b),s_2(r+b))\mods q
\end{gather*}
with $(r,s)$, $(r,s_1,s_2)$ taking values in $\Fq\times \Fqt$ or
$\Fq\times ({\Fqt}^2-\Delta({\Fqt}^2))$. Under suitable assumptions on
$A,M,N$ one can then show that the above maps are essentially
injective (i.e. have fibers bounded in size by $q^{o(1)}$). However,
these maps are far from being surjective, so performing such a change
of variable will result in a loss. This can be tamed by an application
of the Hölder inequality with a sufficiently large exponent, which we
denote by $2l$ in the sequel. This process leads then to the problem
of bounding sums of the shape
$$
\sum_{\bfb\in\mcB}\bigl|\Sigma_{I}(K,\bfb)\bigr|,\quad
\sum_{\bfb\in\mcB}\bigl|\Sigma_{II}(K,\bfb)\bigr|,
$$
where $\mcB$ denotes the set of $2l$-uples of integers
$\uple{b}=(b_1,\cdots,b_{2l})\in[B,2B[^{2l}$ and
\begin{align*}
  \Sigma_I(K,\bfb)&=\sum_{r\in\Fq}\sum_{s \in\Fqt}
                    \bfK(sr,s\bfb),
  \\
  \Sigma_{II}(K,\bfb)&=\sum_{r\in\Fq}\sumsum_\stacksum{
                       s_1,s_2\in\Fqt}{s_1\not =
                       s_2}\bfK(s_1r,s_1\bfb)\ov{\bfK(s_2r,s_2\bfb)}
\end{align*}
where
\begin{equation}\label{eq-bfk}
  \bfK(r,\bfb)=\prod_{i=1}^lK(r+b_i) \ov{K(r+b_{i+l})}.
\end{equation}

The goal is to give 
individual bounds for sums $\Sigma_{I}(K,\bfb)$ and
$\Sigma_{II}(K,\bfb)$
with square-root cancellation, namely we wish to prove that
$$
\Sigma_I(K,\bfb)\ll q,\quad\quad \Sigma_{II}(K,\bfb)\ll q^{3/2}.
$$
A key fact is that these bounds do not always hold, but it will be
enough to prove them outside a sufficiently small subset $\mcB^{diag}$
of ``diagonal'' tuples $\bfb$. This subset will be the set of
$\Fq$-points of a proper algebraic subvariety $\mcV^{diag}\subset
\Aa_{\Fq}^{2l}$. In fact, it is crucial (to avoid the loss involved in
Hölder's inequality) to prove the required estimates outside of a
variety $\mcV^{diag}$ with large codimension, 
and we will do this with
\begin{equation}\label{codimlower}
  \codim(\mcV^{diag})\geq \frac{l-1}{2}.
\end{equation}
The outcome is that by taking $l$ very large, we obtain non-trivial
estimates of $B(K,\uple{\alpha},\uple{\beta})$ and
$B(K,\uple{\alpha},1_N)$ in the ranges defined by
$$MN\geq q^{3/4+\delta}\hbox{ and }MN^2\geq q^{2/3+\delta}$$
for any $\delta>0$.

We now sketch the proof in the case of general bilinear forms (the
special bilinear forms are easier).  Setting
$$
\bfR(r,\bfb)=\sum_{s\in\Fqt}\bfK(sr,s\bfb)
$$ 
we observe that
$$
\Sigma_{II}(K,\bfb)= \sum_{r\in\Fq}|\bfR(r,\bfb)|^2-\sum_{s\in\Fqt}
\sum_{r\in\Fq}|\bfK(sr,s\bfb)|^2.
$$
This is the difference of two sums of positive terms, which therefore
individually will have main terms, and we need these main terms to
compensate exactly for $\bfb\notin\mcV^{diag}$. Our argument for this
in~\cite{KMS} relies on separate evaluations of both sums to witness
the coincidence of the main terms. But one can check that this
evaluation only holds outside of a codimension $1$ subvariety, which
is far from~(\ref{codimlower}) except in the case $l=2$. 

In this paper, we compare directly the two terms in the
difference. This comparison is not a combinatorial or analytic
rearrangement of terms, but is a cohomological comparison using the
ideas of $\ell$-adic cohomology to interpret exponential sums. Using
this formalism, we interpret the functions
$$
(r,\bfb)\ra \bfK(r,\bfb),\ \bfR(r,\bfb)
$$
as trace functions of $\ell$-adic sheaves $\mcK$ and $\mcR$ on
$\Aa\times\Aa^{2l}$, which are pointwise pure of weight $0$ and mixed
of weight $\leq 1$ respectively. The functions
$$
(r,\bfb)\ra |\bfK(r,\bfb)|^2,\ |\bfR(r,\bfb)|^2
$$
are the trace functions of the endomorphisms sheaves $\End(\mcK)$ and
$\End(\mcR)$. By means of the Grothendieck--Lefschetz trace formula
and of Deligne's most general form of the Riemann Hypothesis over
finite fields~\cite{WeilII}, the desired bound
$$
\sum_{r\in\Fq}|\bfR(r,\bfb)|^2-
\sum_{s\in\Fqt}\sum_{r\in\Fq}|\bfK(sr,s\bfb)|^2\ll q^{3/2}
$$
for a given $\bfb$ can be interpreted as stating that the specialized
sheaves $\mcK_\bfb$ and $\mcR_\bfb$ have decompositions into
geometrically irreducible components whose multiplicities precisely
match.

This interpretation relies on the relationship between $\mcK_{\bfb}$
and $\mcR_{\bfb}$. As $\mcR_{\bfb}$ is obtained from applying a
cohomology functor to $\mcK_{\bfb}$, each irreducible component $\rho$
of $\mcK_{\bfb}$ defines a summand $\widetilde{\rho}$ of
$\mcR_{\bfb}$. We check explicitly that these summands
$\widetilde{\rho}$ are nontrivial, which implies that the exponential
sums match if and only if all the summands $\widetilde{\rho}$ are
themselves irreducible, and are pairwise non-isomorphic as $\rho$
varies.

The sheaf $\mcR_{\bfb}$ is a sheaf on the affine line, lisse away from
a finite set of singular points that vary depending on $\bfb$. Using
Deligne's semicontinuity theorem, and assuming that the local
monodromy of $\mcR_{\bfb}$ is tame, we can show that the decomposition
into irreducible components of $\mcR_{\bfb}$ is constant on any set of
parameters $\bfb$ over which this varying finite set $S_{\bfb}$ of
singular points does not itself develop singularities (i.e., over
which the size of $S_{\bfb}$ is constant). The tameness condition can
be verified for large primes (which is sufficient for us) by
expressing $\mcR_{\bfb}$ as the characteristic $p$ fiber of a sheaf
defined in characteristic zero. This reduces the problem to the
generic points of the strata of the stratification of the parameter
space by the number of singular points in $S_{\bfb}$.

To get a handle on this stratification, we first calculate the set of
singular points. By an explicit inductive argument, we show how the
strata can be expressed by equations in the coefficients $b_i$ and
auxiliary variables; these equations split into sums of different
terms involving different subsets of the $b_i$'s. We can then estimate
the average of the complete sums $\Sigma_{II}(K,\bfb)$ over a single
stratum using only estimates for one-variable exponential sums, as
long as the number of equations and auxiliary variables is not too
large (which means that we must keep control of these numbers in the
inductive argument).  From this average estimate and the geometric
interpretation, we deduce that the sheaves $\mcK_{\bfb}$ and
$\mcR_{\bfb}$ have the same decomposition into irreducibles when
$\bfb$ belongs to a stratum of sufficiently large dimension. This
proves the desired result for all $\bfb$ except those in
low-dimensional strata, which we simply consider as part of
``diagonal'' subset. It is therefore crucial that our induction is
efficient enough to get a good bound on the codimension of this
subset.

Stratifications where the validity of a desired estimate on a stratum
only depends on its validity at the generic point exist for arbitrary
families of complete exponential sums, arising from the stratification
of a constructible $\ell$-adic sheaf into lisse sheaves. They can
often be computed by vanishing cycles methods, such as Deligne's
semicontinuity theorem. We expect that proving estimates for
individual strata by passing to the average and applying elementary
analytic methods (which are known to perform very well when the number
of variables to average over is large enough) will be a useful
strategy for many families of exponential sums.

\begin{remark}
  (1) It would be reasonable to expect that the correct codimension is
$$
\codim (\mcV^{diag})\geq l+o(l)
$$ 
as $l\to +\infty$, which would indeed be best possible (it is easy to
see that the codimension is $\leq l$). A lower bound of this quality
was established in \cite{FoMi} in the case $K(x)=e((x^k+a)/q)$ already
mentioned. Although the bound~(\ref{codimlower}) only goes half of the
way to this expectation, it is nevertheless sufficient for our
purpose, and it seems that even the full lower bound would not help in
improving the exponents $3/4$ and $2/3$ in Theorems~\ref{1stthmII}
and~\ref{1stthmI}.
\par
(2) Readers who have some familiarity with either~\cite{FoMi}
or~\cite{KMS} will have noticed that we will make a compromise in our
argument: the new variables $s$ and $(s_1,s_2)$ belong to the subsets
$[A,2AM]$ and $[A,2AM]^2-\Delta([A,2AM]^2)$ of the larger sets $\Fqt$
or ${\Fqt}^2-\Delta({\Fqt}^2)$, so that we lose something by
``forgetting'' this fact by positivity. It is certainly possible to
compensate for this loss using the completion method, introducing
additional twists by additive characters in the $s$-variable, and
handling them by arguments similar to those of \cite[\S
4.5]{KMS}. However, when $l$ is very large, the improvement in the
final bounds is very small (because of \eqref{codimlower}), and more
importantly the final limiting exponents $3/4$ and $2/3$ are not
improved. So we have chosen to avoid the completion step, in order to
simplify an already complex argument. It should be noted however that,
for small values of $l$, the completion step is worth pursuing, and
that is was crucial in \cite{KMS} to obtain non-trivial bounds for
$l=2$ (which was the only case that could be handled in \cite{KMS},
because, as noted earlier, the diagonal variety in that paper was of
codimension $1$).
\end{remark}

\subsection*{Notation}

For any prime number $\ell$, we fix an isomorphism
$\iota\,:\, \bQl\to \Cc$.  Let $q$ be a prime number. Given an
algebraic variety $X_{\Fq}$, a prime $\ell\not=q$ and a constructible
$\bQl$-sheaf $\mcF$ on $X$, we denote by
$t_{\mcF}\,:\, X(\Fq)\lra \Cc$ its trace function, defined by
$$
t_{\mcF}(x)=\iota(\Tr(\frob_{x,\Fq}\mid \mcF_{x})),
$$
where $\mcF_x$ denotes the stalk of $\mcF$ at $x$. More generally, for
any finite extension $\Fqd/\Fq$, we denote by $t_{\mcF}(\cdot;\Fqd)$
the trace function of $\mcF$ over $\Fqd$, namely
$$
t_\mcF(x;\Fqd)=\iota(\tr(\Frob_{x,\Fqd}\mid \mcF_x)).
$$
\par
An $\ell$-adic sheaf will aways means a $\bQl$-sheaf. For standard
facts in $\ell$-adic cohomology (such as proper base change,
cohomological dimension, etc), we refer to the books of Fu~\cite{Fu}
and Milne~\cite{Milne}, and to the notes of Deligne~\cite{sga4h}.
\par
We will usually omit writing down $\iota$. In any expression where
some element $z$ of $\bQl$ has to be interpreted as a complex number,
we mean to consider $\iota(z)$.
\par
We denote by $\sheaf{F}^{\vee}$ the dual of a constructible sheaf
$\mcF$; if $\mcF$ is a middle-extension sheaf, we will use the same
notation for the middle-extension dual.
\par
Let $\psi$ (resp. $\chi$) be a non-trivial additive
(resp. multiplicative) character of $\Fq$. We denote by $\mcL_\psi$
(resp. $\mcL_\chi$) the associated Artin-Schreier (resp. Kummer) sheaf
on $\Aa^1_\Fq$ (resp. on $(\Gg_m)_{\Fq}$), as well (by abuse of
notation) as their middle extension to $\Pp^1_\Fq$. The trace
functions of the latter are given by
\begin{gather*}
  t_\psi(x;\Fqd)=\psi(\tr_{\Fqd/\Fq}(x))\quad\text{ if } x\in\Fqd,\quad
  t_{\psi}(\infty;\Fqd)=0,\\
  t_\chi(x;\Fqd)=\chi(\nr_{\Fqd/\Fq}(x))\quad\text{ if }
  x\in\Ff_{q^d}^{\times},\quad
  t_{\chi}(0;\Fqd)=t_{\chi}(\infty;\Fqd)=0
\end{gather*}
For the trivial additive or multiplicative character, the trace
function of the middle-extension is the constant function $1$.
\par
Given $\lambda\in\Fqd$, we denote by $\mcL_{\psi_\lambda}$ the
Artin-Schreier sheaf of the character of $\Fqd$ defined by
$x\mapsto \psi(\tr_{\Fqd/\Fq}(\lambda x))$.
\par
If $X_{\Fq}$ is an algebraic variety, $\psi$ (resp. $\chi$) is an
$\ell$-adic additive character of $\Fq$ (resp. $\ell$-adic
multiplicative character) and $f\,:\, X\lra \Aa^1$ (resp.
$g\,:\, X\lra \Gg_m$) is a morphism, we denote by either
$\sheaf{L}_{\psi(f)}$ or $\sheaf{L}_{\psi}(f)$ (resp. by
$\sheaf{L}_{\chi(g)}$ or $\sheaf{L}_{\chi}(g)$) the pullback
$f^*\sheaf{L}_{\psi}$ of the Artin-Schreier sheaf associated to $\psi$
(resp. the pullback $g^*\sheaf{L}_{\chi}$ of the Kummer sheaf). These
are lisse sheaves on $X$ with trace functions $x\mapsto \psi(f(x))$
and $x\mapsto \chi(g(x))$, respectively. The meaning of the notation
$\sheaf{L}_{\psi}(f)$, which we use when putting $f$ as a subscript
would be typographically unwieldy, will always be unambiguous, and no
confusion with Tate twists will arise.
\par
\par
Given a variety $X/\Fq$, an integer $k\geq 1$ and a function $c$ on
$X$, we denote by $\mcL_{\psi}(c s^{1/k})$ the sheaf on
$X\times \Aa^1$ (with coordinates $(x,s)$) given by
$\alpha_* \mcL_{\psi (c(x)t)}$, where $\alpha$ is the covering map
$(x,s,t)\mapsto (x,s)$ on the $k$-fold cover
$$
\{(x,s,t)\in X\times\Aa^1\times\Aa^1\,\mid\, t^k=s\}.
$$
\par
Given a field extension $L/\Fp$, and elements $\alpha\in L^{\times}$
and $\beta\in L$, we denote by $[\times\alpha]$ the scaling map
$x\mapsto \alpha x$ on $\Aa^1_L$, and by $[+\beta]$ the additive
translation $x\mapsto x+\beta$. For a sheaf $\mcF$, we denote by
$[\times\alpha]^*\mcF$ (resp. $[+\alpha]^*\mcF$) the respective
pull-back operation. 
\par
We will usually not indicate base points in \'etale fundamental
groups; whenever this occurs, it will be clear that the properties
under consideration are independent of the choice of a base point.

\subsection*{Acknowledgments.} 

Ph. M. and W. S. thank the students of the Arizona Winter School 2016
who studied the paper~\cite{KMS} and worked on generalizing some of
its steps. They also thank MSRI where parts of this work were
completed during the ``Analytic Number Theory'' programme during the
first semester 2017.
\par
We thank \'E. Fouvry and I. Shparlinski for comments, and the referee
for useful remarks, especially for the reference to the
paper~\cite{xu} of J. Xu.
\par
We thank A. Florea for pointing out an inaccuracy in the statement of
Theorem~\ref{thmtypeIIprecise} and the following remark. 

\section{Preliminaries}\label{sec-notation}

We begin by defining Property \CGMT, and a useful variant called \CGM\
(for ``Connected Geometric Monodromy''). These are motivated by
results of Katz (see~\cite[Cor. 8.9.2, Th. 8.8.1--8.8.2]{ESDE}).

\begin{definition}\label{def-cgm}
  Let $A$ be a finite cyclic group and
  $\uple{\chi}=(\chi_1,\ldots,\chi_k)$ a tuple of characters of
  $A$. Let $\Lambda=\chi_1\cdots \chi_k$.
  \par
\begin{enumerate}
\item The tuple $\uple{\chi}$ is \emph{Kummer-induced} if there exists
  a divisor $d$ of $k$, $d\not=1$, and a tuple
  $(\xi_{1},\ldots,\xi_{k/d})$ of characters of $A$ such that the
  $\chi$'s are all the characters with $\chi^d=\xi_j$ for some $j$,
  with multiplicity.
\item The tuple $\uple{\chi}$ is \emph{self-dual} if there is a
  character $\xi$ such that the set of characters
  $\chi\in\uple{\chi}$, with multiplicity, is stable under
  $\chi\mapsto \xi\chi^{-1}$. The character $\xi$ is called a
  ``dualizing character''.
\item A self-dual tuple $\uple{\chi}$ is \emph{alternating} if $k$ is
  even and $\Lambda=\xi^{k/2}$, and otherwise, it is
  \emph{symmetric}.
\item A tuple $\uple{\chi}$ has Property \CGMT\ if it is not
  Kummer-induced and, if $k$ is even, if it is not self-dual
  symmetric.
\item A tuple $\uple{\chi}$ has Property \CGM\ if it is not
  Kummer-induced, and $\chi_1\cdots\chi_k=1$, and one of the following
  conditions holds:
\begin{itemize}
\item $k$ is odd,
\item $\uple{\chi}$ is not self-dual,
\item $k$ is even, $\uple{\chi}$ is self-dual and alternating, and the
  dualizing character $\xi$ is trivial.
\end{itemize} 
\end{enumerate}
\end{definition}

\begin{example}\label{ex-nio}
  We consider Dirichlet characters modulo $q$ in these examples.
\par
  (1) Consider the case $k=2$ and $q$ odd,
  $\uple{\chi}=(\chi_1,\chi_2)$. Denote by $\chi_{(2)}$ the
  non-trivial real character of $\Fqt$. Then $\uple{\chi}$ is:
\begin{itemize}
\item Kummer-induced if and only if $\chi_2=\chi_1\chi_{(2)}$.
\item If not Kummer-induced, always self-dual alternating, taking
  $\xi=\chi_1\chi_2$ as dualizing character.
\end{itemize}

In particular, for $\uple{\chi}=(1,\chi_2)$, the alternating case is
$\chi_2=1$, corresponding to the ``classical'' Kloosterman sum, and
the non self-dual case is $\chi_2^2\not=1$. The Kummer-induced tuple
$\uple{\chi}=(1,\chi_{(2)})$ corresponds to Salié sums.
\par
(2) If $k$ is odd, then $\uple{\chi}$ has \CGMT\ if and only if it is
not Kummer-induced. In particular, this is the case if
$\chi_1=\ldots=\chi_{k-1}=1$.
\par
(3) If $\chi_1=\cdots=\chi_k=1$, then $\uple{\chi}$ has \CGMT.
\end{example}

In the next section, we will need the following useful lemma which
bounds the number of integral points in a box that satisfy a system of
polynomial equations modulo $q$. We thank the referee for giving us a
convenient reference.

\begin{lemma}\label{lm-sz}
  Let $k\geq 1$ be an integer and let $A>0$. Let
  $X_{\Zz}\subset \Aa^k_{\Zz}$ be an algebraic variety of dimension
  $d\geq 0$ given by the vanishing of $\leq A$ polynomials of degree
  $\leq A$. Let $p$ be a prime number and $0\leq B<p/2$ an
  integer. Then
$$
|\{x=(x_1,\ldots, x_k)\in \Ff_p^k\,\mid\, x\in X(\Ff_p)\text{ and }
B\leq x_i\leq 2B\text{ for } 1\leq i\leq k\}|\ll B^d
$$
where the implied constant depends only on $k$ and $A$, and the
notation $B\leq x_i\leq 2B$ means that the unique integer between $1$
and $p-1$ congruent to $x_i$ modulo $p$ belongs to the interval
$[B,2B]$.
\end{lemma}

See~\cite[Lemma~1.7]{xu} for a proof.



\section{An application to moments of $L$-functions}	\label{Zacsec}

In this section, we will prove Theorem~\ref{zacthm}, which we recall
is a variation of a recent result of Zacharias \cite{Zac}.

Let $f$ be a primitive cusp form of level $1$, trivial nebentypus and
weight $k_f$, with Hecke eigenvalues $\lambda_f(n)$. For Dirichlet
characters $\chi$ and $\xi$ modulo $q$, we consider the $L$-function
$$
L((f\oplus \xi)\otimes\chi,s)=L(\ftchi,s)L(\chi\xi,s)
$$
of degree $3$. Note that for $\Reel(s)>1$, we have the Dirichlet
series expansion
$$
L((f\oplus \xi)\otimes\chi,s)= \sum_{n\geq
  1}\chi(n)(\lambda_f\star\xi)(n)n^{-s}.
$$
\par
We wish to evaluate the average
$$
\mathcal{M}=\frac{1}{q-1}
\sum_{\chi\mods q}L((f\oplus\xi)\otimes\chi,1/2),
$$
proving that $\mathcal{M}=1+O(q^{-\alpha})$ for some $\alpha>0$.

The proof is very similar to \cite[\S 6.2]{Zac}, which corresponds
to the case $\xi=1$, so we will only sketch certain steps.

We assume for simplicity that $\xi$ is even (ie. $\xi(-1)=1$), and we
will only evaluate the \emph{even} moment
$$
\mathcal{M}^+=
\frac{2}{q-1}\sump_{\chi\mods q}L((f\oplus\xi)\otimes\chi,1/2)
$$
where $\sump$ restricts the sum to even primitive characters modulo
$q$. We will prove that $\mathcal{M}^+=\demi+O(q^{-\alpha})$ for some
$\alpha>0$. The sum over odd characters satisfies the same
asymptotics, hence this implies Theorem~\ref{zacthm}.
\par
Define $\Gamma_\Rr(s)=\pi^{-s/2}\Gamma(s/2)$ and let
$L_{\infty}(s)=L_{\infty}(f,s)L_{\infty}(\chi\xi,s)$, where
$$
L_\infty(f,s)
=\Gamma_\Rr\Bigl(s+\frac{k-1}2\Bigr)\Gamma_\Rr\Bigl(s+\frac{k+1}2\Bigr),
\quad L_\infty(\chi\xi,s)=\Gamma_\Rr(s),
$$
are the archimedean $L$-factors of $L(f,s)$ and $L(\chi\xi,s)$
respectively. Further, let
$$
\eps((f\oplus\xi)\otimes\chi)=\eps(f)\eps_\chi^2\eps_{\chi\xi}
$$
where $\eps_{\eta}$ denotes the normalized Gauss sum of a Dirichlet
character. 
Define then the completed $L$-function
$$
\Lambda((f\oplus\xi)\otimes\chi,s)=
q^{3s/2}L_\infty(s)L((f\oplus\xi)\otimes\chi,s).
$$
For $\xi$ and $\chi$ even, we then have  the functional equation
$$
\Lambda((f\oplus\xi)\otimes\chi,s)=
\eps((f\oplus\xi)\otimes\chi)\Lambda(f\otimes\ov\chi\oplus\ov{\chi\xi},1-s).
$$
\par
Let $0<\alpha<1/4$ be a parameter to be fixed later.  For $\chi$ even,
non-trivial and not equal to $\xi^{-1}$, we apply the approximate
functional equation to the $L$-function
$L((f\oplus\xi)\otimes \chi,s)$, in an unbalanced form
(\cite[Th. 5.3]{IwKo} with $q$ replaced by the conductor $q^3$ and
$X=q^{1/2-2\alpha}$). After adding the contribution of the character
$\xi^{-1}$, which is $\ll q^{-1/5+\eps}$ for any $\eps>0$, this gives
$\mathcal{M}^+=\mathcal{M}_1+\mathcal{M}_2$, where
\begin{align*}
  \mathcal{M}_1&=\frac{2}{q-1}\sump_{\chi\mods
       q}\sum_{n\geq 1}\frac{\chi(n)(\lf\star\xi)(n)}{n^{1/2}}
       \mcV\Bigl(\frac{n}{q^{2-2\alpha}}\Bigr),\\
  \mathcal{M}_2&=\frac{2}{q-1}\sump_{\chi\mods
       q}\eps((f\oplus\xi)\otimes\chi)
       \sum_{n\geq 1}\frac{\ov{\chi(n)}
       \ov{(\lf\star\xi)(n)}}{n^{1/2}}
       \mcV\Bigl(\frac{n}{q^{1+2\alpha}}\Bigr),
\end{align*}
where the function $\mcV$ is defined by
$$
\mcV(y)=\intc_{(1)}
\frac{L_\infty(\frac12+s)}{L_\infty(\frac12)}G(s)y^{-s}\frac{ds}{s},
\quad\quad G(s)=\exp(s^2)
$$
for $y>0$.  Shifting the $s$-contour to the right if $y\geq 1$ or to
$\Reel(s)=-1/2$ if $y\leq 1$, we deduce that
$$
y^i\mcV^{(i)}(y)\ll_{A,i,f} (1+y)^{-A}
$$
for any $A>0$ and $i\geq 0$, and
$$
\mcV(y)=1+O(y^{1/2})\text{ for }y\leq 1.
$$
It follows from the first of these bounds that, for any $\kappa>0$,
the contribution to both sums of the integers $n\geq q^{3/2+\kappa}$
is $\ll_{A,f,\kappa} q^{-A}$ for any $A\geq 0$.

We first bound $\mathcal{M}_1$.  We add to $\mathcal{M}_1$ the
contribution of the trivial character, up to an error term bounded by
$O(q^{-1/5})$, and perform the summation over the even characters
$\chi$. We obtain
$$
\mathcal{M}_1=\sum_{n\equiv\pm 1\mods
  q}\frac{(\lf\star\xi)(n)}{n^{1/2}}\mcV\Bigl(\frac{n}{q^{2-2\alpha}}\Bigr)
+O(q^{-1/5}) =\mcV\Bigl(\frac{1}{q^{2-2\alpha}}\Bigr)
+O(q^{-\alpha+\eps})=1+O(q^{-\alpha+\eps}),
$$
for any $\eps>0$, where the first term $\mcV(q^{-2+2\alpha})$ is the
contribution of the trivial solution $n=1$ of the congruence
$n\equiv \pm 1\mods q$.

Now we consider $\mathcal{M}_2$.  We add to $\mathcal{M}_2$ the
contribution of the trivial character, up to an error of size
$\ll q^{\eps+\frac12+\alpha-1}\ll q^{\eps+\alpha-1/2}$, for any
$\eps>0$. We then perform the summation over $\chi$ even. We have
$$
\frac{1}{q-1} \sump_{\chi\mods
  q}\eps((f\oplus\xi)\otimes\chi)\ov\chi(n)=
\frac{\eps(f)}{q-1}\sump_{\chi\mods
  q}\eps_\chi^2\eps_{\chi\xi}\ov\chi(n)=\frac{\eps(f)}{q^{1/2}}
\Bigl(\Kl_3(n; \xi,q)+\Kl_3(-n;\xi,q)\Bigr),
$$
where we abbreviate
$$
\Kl_3(\pm n; \xi,q)=\Kl_3(\pm n;(1,1,\xi ),q).
$$
Hence we have
$$
  \mathcal{M}_2=
  \frac{\eps(f)}{q^{1/2}}
  \sum_{n}\frac{\ov{(\lf\star\xi)(n)}}{n^{1/2}}
  (\Kl_3(n; \xi,q)+\Kl_3(-n;\xi,q))
  \mcV\Bigl(\frac{n}{q^{1+2\alpha}}\Bigr)+O(q^{-1/5}).
$$
We open the Dirichlet convolution
$$
\overline{(\lf\star\xi)(n)}=\sum_{ab=n}\lf(a)\ov{\xi(b)}.
$$
By standard techniques (dyadic subdivisions, inverse Mellin transform
to separate the variables), we establish that $\mathcal{M}_2$ is, up
to a factor $\ll q^{\eps}$ for any $\eps>0$, bounded by the sum of
$\ll (\log q)^2$ bilinear sums of the type
$$
\mathcal{M}_2(M,N)=
\frac{1}{(qMN)^{1/2}}\sum_{m,n}{\lf(m)\ov{\xi(n)}}\Kl_3(a mn;
\xi,q)V\Bigl(\frac{m}{M}\Bigr)W\Bigl(\frac{n}{N}\Bigr)
$$
where
$$
1\leq MN\leq q^{1+2\alpha},
$$
$a=1$ or $-1$, and $V$ and $W$ are smooth functions, compactly
supported in $[1,2]$, such that
$$
x^iV^i(x),\ x^iW^i(x)\ll_{f,\eps} q^{i\eps}
$$
for any $\eps>0$ and $i\geq 0$.

We set $M=q^\mu$ and $N=q^\nu$. The trivial bound is
$$
\mathcal{M}_2(M,N)\ll
q^{\eps}\Bigl(\frac{MN}{q}\Bigr)^{1/2}=q^{(\mu+\nu)/2-1/2+\eps}
$$
for any $\eps>0$, which is $\ll q^{-\alpha+\eps}$ if
$\mu+\nu\leq 1-2\alpha$.  Now assume that
$$
1-2\alpha\leq\mu+\nu\leq 1+2\alpha.
$$
Estimating the sum over $n$ by the Pólya-Vinogradov technique
(completion), summing trivially over the $m$ variable,  we obtain
$$
\mathcal{M}_2(M,N)\ll q^{\eps}\Bigl(\frac {M}{N}\Bigr)^{1/2}\ll
q^{1/2-\nu+\alpha+\eps}
$$
for any $\eps>0$. This bound is $\ll q^{-\alpha+\eps}$ if
$\nu\geq \demi+2\alpha$. We then assume that
$$
\nu\leq \demi+2\alpha.
$$
\par
If $\nu$ is small, so that $\mu$ is large, we apply
\cite[Th. 1.2]{FKM1} to the sum over $m$, summing trivially over
$n$. We get
$$
\mathcal{M}_2(M,N)\ll
Nq^{-1/8+\alpha+\eps}=q^{-1/8+\nu+\alpha+\eps}
$$
for any $\eps>0$. Again, this is $\ll q^{-\alpha+\eps}$ provided
$\nu\leq \tfrac{1}{8}-2\alpha$. Now assume that
$$
\tfrac{1}{8}-2\alpha\leq \nu\leq \demi+2\alpha.
$$
Then $\demi-4\alpha\leq \mu\leq \tfrac{7}{8}+4\alpha$. The general
bilinear form estimate in~\cite[Thm 1.17]{FKM2} gives
$$
\mathcal{M}_2(M,N)\ll q^{\eps+\alpha} \min(N^{-1}+M^{-1}q^{1/2},
M^{-1}+N^{-1}q^{1/2})^{1/2}
$$
which is $\ll q^{-\alpha+\eps}$ provided $\alpha\leq 1/32$ and
$$
\max(\mu,\nu)\geq \demi+2\alpha.
$$
We finally consider the case when $\alpha\leq 1/32$ and
$$
\demi-4\alpha\leq \mu,\ \nu\leq \demi+2\alpha.
$$
In this situation, we can then apply Theorem \ref{1stthmII} for the
triple $\uple{\chi}=(1,1,\xi)$, which has Property \CGMT\ for any
$\xi$ by Example~\ref{ex-nio} (2). We obtain the bound
$$
\mathcal{M}_2(M,N)\ll
q^{\eps}\Bigl(\frac{MN}{q}\Bigr)^{1/2}(MN)^{-\eta}\ll
q^{2\alpha+\eps}(MN)^{-\eta}\ll q^{2\alpha-3\eta/4+\eps}
$$
for any $\eps>0$, where $\eta>0$ is the saving exponent in
Theorem~\ref{1stthmII} when the parameter $\delta$ there is
$\delta=\tfrac{1}{4}-8\alpha$. Hence, for $\alpha>0$ fixed and small
enough, we obtain
$$
\mathcal{M}_2(M,N)\ll q^{-\eta'+\eps}
$$
for some fixed $\eta'>0$ and any $\eps>0$, where the implied constant
depends on $\eps$ and $f$.

\section{Reduction to complete exponential sums}

In this section, we will state the general forms of Theorems
\ref{1stthmII} and \ref{1stthmI}, and reduce their proofs to certain
bounds for families of exponential sums over finite fields. In fact,
we begin with slightly more general bilinear sums.

Let $q$ be a prime number, and let $K:\Fq \ra \Cc$ be any function.
Let $M,N$ be integers such that $1\leq M,N\leq q-1$. Let $\mcM$ be a
subset of the positive integers $m\leq q-1$ of cardinality $M$. We set
$M^+=\max_{m\in \mcM} m$. Let finally 
$$
\mcN=\{n\,\mid\, 1\leq n<N\}.
$$

Given tuples of complex numbers $\uple{\alpha}=(\alpha_m)_{m\in\mcM}$
and $\uple{\beta}=(\beta_n)_{n\in \mcN}$, we set
$$
B(K,\uple{\alpha},\uple{\beta})=\sumsum_{m\in \mcM,\,n\in
  \mcN}\alpha_m\beta_n K(mn).
$$
We will prove the following:

\begin{theorem}\label{thmtypeIIprecise} Fix an integer $k\geq 2$. Let
  $q$ be a prime and let $a\in\Fqt$. Let $\uple{\chi}$ be a $k$-tuple
  of Dirichlet characters modulo $q$. Suppose that $\uple{\chi}$ has
  Property \CGMT, and define $K(x)=\Kl_k(ax;\uple{\chi},q)$. With
  notations as above, for any integer $l\geq 2$ and any $\eps>0$, we
  have
$$
B(K,\uple{\alpha},\uple{\beta})\ll
q^\eps\|\alpha\|_2\|\beta\|_2
(MN)^{1/2}\Bigl(\frac{1}{M}
+\Bigl(\frac{q^{\frac{3}{4}+\frac{3}{4l}}}{MN}\Bigr)^{\frac1{l}}\Bigr)
^{1/2},
$$
where the implied constant depends only on $(k,l,\eps)$, provided one
of the two following two conditions holds:
\begin{gather*}
q^{\frac3{2l}}\leq N<\frac12 q^{\frac12+\frac3{4l}},
\\
q^{\frac3{2l}}\leq N,\ NM^+<\frac12q^{1+\frac3{2l}}.
\end{gather*}
\end{theorem}

\begin{rem} 
  This bound is non-trivial only for $l$ large enough (for
  $M=N=q^{1/2}$, this happens precisely for $l\geq 4$). As we will
  explain, this limitation results from our simplifying choice of not
  applying the completion method to detect that an auxiliary variable
  belongs to some interval in $\Fq$.
\end{rem}

In the special case of ``type I'' sums, we obtain

\begin{theorem}\label{thmtypeIprecise}  
  With the same notation and assumption as in
  Theorem~\emph{\ref{thmtypeIIprecise}}, especially assuming that
  $\uple{\chi}$ has \CGMT, and with the additional condition that
  $\beta_n=1$ for $n\in\mcN$, for any integer $l\geq 1$ and any
  $\eps>0$, we have
$$
B(K,\uple{\alpha},\uple{1}) \ll
q^{\eps}\|\uple{\alpha}\|_1^{1-\frac1l}\|\uple{\alpha}\|_2^{\frac1l}
M^{\frac1{2l}}N\Bigl(\frac {q^{1+\frac{1}{l}}}{MN^{2}}\Bigr)^{1/2l},
$$
  where the implied constant depends on $(k,l,\eps)$, provided one of
  the following two conditions holds:
\begin{gather*}
q^{\frac1{l}}\leq N\leq \frac12q^{1/2+1/2l},
\\
q^{\frac1{l}}\leq N,\ NM^+\leq \frac12q^{1+1/2l}.
\end{gather*}
\end{theorem}

\begin{rem}
  As $l$ gets large, this bound is non-trivial if
$$M^+N\leq q,\ MN^2\geq q^{1+\delta}$$
for some $\delta>0$. In particular for $M=M^+=N$, this is non trivial
if
$$N\geq q^{1/3+\delta}.$$	
\end{rem}

\subsection{The type II bilinear sum}

We now start the proof of the reduction step for Theorem
\ref{thmtypeIIprecise}.

Applying Cauchy's inequality, we obtain
$$
|B(K,\uple{\alpha},\uple{\beta})|\leq
\|\beta\|_2\Bigl(\sum_{n}|\sum_{m}\alpha_m K(mn)|^2\Bigr)^{1/2}\ll
\|\beta\|_2(\|\alpha\|_2^2N+S^{\not=})^{1/2}
$$
where
$$
S^{\not=} =\sum_{m_1\not=m_2}\alpha_{m_1}\ov{\alpha_{m_2}}\sum_n
K(m_1n)\ov{K(m_2n)}.
$$
\par
We now use the $+ab$-shift trick of Karatsuba-Vinogradov as in
\cites{FoMi,KMS}. For this we introduce two integer parameters
$A,B\geq 1$ such that $AB\leq N$. Using the notation $a\sim A$ for
$A\leq a<2A$, we then have
$$
S^{\not=}=\frac{1}{AB}\sum_{a\sim A,b\sim
  B}\sumsum_{m_1\not=m_2}\alpha_{m_1}\ov{\alpha_{m_2}}
\sum_{n+ab\in\mcN}K(m_1(n+ab))\ov{K(m_2(n+ab))}.
$$
Using the fact that $\mcN$ is an interval, we deduce as
in~\cite[p. 126, (7.2)]{FoMi} (see also~\cite[(2.11)]{KMS} that
$$
S^{\not=}\ll\frac{\log q}{AB}\sumsum_\stacksum{a,m_1,m_2,n}{a\sim
  A,m_1\not=m_2}|\alpha_{m_1}{\alpha_{m_2}}|\ \Bigl|\sum_{b\sim
  B}K(m_1(n+ab))\ov{K(m_2(n+ab))}e(bt)\Bigr|
$$
for some $t\in \Rr$ and $n$ varying over an interval of length
$\ll N+AB$.  For $(r,s_1,s_2)\in (\Fqt)^3$ set
$$
\nu(r,s_1,s_2)=\sumsum_\stacksum{a,m_1\not=m_2,n}{a\sim A,\ov an\equiv
  r, am_i\equiv s_i}|\alpha_{m_1}{\alpha_{m_2}}|
$$
so that
$$
S^{\not=}\ll\frac{\log q}{AB}\sumsum_{r,s_1,s_2}\nu(r,s_1,s_2)
\ \Bigl|\sum_{b\sim B}K(s_1(r+b))\ov{K(s_2(r+b))}e(bt)\Bigr|
$$ 
(by the change of variable
$r=\ov a\ \cdot n,\ s_i=a\cdot m_i,\ i=1,2$).  We have
$$
\sumsum_{r,s_1,s_2}\nu(r,s_1,s_2)=\sumsum_{a,n,m_1\not=m_2}
|\alpha_{m_1}{\alpha_{m_2}}|\leq AN\|\alpha\|_1^2\leq
AMN\|\alpha\|_2^2
$$
and
$$
\sumsum_{r,s_1,s_2}\nu(r,s_1,s_2)^2=
\sumsum_{a,n,m_1\not=m_2}|\alpha_{m_1}||{\alpha_{m_2}}|
\sumsum_\stacksum{a',n',m'_1\not=m'_2}{\ov a'n'\equiv \ov a n,\
  a'm'_i\equiv am_i\mods q}|\alpha_{m'_1}||{\alpha_{m'_2}}|.
$$

Now assume that
\begin{equation}\label{AcondII1}
2AN<q.
\end{equation}
Then the equation $\ov a'n'\equiv \ov a n\mods q$ is equivalent to
$an'\equiv a' n\mods q$, which is equivalent to $an'=a' n$. Therefore
if we fix $a$ and $n'$, the integers $a'$ and $n$ are determined up to
$q^{o(1)}$ values.

Suppose that $a,a',n,n'$ are so chosen. For $i=1$, $2$, we then have
$$
\sumsum_\stacksum{m_i,m'_i}{am_i\equiv a'm'_i\mods
  q}|\alpha_{m_i}||{\alpha_{m'_i}}|\leq
\sumsum_\stacksum{m_i,m'_i}{am_i\equiv a'm'_i\mods
  q}|\alpha_{m_i}|^2+\sumsum_\stacksum{m_i,m'_i}{am_i\equiv
  a'm'_i\mods q}|{\alpha_{m'_i}}|^2\ll \|\alpha\|_2^2.
$$
Indeed, since $\mcM$ is a subset of $[1,q-1]$, once $m_i$
(resp. $m'_i$) is given, the congruence $am_i\equiv a'm'_i\mods q$
uniquely determines $m'_i$ (resp. $m_i$).  Therefore
\begin{equation}\label{nusecond}
\sumsum_{r,s_1,s_2}\nu(r,s_1,s_2)^2\ll q^{o(1)}AN\|\alpha\|_2^4.	
\end{equation}
Alternatively, if we assume instead of \eqref{AcondII1} that
\begin{equation}\label{AcondII2}
2AM^+<q,
\end{equation} 
then the same reasoning with the equation $am_1\equiv a'm'_1\mods q$
also leads to \eqref{nusecond}.

Fix an integer $l\geq 2$. We apply Hölder's inequality in the
following form:
\begin{align*}
  \sumsum_{r,s_1,s_2}\nu^{1-\frac1l+\frac1l}
  \Bigl|\sum_{b\sim B}\cdots\Bigr|
  &\leq \Bigl(\sumsum_{r,s_1,s_2}\nu\Bigr)^{1-\frac1{l}}
    \Bigl(\sumsum_{r,s_1,s_2}\nu\Bigl|\sum_{b\sim
    B}\cdots\Bigr|^{l}\Bigr)^{1/l}
  \\
  &\leq \Bigl(\sumsum_{r,s_1,s_2}\nu\Bigr)^{1-\frac1{l}}
    \Bigl(\sumsum_{r,s_1,s_2}\nu^2\Bigr)^{1/2l}
    \Bigl(\sumsum_{r,s_1,s_2}\Bigl|\sum_{b\sim B}\cdots\Bigr|^{2l}\Bigr)^{1/2l}\\
  &\leq q^\eps\|\alpha\|_2^{2}(AN)^{1-\frac 1{2l}}M^{1-\frac1{l}}
    \Bigl(\sum_{\bfb\in\mcB}|	
    \Sigma_{II}(K,\bfb)|\Bigr)^{1/2l},
\end{align*}
where $\mcB=[B,2B[^{2l}$, and 
$$  
\Sigma_{II}(K,\bfb)=\sum_{r\in\Fq}\sumsum_\stacksum{
  s_1,s_2\in\Fqt}{s_1\not =
  s_2}\bfK(s_1r,s_1\bfb)\ov{\bfK(s_2r,s_2\bfb)}
$$
is the exponential sum defined in~(\ref{eq-bfk}), where
$$
  \bfK(r,\bfb)=\prod_{i=1}^lK(r+b_i)
  \ov{K(r+b_{i+l})}.
$$
We observe at this point that the sum $\Sigma_{II}(K,\bfb)$ is
independent of the parameter $a$ such that
$K(x)=\Kl_k(ax;\uple{\chi},q)$, by changing the variables $s_1$ and
$s_2$ to $as_1$ and $as_2$ respectively.

We will estimate these sums in different ways depending on the
position of $\bfb$. Precisely:

\begin{theorem}\label{th-complete-1}
 There exist affine varieties
$$
\mcV^\Delta\subset \mcW\subset \Aa_{\Zz}^{2l}
$$
defined over $\Zz$ such that 
$$
\codim(\mcV^{\Delta})=l,\quad\quad \codim(\mcW)\geq \frac{l-1}{2}
$$
which have the following property: for any prime $q$ large enough,
depending only on $k$, for any tuple $\uple{\chi}$ of characters of
$\Fqt$ with Property \CGMT, for any $a\in\Fqt$, and for all
$\bfb\in\Ff_q^{2l}$, with
$$
K(x)=\Kl_k(ax;\uple{\chi},q),
$$
we have
\begin{gather}
  \Sigma_{II}(K,\bfb)\ll
  q^3\text{ if } \bfb\in \mcV^{\Delta}(\Fq)
\label{eq-complete-11}\\
  \Sigma_{II}(K,\bfb)\ll
  q^2\text{ if } \bfb\in (\mcW-\mcV^{\Delta})(\Fq)
  \label{eq-complete-12}\\
  \Sigma_{II}(K,\bfb)\ll q^{3/2}\text{ if } \bfb\notin \mcW(\Fq).
  \label{eq-complete-13}
\end{gather}
In all cases, the implied constant depends only on $k$.
\end{theorem}

We emphasize that the varieties $\mcV^{\Delta}$ and $\mcW$ are
independent of the tuple of characters.  After a number of
preliminaries, the final proof of this theorem will be found in
Section~\ref{sec-end} (see page~\pageref{pg-th-c1}).

We will apply these estimates for the parameters $\bfb$ belonging to
the box $[B,2B)^{2l}$, and for this we use Lemma \ref{lm-sz}.


Let $\mcB^{\mcV}$ (resp. $\mcB^{\mcW}$) be the set of $\bfb\in\mcB$
such that $\bfb\in \mcV^{\Delta}(\Fq)$ (resp. $\bfb\in \mcW(\Fq)$).
Since the subvarieties $\mcV^\Delta$ and $\mcW$ are defined over
$\Zz$, it follows from Lemma~\ref{lm-sz} that
\begin{align}
  \sum_{\bfb}|\Sigma_{II}(K,\bfb)|
  &\ll q^3|\mcB^{\mcV}|+q^2|\mcB^{\mcW}|+q^{3/2}B^{2l}
    \nonumber\\
  &\ll 
    q^3B^{2l-\codim (\mcV^{\Delta})}+q^2B^{2l-\codim
    (\mcW)}+q^{3/2}B^{2l}.	
    \label{SIIbound}
\end{align}
\par
We have $\codim(\mcV^\Delta)=l$ and $\codim(\mcW)\geq (l-1)/2$ by
Theorem~\ref{th-complete-1}.
We choose $B$ so that the first and third terms in \eqref{SIIbound}
are equal, namely 
$$B=q^{3/2l}.$$
We also choose $A$ so that $AB=N$, ie.
$$A=N/B=Nq^{-\frac{3}{2 l}}.$$
Writing $\codim (\mcW)=\gamma l$, we deduce that
$$
|B(K,\uple{\alpha},\uple{\beta})|\leq\|\beta\|_2
(\|\alpha\|_2^2N+S^{\not=})^{1/2}
$$
where
$$
S^{\not=}\ll \frac{q^\eps}{AB}\|\alpha\|_2^2(AN)^{1-\frac
  1{2l}}M^{1-\frac1{l}}(q^2B^{(2-\gamma)l)}+q^{3/2}B^{2l})^{1/2l}.
$$
Hence
\begin{align}
  |B(K,\uple{\alpha},\uple{\beta})|
  &\ll
    q^\eps\|\alpha\|_2\|\beta\|_2(MN)^{1/2}\Bigl(\frac{1}{M}
    +\Bigl(\frac{q^{2}B^{-\gamma l}}{AM^2N}
    +\frac{q^{\frac{3}{2}}}{AM^2N}\Bigr)^{\frac1{2l}}\Bigr) ^{1/2}\nonumber\\
  \label{Bbound}
  &\ll
    q^\eps\|\alpha\|_2\|\beta\|_2(MN)^{1/2}\Bigl(\frac{1}{M}
    +\Bigl(\frac{q^{2-\frac{3}2\gamma+\frac{3}{2l}}}{(MN)^2}
    +\frac{q^{\frac{3}{2}+\frac{3}{2l}}}{(MN)^2}\Bigr)^{\frac1{2l}}\Bigr) ^{1/2}.	
\end{align}
This holds under the condition that
$$A=Nq^{-\frac3{2l}}\geq 1$$
and that either of \eqref{AcondII1} or \eqref{AcondII2} hold.

In particular, since $\gamma\geq 1/3$, the second term on the
right-hand side of \eqref{Bbound} is smaller than the third. This
implies Theorem \ref{thmtypeIIprecise}. Theorem \ref{1stthmII} follows
by choosing $l$ large enough depending on $\delta$.

\subsection{Bounding type I sums}

We turn now to Theorem \ref{thmtypeIprecise}, and consider the special
bilinear form
$$
B(K,\uple{\alpha},1_\mcN)= \sumsum_{m\in\mcM,\,n\in\mcN}\alpha_m K(mn).
$$

Given $l\geq 2$, a trivial bound is
$$
B(K,\uple{\alpha},1_\mcN)\leq
\|\alpha\|_1^{1-\frac1l}\|\alpha\|_2^{\frac1{2l}}M^{\frac1{2l}}N.
$$
Proceeding as before, we get
\begin{align*}
  B(K,\uple{\alpha},1_\mcN)
  &=\frac{1}{AB}
    \sumsum_{a \sim A,\ B \sim B}
  \sum_{m\in\mcM}\alpha_m\sum_{n+ab\in\mcN} K(m(n+ab))\\
  &\ll_{\eps}
\frac{q^\eps}{AB}\sumsum_{r\in\Fq,s\in \Fqt}\nu(r,s)\Bigl|\sum_{b\sim B}\eta_bK(s(r+b))\Bigr|
\end{align*}
with
$$
\nu(r,s)=\sumsumsum_\stacksum{a\sim A,\ m\in\mcM,\
  n\in\mcN}{am=s,\ \ov an\equiv r\mods q}|\alpha_m|
$$
and $|\eta_b|\leq 1$. We have
$$
\sum_{r,s}\nu(r,s)\ll AN\sum_{m\in\mcM}|\alpha_m|.
$$
We also have
$$
\sum_{r,s}\nu(r,s)^2=\multsum_\stacksum{a,m,n,a',m',n'}{
    am\equiv a'm',a'n\equiv an'\mods q} |\alpha_m||\alpha_{m'}|.$$
Assuming that
\begin{equation}\label{ANbound}
2AN<q\hbox{ or }2AM^+<q	
\end{equation}
we show by the same reasoning as above that
$$
\sum_{r,s}\nu(r,s)^2\ll \sum_{a,m}|\alpha_m|^2
\multsum_{\substack{n,a',m',n'\\am=a'm'\\a'n=an'\mods q}}1 \ll_\eps
q^\eps AN\sum_{m}|\alpha_m|^2,
$$

We next apply H\"older's inequality in the form
\begin{multline*}
  \sumsum_{r\in\Fq,s\in \Fqt} \nu(r,s) \Bigl|\sum_{B < b \leq
    2B}\eta_bK(s(r+b))\Bigr|
  \\
  \leq\Bigl(\sum_{r,s}\nu(r,s)\Bigr)^{1-\frac1l}
  \Bigl(\sum_{r,s}\nu(r,s)^2\Bigr)^{\frac1{2l}}
  \Bigl(\sum_{r,s}\Bigl|\sum_{B < b \leq 2B
  }\eta_bK(s(r+b))\Bigr|^{2l}\Bigr)^{\frac{1}{2l}}
  \\
  \ll_{\eps} q^\eps
  (AN)^{1-\frac{1}{2l}}\|\uple{\alpha}\|_1^{1-\frac{1}{l}}\|\uple{\alpha}\|_2^{\frac{1}{l}}
  \Bigl(\sum_{r,s}\Bigl|\sum_{B < b \leq 2B }\eta_b
  K(s(r+b))\Bigr|^{2l}\Bigr)^{\frac{1}{2l}}.
\end{multline*}

Expanding the $2l$-th power, we have
$$
  \sumsum_{r\in\Fq,s\in\Fqt}\Bigl|\sum_{B < b \leq 2B }\eta_b
  K(s(r+b))\Bigr|^{2l}
  \leq \sum_{\bfb\in\mcB}\bigl|\Sigma_I(K, \bfb)\bigr|
$$
with
\begin{equation}\label{eq-def-sigma}
  \Sigma_I(K,\bfb)=\sum_{r\in\Fq}\sum_{s\in\Fqt}\bfK(sr,s\bfb)=
  \sum_{r\in\Fq}\bfR(r,\bfb).
\end{equation}
Note that $\Sigma_I(K,\bfb)$ is independent of the choice of
$a\in\Fqt$ such that $K(x)=\Kl_k(ax;\uple{\chi},q)$.  We have reached
the bound
\begin{equation}\label{eq-intermediate}
  B(K,\uple{\alpha},1_\mathcal{N})
  \ll q^{\eps}\|\uple{\alpha}\|_1^{1-\frac1l}\|\uple{\alpha}\|_2^{\frac1l}M^{\frac1{2l}}N
  \Bigl( \frac{(MN)^{-1}}{AB^{2l}}
  \sum_{\bfb\in\mcB}\bigl|\Sigma_I(K, \bfb)\bigr|\Bigr)^{\frac1{2l}}.
\end{equation}

As before, we can prove different bounds on $\Sigma_I(K,\bfb)$
depending on the position of $\bfb$.

\begin{theorem}\label{th-complete-2}
  Let $\mcV^{\Delta}$ and $\mcW$ be the affine varieties on
  Theorem~\ref{th-complete-1}.  For any prime $q$ large enough,
  depending only on $k$, for any tuple $\uple{\chi}$ with Property
  \CGMT, for any $a\in\Fqt$ and for all $\bfb\in\Ff_q^{2l}$, with
$$
K(x)=\Kl_k(ax;\uple{\chi},q),
$$
we have
\begin{gather}
  \Sigma_{I}(K,\bfb)\ll q^2\text{ if } \bfb\in \mcV^{\Delta}(\Fq)
\label{eq-complete-21}\\
  \Sigma_{I}(K,\bfb)\ll
  q^{3/2}\text{ if } \bfb\in (\mcW-\mcV^{\Delta})(\Fq)
  \label{eq-complete-22}\\
  \Sigma_{I}(K,\bfb)\ll q\text{ if } \bfb\notin \mcW(\Fq).
\label{eq-complete-23}
\end{gather}
In all cases, the implied constant depends only on $k$.
\end{theorem}

This is also proved ultimately in Section~\ref{sec-end}
(page~\pageref{pg-th-c2}). 
\par
\medskip
\par
Taking this for granted, and using the same notation
$\codim(\mcW)=\gamma l$ as before, we have therefore
\begin{align*}
  \sum_{\bfb\in\mcB}\bigl|\Sigma_I(K, \bfb)\bigr|
  &\ll
    |\mcB^{\mcV}|q^2+|\mcB^{\mcW}|q^{3/2}+|\mcB|q
  \\
  &\ll B^lq^2+B^{(2-\gamma)l}q^{3/2}+B^{2l}q,
\end{align*}
by Lemma~\ref{lm-sz}.  Choosing
$$
B=q^{1/l}
$$
to equate the first and third terms above and 
$$
A=N/B=Nq^{-1/l}
$$ 
we obtain from~(\ref{eq-intermediate}) the estimate
\begin{align*}
  B(K,\uple{\alpha},1_\mcN)
  &\ll_{k,\eps} 
    \ll q^{\eps}\|\uple{\alpha}\|_1^{1-\frac1l}\|\uple{\alpha}\|_2^{\frac1l}M^{\frac1{2l}}N
 \Bigl( \frac{(MN)^{-1}}{AB^{2l}}
  \big(qB^{2l}+q^{1/2}B^{{(3-\gamma)l}}\big)\Bigr)^{\frac1{2l}}\\
  &\ll_{k,\eps}q^{\eps}\|\uple{\alpha}\|_1^{1-\frac1l}\|\uple{\alpha}\|_2^{\frac1l}M^{\frac1{2l}}N\Bigl(\frac {q^{1+\frac{1}{l}}}{MN^{2}}+\frac {q^{\frac32-\gamma+\frac1l}}{MN^{2}}\Bigr)^{1/2l},
\end{align*}
assuming that \eqref{ANbound} holds and that $A\geq 1$. Since
$\gamma\geq 1/2$ (by Theorem~\ref{th-complete-1}), the second term on
the right-hand side of the last inequality is smaller than the first.
Together with \eqref{ANbound}, this leads to Theorem
\ref{thmtypeIprecise}, and Theorem \ref{1stthmI} follows by letting
$l$ get large.

\section{Algebraic preliminaries}

We collect in this section some definitions and statements of
algebraic geometry that we will use later. Most are standard, but we
include some proofs for completeness and by lack of a convenient
reference. 

Let $C_{\Fq}$ be a smooth and geometrically connected curve with
smooth projective model $S$. The conductor of a constructible
$\ell$-adic sheaf $\sheaf{F}$ on $C$ is defined by
$$
\cond(\mcF)= g(S)+\rank(\mcF)+|\Sing(\mcF)|+\sum_{x\in
  \Sing(\mcF)}\swan_x(\mcF)+ \dim H^0_c(C_{\bFq},\sheaf{F}),
$$
where $g(S)$ is the genus of $S$, $\Sing(\mcF)$ is the set of
points of $S$ where the middle-extension of $\mcF$ is not lisse and
$\swan_x(\mcF)$ is the Swan conductor at $x$.
\par
Let $C_{\Fq}$ be a curve (not necessarily smooth or irreducible).  Let
$(C_i)_{i\in I}$ be the geometrically irreducible components of
$C_{\bFq}$ and $\pi_i\colon \widetilde{C}_i\to C_i$ their canonical
desingularization. We define the conductor of a constructible
$\ell$-adic sheaf $\sheaf{F}$ on $C_{\Fq}$ by
$$
\cond(\mcF)=\sum_{i\in I}\cond(\pi_i^*(\mcF|C_i))+\sum_{x\in C_{sing}}m_x(C),
$$
where $C_{sing}$ is the singular set of $C$ and $m_x(C)$ the
multiplicity of $x$ as a singularity of $C$.

If $C_{\Fq}$ is a curve, $f$ is a function on $C$ and $\mcF$ an
$\ell$-adic sheaf on $C$, then
\begin{equation}\label{eq-cond-twist}
\cond(\mcF\otimes\sheaf{L}_{f(x)})\ll \cond(\sheaf{L}_{f(x)})^2\cond(\mcF)^2,
\end{equation}
where the implied constant is absolute.

We will use the following version of Deligne's Riemann Hypothesis over
finite fields~\cite{WeilII}.

\begin{proposition}\label{pr-recall-rh}
  Let $\Fq$ be a finite field with $q$ elements and let $C$ be a curve
  over $\Fq$. Let $\mcF$ and $\mcG$ be constructible $\ell$-adic
  sheaves on $C$ which are mixed of weights $\leq 0$ and pointwise
  pure of weight $0$ on a dense open subset. Suppose that the
  restriction of $\mcF\otimes\mcG^{\vee}$ to any geometrically
  irreducible component of $C$ has no trivial summand.  We then have
$$
\sum_{x\in C(\Fq)} t_{\mcF}(x;\Fq)\overline{t_{\mcG}(x;\Fq)}\ll
\sqrt{q}
$$
where the implied constant depend only on the conductors of $\mcF$ and
of $\mcG$.
\end{proposition}

\begin{proof}
  If $C$ is smooth and geometrically connected, and $\mcF$ and $\mcG$
  are geometrically irreducible middle-extensions, this is deduced
  from Deligne's results in~\cite[Lemma 3.5]{counting}; the extension
  to general $\mcF$ and $\mcG$ satisfying our assumptions is
  immediate. For a general smooth curve, one need only apply the bound
  to each component separately.
  
  For a general curve, observe that the difference between the sum
  over $C$ and the sum over a desingularization of $C$ is the sum over
  the singular points of $t_{\mcF}(x;\Fq)\overline{t_{\mcG}(x;\Fq)}$
  minus the sum over points of the desingularization lying over
  singular points of
  $t_{\mcF}(x;\Fq)\overline{t_{\mcG}(x;\Fq)}$. Since the size of both
  those sets of points may be bounded in terms of the sum of the
  multiplicities of singular points, and the value of
  $t_{\mcF}(x;\Fq)\overline{t_{\mcG}(x;\Fq)}$ at those points may be
  bounded in terms of the conductors, this contribution is also
  bounded in terms of the conductors.
\end{proof}

We will also use a criterion for a sheaf to be lisse that might be
well-known but for which we do not know of a suitable reference.

\begin{lemma}\label{lm-critere-lisse}
  Let $\Spec(\Oc)$ be an open dense subset of the spectrum of the ring
  of integers in a number field and $U\to\Spec(\Oc)$ a reduced scheme
  of finite type. Let $\ell$ be a prime number
  invertible in $\Oc$. Let $r\geq 1$ be an integer and let $\mcF$ be
  a constructible $\ell$-adic sheaf on $U$. 
\par
Assume that:
\par
\emph{(1)} For any finite-field valued point $\Spec(k)\to\Spec(\Oc)$,
the sheaf $\mcF_k$ on $U_k$ is lisse of rank $r$.
\par
\emph{(2)} For any finite-field valued point $\Spec(k)\to\Spec(\Oc)$,
any generic point $\eta$ of $U_{k}$, and any
$s\in \Gamma(\Spec(\Oc^{et}_{\eta}),\mcF)$, if $s$ is non-zero at the
special point of the étale local ring $\Oc^{et}_{\eta}$, then it is
non-zero at the generic point.
\par
Then $\mcF$ is lisse on $U$.
\end{lemma}

\begin{proof}
  Let $x\in U_k\subset U$ and let $s$ be a non-zero section of $\mcF$
  over the étale local ring $\Oc_x^{et}$ at $x$. Since (the pullback
  of) $\mcF$ is lisse on $\Oc^{et}_{x,k}$ by Assumption (1), the
  generic point of $\Oc^{et}_{x,k}$ belongs to the support of
  $s$. Hence (the pullback of) $s$ is non-zero at the special point of
  $\Oc^{et}_{\eta}$, which maps to the generic point $\Oc^{et}_{x,k}$
  (for some generic point $\eta$ of $U_{k}$). By Assumption (2), we
  deduce that the generic point of $\Oc^{et}_{\eta}$ belongs to the
  support of (the pullback of) $s$. Since this generic point maps to
  the generic point of $\Oc^{et}_x$, this means that the support of
  $s$ contains the generic point of $\Oc^{et}_x$, hence because the support of $s$ is closed, it is the whole $\Spec(\Oc^{et}_x)$.
\par
Now let $(s_1,\ldots,s_r)$ be a basis of the stalk
$\mcF_x=\Gamma(\Oc^{et}_x,\mcF)$. These sections define a morphism
$$
\bQl^{\, r}\to \mcF_{\Oc^{et}_x}
$$
whose induced map on stalks is, by the above, injective. By Assumption (1) and the fact that the rank of the stalk of a constructible $\ell$-adic sheaf is a constructible function, the rank of the stalk of $\mcF$ at every point is $r$. Hence both
stalks have the same dimension, thus the
induced map on stalks is an isomorphism. This means that $\mcF$ is
locally constant at $x$, and we conclude that $\mcF$ is lisse.
\end{proof}

\section{Generalized Kloosterman sheaves}

In this section, we summarize the basic properties of the generalized
Kloosterman sheaves whose trace functions are the sums
$\Kl_k(x;\uple{\chi},q)$. These were defined by Katz
in~\cite[Th. 4.1.1]{GKM}, building on Deligne's work~\cite[Sommes
trig., Th. 7.8]{sga4h}. They are special cases of the hypergeometric
sheaves defined by Katz in~\cite[8.2.1]{ESDE}.

Throughout this section, we fix a prime number $p$, a prime number
$\ell\not=p$, and we consider a finite field $\Fq$ of characteristic
$p$ with $q$ elements and a non-trivial $\ell$-adic additive character
$\psi$ of $\Fq$. We fix an integer $k\geq 2$ coprime to $q$, and a
tuple $\uple{\chi}=(\chi_1,\ldots,\chi_k)$ of $\ell$-adic characters
of $\Fqt$. We denote by $\Lambda(\uple{\chi})$ (or $\Lambda$ if
$\uple{\chi}$ is understood) the product $\chi_1\cdots\chi_k$.

\begin{proposition}[Generalized Kloosterman sheaves]\label{pr-kl}
  There exists a constructible $\bQl$-sheaf
  $\KL=\HYPK_{k,\psi}(\uple{\chi})$ on $\Pp^1_{\Fq}$, called a
  \emph{generalized Kloosterman sheaf}, with the following properties:
\begin{enumerate}
\item For any $d\geq 1$ and any $x\in \Gm(\Fqd)$, we have
\begin{multline*}
  t_{\KL}(x;\Fqd)=\Kl_k(x;\uple{\chi},\Fqd)\\
  = \frac{(-1)^{k-1}}{q^{d(k-1)/2}}\sum_{x_1\cdots
    x_k=x}\chi_1(N_{\Ff_{q^d}/\Fq}x_1)\cdots
  \chi_k(N_{\Ff_{q^d}/\Fq}x_k)
  \psi\Bigl(\Tr_{\Ff_{q^d}/\Fq}(x_1+\cdots+x_k)\Bigr).
\end{multline*}
\item The sheaf $\HYPK_{k,\psi}(\uple{\chi})$ is lisse of rank $k$ on
  $\Gm$.
\item On $\Gm$, the sheaf $\HYPK_{k,\psi}(\uple{\chi})$ is
  geometrically irreducible and pure of weight $0$.
\item The sheaf $\HYPK_{k,\psi}(\uple{\chi})$ is tamely ramified at
  $0$, and its $I(0)$-decomposition is
$$
\bigoplus_{\chi\in \uple{\chi}}\mcL_{\chi}\otimes J(n_{\chi}),
$$
where $J(n)$ is a unipotent Jordan block of size $n$, and $n_{\chi}$
is the multiplicity of $\chi$ in $\uple{\chi}$.
\item The sheaf $\HYPK_{k,\psi}(\uple{\chi})$ is wildly ramified at
  $\infty$, with a single break equal to $1/k$, and with Swan
  conductor equal to $1$.
\item The stalks of $\HYPK_{k,\psi}(\uple{\chi})$ at $0$ and $\infty$
  both vanish.
\item If $\gamma\in \PGL_2(\bFq)$ is non-trivial, there does not exist
  a rank $1$ sheaf $\sheaf{L}$ such that we have a geometric
  isomorphism 
$$
\gamma^*\HYPK_{k,\psi}(\uple{\chi})\simeq
\HYPK_{k,\psi}(\uple{\chi})\otimes \sheaf{L}
$$
over a dense open set.
\item The conductor of $\HYPK_{k,\psi}(\uple{\chi})$ is $k+3$.
\end{enumerate}
\end{proposition}

\begin{proof}
  Let $j\colon \Gg_m\to \Pp^1$ be the open inclusion. We define
$$
\HYPK_{k,\psi}(\uple{\chi})=j_!\mathrm{Kl}(\psi;\uple{\chi}; 1,\ldots,
1)\Bigl(\frac{n-1}{2}\Bigr),
$$
where the sheaf on the right-hand side is the lisse sheaf on $\Gg_m$
defined by Katz in~\cite[4.1.1]{GKM}. We also have a formula in terms
of hypergeometric sheaves, namely
$$
\HYPK_{k,\psi}(\uple{\chi})=
j_!\mathcal{H}_1(!,\psi;\uple{\chi},\emptyset)\Bigl(\frac{n-1}{2}\Bigr)
$$
(see~\cite[8.4.3]{ESDE}).  Assertions (1) and (2) are, respectively,
assertions (2) and (1) of~\cite[4.1.1]{GKM}. Assertion (3) results
from the identification with hypergeometric sheaves
and~\cite[Th. 8.4.2 (1), (4)]{ESDE}.
\par
Assertions (4) and (5)) are given in~\cite[Th. 8.4.2
(6)]{ESDE}. Assertion (6) is clear from the definition as an extension
by zero of a sheaf on $\Gg_m$.
\par
Finally, (7) is a special case of~\cite[Prop. 3.6 (2)]{FKMSP}, and (8)
follows from the definition of the conductor and the previous
statements.
\end{proof}

All parts of Definition~\ref{def-cgm}, including the definition of
Property \CGM\ and Property \CGMT, make sense for tuples of
$\ell$-adic characters of $\Fqt$. When we wish to emphasize the base
finite field, we will speak of Property \CGM\ or \CGMT\ over
$\Fq$. The names \CGM\ and \CGMT\ are justified by the following
theorem of Katz.

\begin{theorem}[Katz]\label{thm-geometric-monodromy}
  Assume that $k\geq 2$, that $p>2k+1$ and that $\uple{\chi}$ is not
  Kummer induced.  Let $G$ be the geometric monodromy group of
  $\HYPK_{k,\psi}(\uple{\chi})$. We then have $G^0=G^{0,der}$, the
  derived group. Moreover
\begin{enumerate}
\item If $k$ is odd, then $G^0=G^{0,der}=\SL_k$.
\item If $k$ is even, then $G^0=G^{0,der}$
  is either
\begin{itemize}
\item  $\SO_k$ if $\uple{\chi}$ is self-dual and symmetric.
\item $\Sp_k$ if $\uple{\chi}$ is self-dual and alternating.
\item $\SL_k$ if $\uple{\chi}$ is not self-dual.
\end{itemize}
\end{enumerate}
Finally, if $\uple{\chi}$ has \CGM, then $G=G^0$ is either $\SL_k$ or
$\Sp_k$.
\end{theorem}

\begin{proof}
  The claims about $G^0$ are proved by Katz in~\cite[Th. 8.11.3 and
  Corollary 8.11.2.1]{ESDE}.

  To evaluate $G$, note that when $G^0= \SL_k$, $G$ is contained in
  $\GL_k$. To show $G=G^0$, it suffices to show the determinant is
  trivial. But the determinant character is $\mathcal L_\Lambda$
  by~\cite[Lemma 8.11.6]{ESDE}, and we have assumed $\Lambda$ trivial.

  If $G^0 \neq \SL_k$ then $k$ is even and $\uple{\chi}$ is
  self-dual. Let $\xi$ be the dualizing character
  (Definition~\ref{def-cgm}). Under the assumptions $\Lambda=1$ and
  $\xi=1$, we always have $\Lambda =\xi^{k/2}$, so the self-duality is
  alternating. Thus $G^0=\Sp_k$, hence $G$ is contained in $\GSp_k$,
  and it suffices to show that the similitude character is trivial,
  i.e., that $\HYPK_{k,\psi}(\uple{\chi})$ is actually self-dual and
  not just self-dual up to a twist. This follows from~\cite[Theorem
  8.8.1]{ESDE}. Reviewing Definition~\ref{def-cgm}, we obtain the
  desired statements.
\end{proof}

The need to sometimes increase the base field is justified by the
following lemma that will allow us to work with tuples satisfying the
weaker \CGM\ Property.

\begin{lemma}\label{how-to-twist}
  Assume that $\uple{\chi}$ has \CGMT. Then there exists an
  $\ell$-adic character $\chi_0$, possibly over a finite extension
  $\Ff_{q^{\nu}}$ of $\Fq$, such that the tuple $\chi_0\uple{\chi}$
  has \CGM\ over $\Ff_{q^{\nu}}$.
\end{lemma}

\begin{proof} 
  If $k$ is even and $\uple{\chi}$ is self-dual alternating, take
  $\chi_0$ to be the inverse of a square root of the duality
  character. Otherwise, take $\chi_0$ to be the inverse of a $k$-th
  root of $\Lambda$.
\end{proof}




For convenience, we will most often simply denote
$\HYPK_k=\HYPK_{k,\psi}(\uple{\chi})$ since we assume that $\psi$ and
$\uple{\chi}$ are fixed.
\par
The next lemma computes precisely the local monodromy of
$\HYPK_{k,\psi}(\uple{\chi})$ at $\infty$.

\begin{lemma}\label{lm-kl-infty}
  Assume $p>k\geq 2$. Denote by $\widetilde{\psi}$ the additive character
  $x\mapsto \psi(kx)$ of $\Fq$. Then, as representations of the
  inertia group $I(\infty)$ at $\infty$, there exists an isomorphism
$$
\HYPK_{k,\psi}(\uple{\chi})\simeq [x\mapsto
x^k]_*(\mcL_{\chi_{(2)}^{k+1}} \otimes \mcL_\Lambda\otimes\mcL_{\widetilde{\psi}}),
$$
where $\chi_{(2)}$ is the unique non-trivial character of order $2$ of
$\Fqt$.
\end{lemma}

\begin{proof}
  This follows from a more precise result of
  L. Fu~\cite[Prop. 0.8]{Fu} (who describes the local representations
  of the decomposition group).
\end{proof}


\section{Sheaves and statement of the target theorem}

As in the previous section, we fix a prime number $p$, a prime number
$\ell\not=p$, and we consider a finite field $\Fq$ of characteristic
$p$ with $q$ elements and a non-trivial $\ell$-adic additive character
$\psi$ of $\Fq$. We assume that $p>2k+1$.

Let $\uple{\chi}$ be a $k$-tuple of $\ell$-adic characters of
$\Fqt$. We define
$$
\mcF=\HYPK_{k,\psi}(\uple{\chi}),
$$
a constructible $\ell$-adic sheaf on $\Aa^1_{\Fq}$. \emph{In this
  section we impose no further conditions on $\uple{\chi}$}.

Fix $l\geq 2$. For $1\leq i\leq 2l$, let $f_i=s(r+b_i)$ on
$\Aa^{2+2l}$ with coordinates $(r,s,\uple{b})$.

We now define the ``sum-product'' sheaf
$$
\mcK(\uple{\chi})=\bigotimes_{1\leq i\leq l} f_i^*\mcF\otimes
f_{i+l}^*\mcF^{\vee}
$$
on $\Aa^{2+2l}_{\Fq}$.

Let $V/\Zz$ be the open subset of $\Aa^{2+2l}_{\Zz}$ where
$s(r+b_i) \neq 0$ for all $i$, so that $\mcK(\uple{\chi})$ is lisse on
$V_{\Fq}$ for all $q$. Let $\pi\colon \Aa^{2+2l}\to \Aa^{1+2l}$ be the
projection $(r,s,\uple{b})\mapsto (r,\uple{b})$ (defined over $\Zz$).
We define
$$
\mcR(\uple{\chi})=R^1\pi_!\mcK(\uple{\chi}),
$$ 
a constructible $\ell$-adic sheaf on $\Aa^{1+2l}_{\Fq}$. 

We will most often drop the dependency on $\uple{\chi}$ in these
notation and write $\mcK=\mcK(\uple{\chi})$ and
$\mcR=\mcR(\uple{\chi})$.

We define the diagonal variety $\mcV^{\Delta}$ by the condition
$$
\mcV^{\Delta}=\{\uple{b}\in\Aa^{2l}\,\mid\, \text{ for all $i$, there
  exists $j\not=i$ such that $b_i=b_j$}\}.
$$
Note that $\mcV^{\Delta}$ does not depend on the tuple of characters
considered.

\begin{lemma}\label{lm-top-vanishing}
  Outside $\mcV^{\Delta}$, we have $R^0\pi_!\mcK=R^2\pi_!\mcK=0$.
\end{lemma}

\begin{proof}
  This is very similar to~\cite[Lemma 4.1 (2)]{KMS}. By the proper
  base change theorem, the stalk of $R^i\pi_!\mcK$ at
  $x=(r,\bfb)\in \Aa^{1+2l}$ is
$$
H^i_c(\Aa^1_{\bFq},\bigotimes_{i=1}^{l}[\times (r+b_i)]^*\mcF\otimes
[\times (r+b_{i+2})]^*\mcF^{\vee}),
$$
where $s$ is the coordinate on $\Aa^1$. This cohomology group vanishes
for $i=0$ and any $x$, and it vanishes for $i=2$ and
$x\notin \mcV^{\Delta}$ by~\cite[Theorem 1.5]{FKMSP}.
\end{proof}

We now compute the local monodromy at infinity of the sheaf $\mcK$.
For any additive character $\psi$, we denote by $\widetilde{\psi}$ the
character $x\mapsto \psi(kx)$.

\begin{lemma}\label{lm-tilde-psi}
\emph{(1)} Let $r\in\Fq$ and $\bfb\in\Ff_q^{2l}$ be such that
$r+b_i\not=0$ for all $i$.  Let $(r+b_i)^{1/k}$ be a fixed $k$-th root
of $r+b_i$ in $\bFq$. Define signs $\eps_i=1$ for $1\leq i\leq l$ and
$\eps_i=-1$ for $l+1\leq i\leq 2l$.
\par
The local monodromy at $s=\infty$ of $ \mcK_{r,\bfb}$ is isomorphic to
the local monodromy at $s=\infty$ of the sheaf
$$
 \bigoplus_{(\zeta_2,\ldots,\zeta_{2l}) \in
    \mmu_k^{2l-1}} \mcL_{\widetilde{\psi}}\Bigl( \Bigl( 
    (r+b_1)^{1/k}
    +\sum_{i=2}^{2l}\eps_i \zeta_i (r+b_i)^{1/k}
\Bigr) s^{1/k}\Bigr).
$$
where $\mmu_k$ is the group of $k$-th roots of unity in $\bFq$.
\par
\emph{(2)} Let $K$ be a field of characteristic $p\nmid k$, and let
$r\in K$ and $\uple{b}\in K^{2l}$ be such that $r+b_i\not=0$ for all
$i$.  Assume that $K$ contains all $k$-th roots $(1+b_i/r)^{1/k}$ of
$1+b_i/r$ for all $i$. Let $\psi$ be a non-trivial $\ell$-adic
additive character and let $\uple{\chi}$ be a $k$-tuple of
multiplicative characters of a finite subfield of $K$. The local
monodromy at $t=\infty$ of the lisse sheaf
$$
\widetilde{\mcK}=\bigotimes_{1\leq i\leq l}
\HYPK_{k,\psi}(\uple{\chi})(t(1+b_i/r))\otimes
\HYPK_{k,\psi}(\uple{\chi})(t(1+b_{i+l}/r))^{\vee}
$$
on $\Gg_{m,K}$ is isomorphic to the local monodromy at $t=\infty$ of
the sheaf
$$
\bigoplus_{(\zeta_2,\ldots,\zeta_{2l}) \in \mmu_k^{2l-1}}
\mcL_{\widetilde{\psi}}\Bigl( \Bigl( (t(1+b_1/r))^{1/k}
+\sum_{i=2}^{2l}\eps_i \zeta_i(t(1+b_i/r))^{1/k} \Bigr)\Bigr).
$$
\end{lemma}

\begin{proof}
  Since Lemma~\ref{lm-kl-infty} has the same form as~\cite[Lemma
  4.9]{KMS}, up to the additional factor $\mcL_{\Lambda}$, the first
  assertion may be proved exactly like~\cite[Lemma 4.16 (1)]{KMS}
  (with $\lambda=0$ there), replacing throughout the tensor product
$$
\bigotimes_{i=1}^2 [\times (r+b_i)]^*\HYPK_{k}
\otimes [\times (r+b_{i+2})]^*\HYPK_{k}^{\vee}
$$
by
$$
\bigotimes_{i=1}^{l} [\times (r+b_i)]^*\HYPK_{k,\psi}(\uple{\chi})
\otimes [\times (r+b_{i+2})]^*\HYPK_{k,\psi}(\uple{\chi})^{\vee}
$$
(note that the factors involving $\Lambda$ cancel-out at the end).
The second statement is proved in the same manner.
\end{proof}

Let $\widetilde{Z}\subset \Aa^{1+2l}_{\Zz}$ be the image of 
\begin{equation}\label{eq-new-z}
  \widetilde{\mathcal{Z}}=\Bigl\{(r,\uple{b},\uple{x})\in \Aa^{1+4l}\,\mid\,
  x_i^k=r+b_i
  \text{ for
  } 1\leq i\leq 2k,\quad \sum_{i=1}^lx_i=\sum_{i=l+1}^{2l}x_i \Bigr\}\subset
  \Aa^{1+4l}_{\Zz}
\end{equation}
under the projection onto $(r,\uple{b})$.  Let 
$$
Z=\widetilde{Z}\cup\bigcup_{1\leq i\leq 2l}\{ r=-b_i\}.
$$
Let $U$ be the complement of $Z$. We emphasize that $\widetilde{Z}$,
$Z$ and $U$ are defined over $\Zz$, and independent of $\uple{\chi}$.

\begin{lemma}\label{lm-lisse}
  The subscheme $\widetilde{Z}$ of $\Aa^{2l+1}$ is closed and
  irreducible, and $\mcR$ is lisse on $U_{\Fq}$.
\end{lemma}

\begin{proof}
  This is analogue to~\cite[Lemma 4.26, (1) and (2)]{KMS}, so we will
  be brief.\footnote{To avoid confusion, note that what is called $Z$
    in~\cite{KMS} is not the analogue of what is called $Z$ here.}
  The projection $(r,\uple{b},\uple{x}) \mapsto (r,\uple{b})$ from the
  subscheme
$$
\mathcal{Z}'=\Bigl\{(r,\uple{b},\uple{x})\in \Aa^{1+4l}\,\mid\,
x_i^k=r+b_i\text{ for } 1\leq i\leq 2k \Bigr\}
$$
to $\Aa^{1+2l}$ is finite, since the domain is defined by adjoining
the coordinates $(x_1,\ldots,x_{2l})$ to $\Aa^{1+2l}$, and each
satisfies a monic polynomial equation. Thus the closed subscheme
$\widetilde{\mathcal{Z}}$ defined by~(\ref{eq-new-z}) is also finite over
$\Aa^{1+2l}$, and its image $\widetilde{Z}$ is closed.  Moreover, the
subscheme~(\ref{eq-new-z}) is the divisor in $\mathcal{Z}'$ given by
the equation
$$
\sum_{i=1}^lx_i=\sum_{i=l+1}^{2l}x_i.
$$ 
In particular, this subscheme, and consequently its projection
$\widetilde{Z}$, is irreducible.

To prove that $\mcR$ is lisse on $U_{\Fq}$, we use Deligne's
semicontinuity theorem~\cite{LaumonSMF}.  The sheaf $\mcK$ is lisse on
the complement of the divisors given by the equations $r=-b_i$ and
$s=0$ in $\Aa^{2+2l}$.  We compactify the $s$-coordinate by $\Pp^1$
and work on
 $$
 X=(\Aa^1\times\Pp^1\times \Aa^{2l})\cap \{(r,s,\bfb)\,\mid\,
 (r,\bfb)\in U\}.
$$
\par
By extending by $0$, we view $\mcK$ as a sheaf on $X$ which is lisse
on the complement in $X$ of the divisors $s=0$ and $s=\infty$ (because
$U$ is contained in the complement of the divisors $r=-b_i$ and thus
$X$ is as well). Let
$$
\pi^{(2)} \,:\, X\lra U
$$
denote the projection $(r,s,\bfb)\mapsto (r,\bfb)$. Then $\pi^{(2)}$
is proper and smooth of relative dimension $1$ and
$\mcR|U=R^1\pi^{(2)}_*\mcK$.
\par
Since the restrictions of $\mcK$ to the divisors $s=\infty$ and $s=0$
are zero, this sheaf is the extension by zero from the complement of
those divisors to the whole space of a lisse sheaf. Deligne's
semicontinuity theorem \cite[Corollary 2.1.2]{LaumonSMF} implies that
the sheaf $\mcR$ is lisse on $U$ if the Swan conductor is constant on
each of these two divisors.  By Proposition~\ref{pr-kl}, the
generalized Kloosterman sheaf has tame ramification on $s=0$, hence
any tensor product of generalized Kloosterman sheaves (such as $\mcK$)
has tame ramification, hence Swan conductor $0$, on $s=0$.  On the
other hand, Lemma \ref{lm-tilde-psi} gives a formula for the local
monodromy representation of $\mcK$ at $s=\infty$ as a sum of
pushforward of representations from the tame covering $x\mapsto x^k$.
Since the Swan conductor is additive and since the Swan conductor is
invariant under pushforward by a tame covering (see,
e.g.,~\cite[1.13.2]{GKM}), it follows that
$$
  \swan_{\infty}(\mcK_{r,\bfb})=
  \sum_{\zeta_2,\ldots,\zeta_{2l}\in\mmu_k}
  \swan_{\infty}\Bigl(\mcL_\psi \Bigl( \Bigl( (r+b_1)^{1/k}
    +\sum_{i=2}^{2l}\eps_i \zeta_i (r+b_i)^{1/k} \Bigr)
    s^{1/k}\Bigr)\Bigr)=k^{2l-1}
$$
by definition of $U$, since the Swan conductor of $\mcL_{\psi(a t)}$
is $1$ for $a\not=0$.
\end{proof}

\begin{lemma}\label{z-polynomial} The subscheme $Z$ is a hypersurface
  in $\Aa^{1+2l}_{\Zz}$. It is defined by the vanishing of a
  polynomial $P$ in $\Zz[r,b_1,\dots,b_{2l}]$ such that, for any fixed
  $\uple{b}\not\in \mcV^{\Delta}$, the polynomial
  $P_{\uple{b}}=P(\cdot,\uple{b})$ of the variable $r$ is not zero.
\end{lemma}

\begin{proof} 
  First we check that $\widetilde{Z}$ is a hypersurface in
  $\Aa^{1+2l}_{\Zz}$.  It is the projection of the closed subscheme
  $$
  \widetilde{\mathcal{Z}}=\Bigl\{(r,\uple{b},\uple{x})\in
  \Aa^{1+4l}\,\mid\, x_i^k=r+b_i\text{ for } 1\leq i\leq 2l,\
  \sum_{i=1}^lx_i=\sum_{i=l+1}^{2l}x_i \Bigr\}\subset \Aa^{1+4l}.
$$
This closed subscheme is pure of dimension $2l$, since the first $2l$
equations let us eliminate the variables $b_i$ and the last equation
is nontrivial. The projection $\widetilde{\mathcal{Z}}\to \widetilde{Z}$ is
finite (as already observed in the proof of the previous lemma)
and hence $\widetilde{Z}$ is a closed subscheme of $\Aa^{2l+1}$ that is
pure of dimension $2l$, i.e. a hypersurface.  Since $Z$ is the union
of $\widetilde{Z}$ and the hyperplanes with equation $r+b_i=0$, it is also
a hypersurface.

Let $P\in\Zz[r,\uple{b}]$ be a polynomial whose vanishing set is
$\widetilde{Z}$.  Suppose $\uple{b}$ is such that $P_{\uple{b}}$ is the
zero polynomial in the variable $r$, i.e., such that the projection
$\widetilde{\mathcal{Z}}_{\uple{b}}\to \Aa^1$ given by
$(r,\uple{x})\mapsto r$ is surjective.

The scheme $C\subset \Aa^{1+2l}$ given by  the equations
$$
x_i^k=r+b_i\quad\quad 1\leq i\leq 2k
$$
is a curve and the projection $C\to \Aa^1$ given by
$(r,\uple{x})\mapsto r$ is finite. The fiber
$\widetilde{\mathcal{Z}}_\uple{b}$ is the intersection of $C$ and the
hyperplane
$$
\sum_{i=1}^lx_i=\sum_{i=l+1}^{2l}x_i,
$$
so that $P_{\uple{b}}=0$ if and only if the function
$$
F=\sum_{i=1}^lx_i-\sum_{i=l+1}^{2l}x_i
$$
vanishes on an irreducible component of $C$.  
\par
If we assume that $\uple{b}\notin \mcV^{\Delta}$ then by definition
there exists some $i$ such that $b_i \neq b_j$ for all $j \neq
i$. Locally on $\Aa^1$ with coordinate $r$ near the point $r=-b_i$,
the covering maps $x^j_k = r+b_j$ for $j\not=i$ are étale, so the
functions $x_j$ (on the curve $C$) ``belong'' to the étale local ring
$R$ of $\Aa^1$ at $-b_i$. The function $x_i$, however, does not belong
to $R$, hence the function $F$ is non-zero in an algebraic closure of
the fraction field of $R$, which is also an algebraic closure of the
function field of any irreducible component of $C$. This concludes the
proof.
\end{proof}

\begin{definition}
  The sheaf $\mcR^*$ on $U_{\Fq}$ is the maximal quotient of the sheaf
  $\mcR|U_{\Fq}$ that is pure of weight $1$ (see~\cite{WeilII}).
\end{definition}

Define $f\colon U\to \Aa^{2l}$ over $\Zz$ by
$(r,\uple{b})\mapsto \uple{b}$.

Below, by $\End_{V_{\uple{b}}}(\mcG)$, where $\mcG$ is a lisse sheaf
on $V_{\Fq, \uple{b}}$, we mean the
$\pi_1(V_{\Fq,\uple{b}}\times\bFq)$-homomorphisms, etc.

Let $\uple{b}\in \Aa^{2l}_{\Fq}$ and let $\kappa(\uple{b})$ be the
residue field of $\uple{b}$. Since
$\mcR_{\uple{b}}=R^1\pi_!\mcK_{\uple{b}}$ by the proper base change
theorem, there exists a natural
$\Gal(\overline{\kappa(\uple{b})}/\kappa(\uple{b}))$-equivariant
morphism
$$
\End_{V_{\uple{b}}}(\mcK_{\uple{b}})\fleche{} 
\End_{U_{\uple{b}}}(\mcR_{\uple{b}}).
$$
Since every $V_{\uple{b}}$-endomorphism of $\mcK_{\uple{b}}$ preserves
the weight filtration, the image of this morphism is contained in the
subring of endomorphisms of $\mcR_{\uple{b}}$ that preserve the weight
filtration, and hence we have an induced morphism
$$
\theta_{\uple{b}}\colon
\End_{V_{\uple{b}}}(\mcK_{\uple{b}})\fleche{} 
\End_{U_{\uple{b}}}(\mcR^*_{\uple{b}}),
$$
which by construction is still Frobenius-equivariant.

In the next definition, we already describe the subvariety $\mcW$ of
Theorem~\ref{th-complete-1}; in particular, we see that it is
independent of the tuple of characters $\uple{\chi}$, since this is
the case for $X_{\infty}$ and $Z$. The difficulty will be to prove
that it satisfies the required properties.

\begin{definition}\label{def-xj}
  We denote $X_{\infty}=\Aa^{2l}-\mcV^{\Delta}$, and for any integer
  $j\geq 0$, we let
 $$
 X_j=\{\uple{b}\in X_{\infty}\,\mid\, |Z_{\uple{b}}|\leq j\}.
$$
\par
We define $\mcW$ to be the union of $\mcV^{\Delta}$ and of all
irreducible components of all $X_j$ of dimension strictly less than
$(3l+1)/2$.
\end{definition}

By definition, we therefore have the codimension bound
\begin{equation}\label{eq-codim-w}
  \codim(\mcW)\geq \frac{l-1}{2}.
\end{equation}

Our main geometric goal will be to prove the following result:

\begin{theorem}\label{th-thetab} 
  Assume that $\uple{\chi}$ has \CGMT.  If $p$ is large enough,
  depending only on $k$ and $l$, then the natural morphism
  $\theta_{\uple{b}}$ is an isomorphism for all
  $\uple{b}\in \Aa^{2l}(\Fq)-\mcW(\Fq)$.  Furthermore, each
  geometrically irreducible component of $\mcR^*_{\uple{b}}$ has rank
  greater than one.
\end{theorem}

The basic strategy to be used is as follows:

\begin{enumerate}
\item We show that for $q$ large enough and for
  $\uple{b}\in\Aa^{2l}(\Fq)$ outside an explicit subscheme $\mcW_1$ of
  codimension $l-1$, the natural morphism $\theta_{\uple{b}}$ is
  injective. This reduces the target statement to a proof that the
  dimensions $\End_{V_{\uple{b}}}(\mcK_{\uple{b}})$ and
  $\End_{U_{\uple{b}}}(\mcR^*_{\uple{b}})$ are equal.
\item We show that, when these dimensions agree for the generic point
  of an irreducible component of a stratum, this implies the
  corresponding statement on the whole irreducible component.
\item Finally, we prove the target theorem at the generic point of an
  irreducible component of a stratum with dimension $>(3l+1)/2$.
\end{enumerate}

The most difficult part is the last one. This we prove by showing the
strata can be covered by the vanishing sets of equations of a certain
type in products of curves. Using this description, and a variant of
Katz's Diophantine criterion for irreducibility, we show that the
dimension of the space of endomorphisms of $\mathcal K$ is equal to
that of the space of endomorphisms of $\mcR$ that are invariant under
the Galois group of the function field of this cover. Finally, by a
vanishing cycles argument, we show that the Galois group in fact acts
trivially.

\begin{remark}
  We have defined $U$, the stratification $X_j$, and $\mcW$ as objects
  over the integers rather than over a finite field $\Fq$. This is
  used in a few different places: first, when comparing the generic
  point and the special point of a stratum, we use a tameness property
  of the sheaf $\mcR$, which we verify by showing that the sheaf is
  defined over the integers. Second, when describing the defining
  equations of the strata, at one point we make a large characteristic
  assumption. Third, we need the set $\mcW$ to be uniform in $q$ to
  allow us to apply Lemma \ref{lm-sz}.
\end{remark}


\section{Integrality}

We fix an integer $n\geq 1$ and an integer $k\geq 2$. Let $\ell$ be a
prime number.  We denote in this section
$S=\Spec(\Zz[\mmu_n,1/n\ell])$.  For any $\ell$-adic character
$\widetilde{\chi}$ of $\mmu_n$, we have an associated lisse
$\ell$-adic sheaf $\sheaf{L}_{\widetilde{\chi}}$ over $S$ defined by
Kummer theory. If $\Fq$ is a residue field of $S$ of characteristic
$p\nmid n\ell$, so that $q\equiv 1\bmod{n}$, then there is a natural
isomorphism between the group of $\ell$-adic characters
$\widetilde{\chi}$ of $\mmu_n$ and the group of $\ell$-adic characters
$\chi$ of order dividing $n$ of $\Fqt$, such that
$\chi(x)=\widetilde{\chi}(\xi)$, where $\xi$ is the $n$-th root of
unity in $\Zz[\mmu_n,1/n\ell]$ mapping to $x^{(q-1)/n}$. We then have
a natural isomorphism $\mcL_{\widetilde{\chi},\Fq}=\mcL_{\chi}$ of
$\ell$-adic sheaves.


\begin{proposition}\label{pr-integrality}
  Let $\widetilde{\uple{\chi}}$ be a $k$-tuple of characters of
  $\mmu_n$.  There exists an $\ell$-adic sheaf
  $\mcR^{univ}(\widetilde{\uple{\chi}})$ on $\Aa^{1+2l}_{S}$, lisse on
  $U_{S}$, with the following property: for any prime $p\nmid \ell n$,
  for any finite field $\Fq$ of characteristic $p$ which is a residue
  field of a prime ideal in $\Zz[\mmu_n,1/n\ell]$, for any non-trivial
  additive character $\psi$ of $\Fq$, we have
$$
\mcR^{univ}(\widetilde{\uple{\chi}}) | \Aa^{1+2l}_{\Fq} =
\mcR(\uple{\chi})
$$
where $\uple{\chi}$ is the $k$-tuple of $\ell$-adic characters of
$\Fqt$ corresponding to $\widetilde{\uple{\chi}}$.
\end{proposition}

\begin{proof}
  We will first construct a sheaf
  $\mcR^{univ}(\widetilde{\uple{\chi}})$ over $S$ with the desired
  specialization property, and we will then check that the sheaf thus
  defined is lisse on $U_{S}$.  The existence statement is a fairly
  straightforward generalization of~\cite[Lemma 4.27]{KMS}, but we
  give full details since the precise construction is needed to check
  the lisseness assertion.
\par
Let $X_1\subset \Gg_m^{k+1}$ be the subscheme over
$S$ with equation
$$
x_1\cdots x_k=t
$$
and let
$$
f_1\,:\, X_1\lra \Aa^1
$$ 
be the projection $(x_1,\ldots, x_k,t)\mapsto t$.  Let $X_2$ be the
subscheme of $\Gg_m^{2lk}\times\Aa^{2+2l}$ over $S$ defined by the
equations
$$
\prod_{j=1}^kx_{i,j} = s(r+b_i),\quad\quad 1\leq i\leq 2l,
$$
and let $f_2\,:\, X_2\lra \Aa^{1+2l}$ be the projection
$$
f_2(x_{1,1},\ldots, x_{2l,k},r,s,\bfb)= (r,\bfb).
$$
\par
Let further $X \subset X_2$ be the closed subscheme over $S$ defined
by the equation $x_{1,1}=1$. The morphism
$$
\Gg_m \times X \to X_2
$$ 
defined by
$$
(t, x_{1,1},\ldots,x_{2l,k},r,s,\bfb)\mapsto
(tx_{1,1},\dots,tx_{2l,k}, r ,t^k s, \bfb)
$$ 
is an isomorphism, with  inverse given by
$$
(x_{1,1},\ldots, x_{2l,k}, r,s,\bfb)\mapsto
\Bigl(x_{1,1},1,\frac{x_{1,2}}{x_{1,1}},\ldots,
\frac{x_{2l,k}}{x_{1,1}},r, \frac{s}{t^k},\bfb\Bigr).
$$ 
\par
Let now $p\nmid n\ell$ be a prime and $\Fq$ a finite field of
characteristic $p$ that is a residue field of a prime ideal in
$S$. Let $\psi$ be a non-trivial additive character of $\Fq$. We
have an isomorphism
$$
\HYPK_{k,\psi}(\uple{\chi})\Bigl(\frac{1-k}{2}\Bigr) [1-k]\simeq
Rf_{1,!}  \sheaf{L}_{\psi}(x_1+\cdots +x_k) \otimes
\bigotimes_{i=1}^k \mathcal L_{\chi_i} (x_i)
$$
of sheaves on $\Aa^1_{\Fq}$.  By definition and Lemma
\ref{lm-top-vanishing}, it follows that
$$
  \mcR(\uple{\chi})=R^{2l(k-1)+1}f_{2,!}\Bigl(\sheaf{L}_{\psi}
  \Bigl(\sum_{j=1}^k\Bigl(\sum_{i=1}^lx_{i,j}-\sum_{i=1}^lx_{l+i,j}\Bigr)\Bigr)
  \otimes \bigotimes_{j=1}^k \bigotimes_{i=1}^l \mcL_{\chi_j}(
  x_{i,j}/x_{l+i,j})\Bigr).
$$
\par
We now translate this by ``transport of structure'' to
$\Gg_m\times X\simeq X_2$. First, we have $f_2=f\circ p_2$ where $p_2$
is the projection $\Gg_m \times X \to X$.  Next, let
$f: X \to \Aa^{1+2l}$ be the projection onto $(r,\bfb)$, and let
$g: X \to \Aa^1$ be defined by
$$
g(x_{1,1},\ldots,x_{2l,k},r,s,\bfb)=\sum_{j=1}^k\Bigl(
\sum_{i=1}^lx_{i,j}-\sum_{i=1}^lx_{l+i,j}\Bigr).
$$
Let $g'$ be the function
$$
g'=\sum_{j=1}^k\Bigl(\sum_{i=1}^lx_{i,j}-\sum_{i=1}^lx_{l+i,j}\Bigr)
$$
on $X_2$. Then $g'$ corresponds to $tg$ under the isomorphism
$X_2\simeq \Gg_m\times X$. Moreover, the sheaves
$\mcL_{\chi_j} (x_{i,j}/ x_{l+i,j})$ are transported to
$\mcL_{\chi_j} (x_{i,j}/ x_{l+i,j})$ under this isomorphism (since
both variables involved are multiplied by $t$). We conclude that
$$
\mcR(\uple{\chi})[ -2l (k-1)-1] \simeq R(f\circ p_2)_! \Bigl(
\mcL_{\psi}(tg) \otimes \bigotimes_{j=1}^k \bigotimes_{i=1}^l
\mcL_{\chi_j}( x_{i,j}/x_{l+i,j}) \Bigr)
$$
on $\Aa^{1+2l}_{\Fq}$.
\par
We can now apply the strategy of~\cite[Lemma 4.23]{KMS}. By the
projection formula, we have
$$ 
R{p_2!} \Bigl( \mcL_{\psi}(tg) \otimes \bigotimes_{j=1}^k
\bigotimes_{i=1}^l \mcL_{\chi_j}( x_{i,j}/x_{l+i,j}) \Bigr) =\Bigl(
\bigotimes_{j=1}^k \bigotimes_{i=1}^l \mcL_{\chi_j}(
x_{i,j}/x_{l+i,j})\Bigr) \otimes Rp_{2!} \mcL_\psi(tg)
$$ 
and $Rp_{2!} \mcL_\psi(tg)$ is the pullback along $g$ of the Fourier
transform of the extension by zero of the constant sheaf on
$\Gg_{m,\Fq}$, which is $(Ru_* \bQl [-1])_{\Fq}$ for
$u:\Gg_m \to \Aa^1$ the inclusion.
\par
We then define the sheaf
$$
\mcR^{univ}(\widetilde{\uple{\chi}}) = R^{2l(k-1)}f_! \Bigl( g^* (
Ru_*\bQl) \otimes \bigotimes_{j=1}^k \bigotimes_{i=1}^l
\mcL_{\widetilde{\chi}_j}( x_{i,j}/x_{l+i,j} )\Bigr)
$$
over $S$. The preceeding computation gives an isomorphism
$\mcR^{univ}(\widetilde{\uple{\chi}})_{\Fq}\simeq \mcR(\uple{\chi})$
over $\Fq$.
\par
Furthermore, since the complex
$$ 
Rf_! \Bigl( g^* ( Ru_*\bQl) \otimes \bigotimes_{j=1}^k
\bigotimes_{i=1}^l \mcL_{\chi_j}( x_{i,j}/x_{l+i,j} )\Bigr),
$$ 
is supported in degree $2l(k-1)$ over $U_{\Fq}$ for all $\Fq$, the
corresponding complex
$$
Rf_! \Bigl( g^* ( Ru_*\bQl) \otimes \bigotimes_{j=1}^k
\bigotimes_{i=1}^l \mcL_{\widetilde{\chi}_j}( x_{i,j}/x_{l+i,j}
)\Bigr)
$$
is supported in a single degree on $S$.

We will now check that $\mcR^{univ}(\widetilde{\uple{\chi}})$ is lisse
on $U_{S}$.  By the specialization property and Lemma~\ref{lm-lisse},
we know that $\mcR^{univ}(\widetilde{\uple{\chi}})$ is lisse on
$U_{\Fq}$ for any residue field $\Fq$ of characteristic
$p\nmid\ell n$, and that it has constant rank. Because it is a
constructible sheaf, its rank is a constructible function, and hence
it has the same rank everywhere on $U_{S}$.

Write $\mcR^{univ}=\mcR^{univ}(\widetilde{\uple{\chi}})$ for
simplicity. We show that $\mcR^{univ}$ is lisse on $U_{S}$ by
contradiction. By the criterion in Lemma~\ref{lm-critere-lisse}, if
$\mcR^{univ}$ is \emph{not} lisse on $U_{S}$, then there exists a
finite-field-valued point (say over $\Fq$) and a section of
$\mcR^{univ}$ over the étale local ring $\Oc^{et}_{\eta}$ for some
generic point $\eta$ of $U_{\Fq}$ which is non-zero at the special
point, but zero at the generic point.
If we denote by 
$i$ the inclusion of $\eta$ in $\Spec(\Oc^{et}_{\eta})$, then such a
section corresponds to a 
morphism $i_* \bQl \to \mcR^{univ}$ over this local ring that is
non-trivial at the generic point. Because
$$
\mcR^{univ}= R^{2l(k-1) }f_! \Bigl( g^* ( Ru_*\bQl) \otimes
\bigotimes_{j=1}^k \bigotimes_{i=1}^l \mcL_{\widetilde{\chi}_j}(
x_{i,j}/x_{l+i,j} ) \Bigr)
$$ 
and the complex
$$
Rf_!  \Bigl( g^* ( Ru_*\bQl) \otimes \bigotimes_{j=1}^k
\bigotimes_{i=1}^l \mcL_{\widetilde{\chi}_j}( x_{i,j}/x_{l+i,j} )
\Bigr)
$$ 
is supported in a single degree, we obtain a nontrivial map.
\begin{equation}\label{eq-morph}
  Ri_* \bQl [ -2l(k-1)] \to Rf_! \Bigl( g^* ( Ru_*\bQl) \otimes
  \bigotimes_{j=1}^k \bigotimes_{i=1}^l \mcL_{\widetilde{\chi}_j}( x_{i,j}/x_{l+i,j}
  ) \Bigr).
\end{equation}

We then apply the Verdier duality functor, taking our base scheme
$S = \Spec(\Oc^{et}_{\eta})$. In this case our dualizing complex is
$\bQl$ and we set $\dual(\mcF) =\Hom(\mcF, \bQl)$. Later, we will
apply also apply Verdier duality on schemes of finite type over $S$
(see, e.g.,~\cite[Ch. 8, Ch. 10.1]{Fu} for the $\ell$-adic formalism
of Verdier duality in this setting). As usual, for a scheme of finite
type over $S$ with structural morphism $\varpi$, we set
$\dual(\mcF)=\Hom(\mcF,\varpi^!\bQl)$. Dualizing the
morphism~(\ref{eq-morph}), we obtain a morphism
\begin{equation}\label{eq-iso2}
  \dual Rf_! \Bigl( g^* ( Ru_*\bQl) \otimes \bigotimes_{j=1}^k
  \bigotimes_{i=1}^l \mcL_{\widetilde{\chi}_j}( x_{i,j}/x_{l+i,j} ) \Bigr) \to \dual
  Ri_* \bQl [2l(k-1)],
\end{equation}
that is also nontrivial, since by double-duality its dual
is~(\ref{eq-morph}).
 
We have
$$
\dual Ri_* \bQl = Ri_! \dual\bQl = R i_! i^! \bQl = R
i_! \bQl[-2] = R i_* \bQl [-2],
$$ 
where the last two
equalities follow respectively from the fact that $i$ is the inclusion
of a smooth divisor of codimension one and the fact that $i$ is
proper.  The left-hand side of~(\ref{eq-iso2}) is
$$
Rf_* \dual\Bigl( g^* ( Ru_*\bQl) \otimes \bigotimes_{j=1}^k
\bigotimes_{i=1}^l \mcL_{\widetilde{\chi}_j}( x_{i,j}/x_{l+i,j} )
\Bigr)= Rf_* \dual( g^* (Ru_* \bQl)) \otimes \bigotimes_{j=1}^k
\bigotimes_{i=1}^l \mcL_{\widetilde{\chi}_j^{-1} }( x_{i,j}/x_{l+i,j}
),
$$
since duality is local, and therefore commutes with twisting with a
locally constant sheaf.  Hence the existence of a non-trivial
morphism~(\ref{eq-iso2}) would lead to a morphism
$$ 
i^* Rf_* \dual g^* ( Ru_*\bQl) \otimes \bigotimes_{j=1}^k
\bigotimes_{i=1}^l \mcL_{\widetilde{\chi}_j^{-1} }( x_{i,j}/x_{l+i,j}
) \to \bQl[2l(k-1) +2]
$$
that is nontrivial at $\eta$. Finally, this would force the stalk of
the sheaf
$$ 
i^* Rf_* \dual g^* ( Ru_*\bQl) \otimes \bigotimes_{j=1}^k
\bigotimes_{i=1}^l \mcL_{\widetilde{\chi}_j^{-1} }( x_{i,j}/x_{l+i,j}
)
$$
in degree $-2l(k-1)-2$ to be nontrivial at the generic point of
$\Aa^{2l+1}$. We will now prove that this last property fails.
 
Away from the vanishing set of $g$, the sheaf
$g^* (Ru_* \bQl)$ is the constant sheaf $\bQl$, so
its dual is $\bQl [ 2 (2l(k-1)  ) ] $, where
$2l (k-1)$ is the relative dimension of $X$.  

On the other hand, we claim that the morphism $g$ is smooth in a
Zariski-open neighborhood of the vanishing set of $g$. To check this,
because $g'=gt$, it suffices to check that $g'$ is smooth in a
neighborhood of its vanishing set. Examining just the contribution
$\sum_{j=1}^k x_{i,j}$ to $g'$, observe that the only equation
defining $X_2$ involving $(x_{i,1},\ldots, x_{i,k})$ is of the form
$\prod_{j=1}^k x_{i,j}=\alpha$, so the derivative of this contribution
in a transverse direction is nonzero, and $g'$ is smooth, unless
$x_{i,1}=x_{i,2} = \cdots= x_{i,k}$. In this case, all the $x_i$ are
equal to some $k$-th root of $s (r+b_i)$, and thus
$$
g'=\sum_{i=1}^l (s(r+b_i))^{1/k}-
\sum_{i=l+1}^{2l} (s(r+b_i))^{1/k}
$$
which is non-zero when $(r,\bfb)\in U$.  

Since $g$ is smooth in a neighborhood of the vanishing locus of $g$,
the sheaf
$\dual g^* (Ru_* \bQl) = g^! \dual(Ru_* \bQl) $ is
there a shift (and Tate twist) of $g^* \dual(Ru_* \bQl)$,
which is a shift (and Tate twist) of $g^* Ru_! \bQl$, and
thus vanishes on the zero-set of $g$. We conclude that
$\dual g^* (Ru_* \bQl)$ is everywhere supported in degree
$$
-4l(k-1).
$$
Finally, we observe that $f$ is an affine morphism from a scheme of
dimension $2l(k-1) $. By results of Gabber (see~\cite[XV,
Theorem 1.1.2]{travaux-de-gabber}), the support of the sheaf
$$
R^df_* \dual\Bigl( g^* (Ru_* \bQl) \otimes \bigotimes_{j=1}^k
\bigotimes_{i=1}^l \mcL_{\widetilde{\chi}_j^{-1} }( x_{i,j}/x_{l+i,j}
) \Bigr)
$$
has dimension $ 2l(k-1) -d- 4l(k-1)$ relative to $S$. Hence, its stalk
in degree $2 - 2l(k-1) $ has support of dimension
$$ 
2l (k-1) +2l(k-1) -2 - 4l (k-1)=-2
$$
and therefore vanishes at the generic point of the special fiber,
which has dimension $-1$ (relative to $\Spec(\Oc^{et}_{\eta})$). This
is the desired contradiction.

\end{proof}

\section{Injectivity}

Let
$$
\mcW_1=\mcV^{\Delta}\cup \{\uple{b}\in\Aa^{2l}\,\mid\, \text{ at most two
  coordinates of $\uple{b}$ have multiplicity $1$}\}.
$$
This is a closed subvariety of codimension $l-1$ of
$\Aa^{2l}_{\Zz}$. The goal of this section is to prove the following
injectivity statement for $\theta_{\uple{b}}$:

\begin{theorem}\label{th-inj} 
  Let $p>2k+1$ be a prime and let $\Fq$ be a finite field of
  characteristic $p$ with $q$ elements. Let $\uple{\chi}$ be a
  $k$-tuple of $\ell$-adic characters of $\Fqt$ with Property \CGM.
\par
For $p$ large enough, depending only on $(k,l)$ and for
$\uple{b}\in\Aa^{2l}(\Fq)$ outside $\mcW_1(\Fq)$, the natural morphism
$$
\theta_{\uple{b}}\colon \End_{V_{\uple{b}}}(\mcK_{\uple{b}})\fleche{}
\End_{U_{\uple{b}}}(\mcR^*_{\uple{b}})
$$ 
is injective.
\end{theorem}

We begin with a lemma. First, we observe that for any $\uple{b}$, and
any geometrically irreducible component $\sheaf{H}$ of
$\sheaf{K}_{\uple{b}}$, we can meaningfully speak of the weight one
part of $R^1\pi_!\sheaf{H}$, since $\sheaf{H}$ is defined over a
finite field extension of $\Fq$.

\begin{lemma}\label{lm-inj1}
  For any $\uple{b}\in\Aa^{2l}(\Fq)$, the morphism $\theta_{\uple{b}}$
  is injective if, and only if, for any geometrically irreducible
  component $\sheaf{H}$ of $\sheaf{K}_{\uple{b}}$, the weight one part
  of $R^1\pi_!\sheaf{H}$ is non-zero.
\end{lemma}

\begin{proof} 
  Since $\mcK_{\uple{b}}$ is pointwise pure, hence geometrically
  semisimple, it is geometrically isomorphic to a direct sum
$$
\bigoplus_{i\in I}\sheaf{F}_i^{\oplus n_i}
$$
for some geometrically irreducible sheaves $\sheaf{F}_i$ and some
integers $n_i\geq 1$.  Then 
$$
R^1 \pi_! \mcK_{\uple{b}}\simeq\bigoplus_{i\in
  I}(R^1\pi_!\sheaf{F}_i)^{\oplus n_i},
$$
and the maximal weight one quotient of $R^1 \pi_! \mcK_{\uple{b}}$ is
also the corresponding direct sum of the maximal weight one quotients
$(R^1\pi_!\sheaf{F}_i)^{w=1}$ of $R^1\pi_!\sheaf{F}_i$, with
multiplicity $n_i$.
If one of these quotients vanishes, then any
$u\in \End_{V_{\uple{b}}}(\mcK_{\uple{b}})$ that is non-zero only on
the corresponding summand $\sheaf{F}_i$ satisfies
$\theta_{\uple{b}}(u)=0$. 

Conversely, suppose that all the quotients
$(R^1\pi_!\sheaf{F}_i)^{w=1}$ are non-zero.  By Schur's Lemma, the
endomorphism algebra $\End_{U_{\uple{b}}}(\mcR^*_{\uple{b}})$ is
isomorphic to a product of matrix algebras $M_{n_i}(\bQl)$. For each
$i$, $\theta_{\uple{b}}$ maps an endomorphism $u$ to the endomorphism
of $(R^1\pi_!\sheaf{F}_i)^{w=1,\oplus n_i}$ represented by a block
matrix with diagonal scalar matrices in each block, whose entries are
the coefficients of the matrix in $M_{n_i}(\bQl)$ corresponding to
$u$. Since the blocks have non-zero size, such a matrix is zero if and
only if $u$ is zero.
\end{proof}

Let $G$ be the geometric monodromy group of
$\HYPK_{k,\psi}(\uple{\chi})$.  Let $\uple{b}\in\Aa^{2l}(\Fq)$. We
denote by $B\subset \Aa^1$ the set of values $\{b_i\}$. For any family
$\uple{\rho}=(\rho_x)_{x\in B}$ of irreducible representations of $G$,
we denote by $\sheaf{H}_{\uple{\rho}}$ the sheaf
$$
\sheaf{H}_{\uple{\rho}}=\bigotimes_{x\in
  B}\rho_x(\HYPK_{k,\psi}(\uple{\chi}))(s(r+x)).
$$
on $\Aa^2$ with coordinates $(r,s)$.

\begin{lemma}\label{lm-inj2} 
  Assume that $\uple{\chi}$ has \CGM.  Any geometrically irreducible
  component $\sheaf{H}$ of $\sheaf{K}_{\uple{b}}$ is isomorphic to
  $\sheaf{H}_{\uple{\rho}}$ for some family
  $\uple{\rho}=(\rho_x)_{x\in B}$ such that, for all $x\in B$, the
  representation $\rho_x$ is an irreducible summand of the
  representation
  $\mathrm{Std}^{\otimes n_1}\otimes(\mathrm{Std}^{\vee})^{\otimes
    n_2}$, where
\begin{equation}\label{eq-ni}
n_1=\sum_{\substack{1\leq i\leq l\\b_i=x}}1,\quad\quad
n_2=\sum_{\substack{l+1\leq i\leq 2l\\b_i=x}}1.
\end{equation}
\end{lemma}

\begin{proof}
Write
$$
\mcK_{\uple{b}}=
\bigotimes_{x\in B} \HYPK_{k,\psi}(\uple{\chi})(s(r+x))^{\otimes n_1}\otimes
(\HYPK_{k,\psi}(\uple{\chi})(s(r+x))^{\vee})^{\otimes n_2}.
$$
By the Goursat--Kolchin--Ribet criterion (see~\cite{ESDE}
or~\cite{FKMSP}), which may be applied since the sheaf
$\HYPK_{k,\psi}(\uple{\chi})$ has geometric monodromy group $\SL_k$ or
$\Sp_k$ by Theorem~\ref{thm-geometric-monodromy}, the sheaf
$$
\bigoplus_{x\in B}\HYPK_{k,\psi}(\uple{\chi})(s(r+x))
$$
has geometric monodromy group $G^{|B|}$, so that its irreducible
components correspond exactly to the tuples $\uple{\rho}$.
\end{proof}

\begin{lemma}\label{lm-rank-without-local} 
  Let $\uple{b}$ be a point in $\Aa^{2l} - \mathcal V^\Delta$. Let
  $\sheaf{H}_{\uple{\rho}}$ be an irreducible component of
  $\sheaf{K}_{\uple{b}}$. Then the rank of
  $R^1 \pi_! \sheaf{H}_{\uple{\rho}}$ on the dense open set where
  $P_{\uple{b}}(r) \neq 0$ is equal to the rank of
  $\sheaf{H}_{\uple{\rho}}$ divided by $k$.
\end{lemma}

\begin{proof} 
  Note that the set where $P_{\uple{b}}$ doesn't vanish is indeed a
  dense open subset by Lemma~\ref{z-polynomial}.

  Let $r$ be such that $P_{\uple{b}}(r) \neq 0$. Then by proper base
  change, the stalk of $R^1 \pi_! \sheaf{H}_{\uple{\rho}}$ at $r$ is
  equal to $H^1_c( \Gg_{m,\bFq}, \sheaf{H}_{\uple{\rho},r})$.

  Because $P_{\uple{b}}(r)\neq 0$, Lemma~\ref{lm-tilde-psi} shows that
  the local monodromy representation at $\infty$ of
  $\mcK_{\uple{b},r}$ is isomorphic 
  to a sum of sheaves of the form
  $\mathcal L_\psi ( \alpha \cdot s^{1/k})$ for nonzero $\alpha$. Each
  sheaf $\mathcal L_\psi ( \alpha \cdot s^{1/k})$ has all breaks $1/k$
  at $\infty$, so the same is true for $\mcK_{\uple{b},r}$.  

  The sheaf $\sheaf{H}_{\uple{\rho},r}$ is a summand of
  $\mcK_{\uple{b},r}$, hence it also lisse on $\Gg_m$, tamely ramified
  at $0$, and has all breaks $1/k$ at $\infty$. Moreover, it also
  satisfies
$$
H^0_c(\Gg_{m,\bFq},\sheaf{H}_{\uple{\rho}})=
H^2_c(\Gg_{m,\bFq},\sheaf{H}_{\uple{\rho}})=0,
$$
and therefore
the Euler-Poincaré characteristic formula for a lisse sheaf on $\Gg_m$
implies that
$$
\dim H^1_c( \Gg_{m,\bFq}, \sheaf{H}_{\uple{\rho},r})=
-\chi(\Gg_{m,\bFq},\sheaf{H}_{\uple{\rho},r}) =
\swan_0(\sheaf{H}_{\uple{\rho},r})+\swan_{\infty}(\sheaf{H}_{\uple{\rho},r})\\
=\frac{1}{k}\rk(\sheaf{H}_{\uple{\rho}}).
$$
\end{proof}

In the next lemmas, we fix a point $\uple{b}$ in
$\Aa^{2l} - \mathcal V^\Delta$, and an index $i$ such that
$b_i\not=b_j$ for $j\not=i$.

We denote $\epsilon =-1$ if $1\leq i\leq l$, and $\epsilon=1$ if
$l+1\leq i\leq 2l$. For any character $\chi$, we denote $n_\chi$ the
multiplicity of $\chi$ in $\uple{\chi}$, which is $0$ if
$\chi\not\in \uple{\chi}$.

For an irreducible component
$$
\sheaf{H}_{\uple{\rho}}=\bigotimes_{x\in
  B}\rho_x(\HYPK_{k,\psi}(\uple{\chi}))(s(r+x))
$$
of $\mcK_{\uple{b}}$ (all are of this type by Lemma~\ref{lm-inj2}), we
denote
\begin{equation}\label{eq-mrho}
\sheaf{M}_{\uple{\rho}}= \bigotimes_{\substack{x\in B\\x\neq
    b_i}}\rho_x(\HYPK_{k,\psi}(\uple{\chi}))(s(r+x)).
\end{equation}
Since $\sheaf{M}_{\uple{\rho}}$ is tamely ramified at $0$, its local
monodromy representation at $s=0$ can be expressed as a sum of Jordan
blocks, which we write
$$
\bigoplus_{\eta} \mcL_{\eta} \otimes J(m_{\eta})
$$
where $\eta$ runs over a finite set of characters.

\begin{lemma}\label{lm-rank-with-local} 
With notation as above, 
the rank of the weight one part of $R^1 \pi_! \sheaf{H}_{\uple{\rho}}$
on the nonempty open set where $P_{\uple{b}}(r)\neq 0$ is equal to
$$
\sum_{\eta}\max( m_{\eta} - n_{\eta^\epsilon} ,0 ).
$$
\end{lemma}

\begin{proof} 
  Because $b_i$ occurs with multiplicity one in $B$, the
  representation $\rho_{b_i}$ is necessarily the standard
  representation if $i \leq l$ or its dual if $i>l$
  (see~(\ref{eq-ni})), and in any case has rank $k$. This implies that
$$
\rk(\sheaf{H}_{\uple{\rho}})=k\rk(\sheaf{M}_{\uple{\rho}})
$$ 
and hence by Lemma~\ref{lm-rank-without-local}, we have
$$
\rk (R^1 \pi_! \sheaf{H}_{\uple{\rho}})= \rk(\sheaf{M}_{\uple{\rho}})
= \sum_{\eta} m_{\eta},
$$
so that it suffices to show that the weight $<1$ part of
$R^1 \pi_! \sheaf{H}_{\uple{\rho}}$ has the rank
$$
\sum_{\eta} \min(m_{\eta} ,n_{\eta^\epsilon}).
$$

To prove this, observe that the weight $<1$ part is the sum over the
singularities of the sheaf of the local monodromy invariants (see,
e.g.,~\cite[Lemma 4.22(2)]{KMS}). Because $\sheaf{H}_{\uple{\rho},r}$
is a summand of $\mcK_{\uple{b},r}$ which by Lemma~\ref{lm-tilde-psi}
has no nontrivial local monodromy invariants at $\infty$,
$\sheaf{H}_{\uple{\rho},r}$ has no nontrivial local monodromy
invariants at $\infty$.

If $i \leq l$, then the local monodromy representation at $0$ is given
by
\begin{align*}
  \sheaf{H}_{\uple{\rho},r} 
  = \sheaf{M}_{\uple{\rho}} \otimes \HYPK_{k,\psi}(\uple{\chi})
  (s(r+b_i)) &= \Bigl( \bigoplus_{\eta} \mcL_{\eta} \otimes J(m_{\eta})
               \Bigr) \otimes 
               \Bigl( \bigoplus_{\chi\in
               \uple{\chi}}\mcL_{\chi}\otimes J(n_{\chi}) \Bigr)\\
             & = \bigoplus_{\eta} \bigoplus_{\chi \in \uple{\chi} } \mcL_{\eta
               \chi}  \otimes J(m_{\eta}) \otimes J(n_\chi).
\end{align*}
The dimension of the invariant subspace of
$ \mcL_{\eta\chi}\otimes J(m_{\eta}) \otimes J(n_\chi) $ is zero
unless $\eta \chi=1$, in which case it is $\min(m_{\eta},n_\chi)$,
hence the result follows in that case.
If $l+1\leq i\leq 2l$, the same calculation applies, except that
$\mcL_{\chi^{-1}}$ appears instead of $\mcL_\chi$.
\end{proof}

The next lemma continues with the same notation.

\begin{lemma}\label{lm-injectivity-plus-epsilon} 
  Assume that $\uple{\chi}$ has \CGM.  
  Then the rank of the weight one part of
  $R^1 \pi_! \sheaf{H}_{\uple{\rho}}$ is at least two.
\end{lemma}

\begin{proof} 
  By the previous lemma, it is enough to prove that
\begin{equation}\label{eq-rank-1}
  \sum_{\eta} \max( m_{\eta} - n_{\eta^\epsilon} ,0 ) \geq 2.
\end{equation}

  Since $\uple{b} \not \in \mcW_1$, there are at least three elements
  of $B$ that occur with multiplicity one, say $b_i$, $b_j$ and
  $b_{j'}$.

  Let $\delta = 1$ if $j \leq l$ and $\delta=-1$ if $j>l$, so that
  $\rho_{b_j}$ is the standard representation if $\delta=1$ and the
  dual representation if $\delta=-1$.

  Let
  $$
  \sheaf{M'}_{\uple{\rho}}= \bigotimes_{\substack{x\in B\\x\neq
      b_i,b_j}}\rho_x(\HYPK_{k,\psi}(\uple{\chi}))(s(r+x))
$$
so that
$$
\sheaf{M}_{\uple{\rho}} = \sheaf{M'}_{\uple{\rho}} \otimes \HYPK_{k,\psi}(\uple{\chi}) ( s
(r+b_j))
$$ 
if $\delta=1$ and
$$
\sheaf{M}_{\uple{\rho}}=\sheaf{M'}_{\uple{\rho}} \otimes \HYPK_{k,\psi}(\uple{\chi}) ( s
(r+b_j))^{\vee}
$$ 
if $\delta=-1$.
    
Let $\mcL_{\theta}\otimes J(r)$ be a Jordan block in the local
monodromy representation of $\sheaf{M'}_{\uple{\rho}}$ at $s=0$.  We
estimate the contribution from this factor in the local monodromy
representation~(\ref{eq-mrho}) of $\sheaf{M}_{\uple{\rho}}$.

This contribution contains a direct sum
\begin{equation}\label{eq-contrib-j}
\bigoplus_{\chi \in \uple{\chi}} \mcL_{\chi^\delta \theta} \otimes
J(n_\chi+r-1).
\end{equation}
If the character $\theta$ is nontrivial, then the tuple of characters
$\theta^{\epsilon}\uple{\chi}^{\epsilon \delta}$ cannot be equal to
$\uple{\chi}$, up to permutation because this would contradict the
\CGM\ assumption. Hence, there exists a character $\chi$ such that
$n_\chi > n_{ \chi^{\delta \epsilon } \theta^\epsilon}$, and therefore
the Jordan blocks~(\ref{eq-contrib-j}) include a character
$\eta=\chi^{\delta}\theta$ with $m_{\eta}>n_{\eta^{\epsilon}}$. Hence
these blocks have a contribution 
$$
\geq \min(n_{\chi}+r-1-n_{\chi^{\delta \epsilon} \theta^{\epsilon}},0)\geq r
$$ 
to the sum on the left-hand side of~(\ref{eq-rank-1}).

On the other hand, if $\theta$ is trivial, then the character $\chi$
with $n_{\chi}$ maximal contributes
$$
\geq \min(n_{\chi}+r-1-n_{\chi},0)=r-1.
$$
    
In particular, we obtain~(\ref{eq-rank-1}) except if the local
monodromy of $\sheaf{M'}_{\uple{\rho}}$ at zero consists of at most
one unipotent Jordan block of rank two, or of at most one nontrivial
character of rank one, plus a sum of any number of trivial
representations. This conditions means that local monodromy
representation of $\sheaf{M'}_{\uple{\rho}}$ at zero is either trivial
or is a pseudoreflection (unipotent or not).
  
In the first case, we have a sheaf with trivial local monodromy at $0$
that is expressed as a tensor product. Then all the tensor factors
must have scalar local monodromy at $0$. This is impossible here, since
one of the tensor factors is $\HYPK_{k,\psi}(\uple{\chi}) (s (r+b_{j'}))$ or its dual, and
the local monodromy of this sheaf is not scalar (because $k \geq 2$).
 
If the local monodromy representation is a pseudoreflection, then when it
is expressed as a tensor product, all but one of the tensor factors
must be one-dimensional, and the remaining factor must have local
monodromy that is given by a pseudoreflection times a scalar.  Again,
because one of the tensor factors is $\HYPK_{k,\psi}(\uple{\chi}) (s (r+b_{j'}))$ or its
dual, this must be the special factor, and this can only happen when
$k=2$ by Proposition~\ref{pr-kl}.  All the remaining tensor factors
are one-dimensional. But since the geometric monodromy group is
$\SL_2$ in that case (because $\uple{\chi}$ has \CGM), and the only
one-dimensional representation of $\SL_2$ is the trivial
representation, and this only appears in even tensor powers of the
standard representation, we conclude that all remaining factors must
have even multiplicity. This is a contradiction, since we have three
factors with multiplicity one, and the sum of the multiplicities is
$2l$, which is even.
%
%
\end{proof}

Now Theorem~\ref{th-inj} follows immediately from Lemma~\ref{lm-inj1}
and Lemma \ref{lm-injectivity-plus-epsilon}.

\section{Specialization statement}

We continue with the previous notation. Recall that
$X_{\infty}=\Aa^{2l}-\mcV^{\Delta}$ and that $X_j$ is defined in
Definition~\ref{def-xj}. We recall that we have the projection
$f\colon U\to \Aa^{2l}$.

\begin{lemma}\label{strata-finite-etale}
For each $j$, the subvariety $X_j$ is closed in $X_{\infty}$.

For each irreducible component $X$ of $X_j$ that intersects the
characteristic zero part, the morphism
$$
f\colon Z\cap f^{-1}(X- X \cap X_{j-1} )\to X- X\cap X_{j-1}
$$
is finite étale.
\end{lemma}

\begin{proof} These claims follow from Lemma
  \ref{z-polynomial}. Indeed, $Z$ is the solution set of a family of
  nonzero polynomials in one variable indexed by points of
  $X_{\infty} = \Aa^{2l} - \mcV^{\Delta}$. The set $X_j$ is
  constructible, so to show it is closed it suffices to show that it
  is closed under specialization. The polynomial factorizes completely
  over any geometric generic point into one distinct factor for each
  root, raised to some power, and each factor has at most one root
  over the special point, so the number of roots over the special
  point is at most the number of roots over the generic point, as
  desired.

  To check that $Z\cap f^{-1}(X- X \cap X_{j-1} ) $ is finite
  \'{e}tale over $X- X \cap X_{j-1} $, we consider the polynomial
  $P(r)$ over the \'{e}tale local ring of a point of
  $X- X \cap X_{j-1}$, which is an integral strict Henselian local
  ring, and use the fact that the polynomial has the same number of
  roots over the special point and over the generic point. By the
  previous discussion each linear factor over the geometric generic
  point must admit a root over the residue field, which means the
  polynomial is monic. Because it is monic, and the ring is strict
  henselian, we can factor it into a product of irreducible factors,
  each with exactly one root in the residue field. Over the generic
  point each such factor will have only one root in the residue field,
  hence have only one root in the fraction field. Therefore, because
  the generic point has characteristic zero, so all polynomials are
  separable, each such factor is a power of $(x-\alpha)$ where
  $\alpha$ is its unique root, so the polynomial is a product of
  linear factors, with at most one distinct linear factor with each
  possible root in the residue field, hence its vanishing set is the
  disjoint union of the vanishing sets of these linear factors and
  thus is finite \'{e}tale.
\end{proof}

Fix $j\geq 0$. Let $X\subset X_j\subset \Aa^{2l}$ be an irreducible
component of $X_{j}$ over $\Zz$ which intersects the characteristic
zero part. We consider a finite field $\Fq$ of characteristic $p>2k+1$
such that $X_{\bFq}$ is irreducible and nonempty.

\begin{lemma}\label{strata-tame} 
  Let $\uple{\chi}$ be a $k$-tuple of characters of $\Fqt$.  The sheaf
  $\mcR^*|(U\cap f^{-1}(X_{\bFq} - X_{\bFq} \cap X_{j-1}))$ is tamely
  ramified around the divisor $Z \cup \{\infty\}$.
\end{lemma}

\begin{proof} 
  Let $n$ be the lcm of the orders of the characters $\chi_i$.  By the
  remarks before Proposition~\ref{pr-integrality}, there exists a
  tuple $\widetilde{\uple{\chi}}$ of characters of $\mmu_n$ such that
  $\uple{\chi}$ is associated to this tuple. Let
  $\mcR^{univ}(\widetilde{\uple{\chi}})$ be the sheaf over
  $\Zz[\mmu_n,1/(n\ell)]$ given by Proposition~\ref{pr-integrality}.
  This sheaf $\mcR^{univ}(\widetilde{\uple{\chi}})$ is lisse on the
  open set $U\cap f^{-1}(X_j - X \cap X_{j-1})$, whose complement is
  the étale divisor $Z \cup \{\infty\}$. Hence, by Abyankhar's
  Lemma~\cite[Exposé XIII, \S 5]{sga1}, the sheaf
  $\mcR^{univ}(\widetilde{\uple{\chi}})$ is tamely ramified, and hence
  so is 
$$
\mcR^{univ}(\widetilde{\uple{\chi}}) | \Aa^{1+2l}_{\Fq} =
\mcR(\uple{\chi}),
$$
and also $\mcR^*(\uple{\chi})$.
\end{proof}

\begin{proposition}\label{pr-specialization}
  Let $\eta$ be the generic point of $X_{\bFq}$, and let $\bar{\eta}$
  be a geometric generic point over $\eta$.  Let $\uple{\chi}$ be a
  $k$-tuple of characters of $\Fqt$ with Property \CGM.  Suppose that
  $$
  \dim \End_{U_{\bar{\eta}}}(\mcR^*_{\bar{\eta}})
  =
  \dim \End_{V_{\bar{\eta}}}(\mcK_{\bar{\eta}}).
  $$
  Let $\uple{b}\in X(\Fq)$ such that $\uple{b}\notin X_{j-1}$ and
  $\uple{b}\notin \mcW_1$.  Then we have
  $$
  \dim \End_{U_{\uple{b}}}(\mcR^*_{\uple{b}}) =
  \dim \End_{V_{\uple{b}}}(\mcK_{\uple{b}}).
  $$
\end{proposition}

\begin{proof}
Consider the sheaf
$$
\mcE=R^2f_!(\mcR^*\otimes\mcR^{*,\vee})
$$
on $\Aa^{2l}_{\Fq}$. We claim that
\begin{enumerate}
\item[(a)] The restriction of $\mcE$ to $X_j-X_{j-1}$ is lisse.
\item[(b)] We have an isomorphism
$$
\mcE_{\bar{\eta}}\simeq \End_{U_{\bar{\eta}}}(\mcR^*_{\bar{\eta}})(-1).
$$
\item[(c)] We have an isomorphism
$$
\mcE_{\uple{b}}\simeq \End_{U_{\uple{b}}}(\mcR^*_{\uple{b}})(-1).
$$
\end{enumerate}
Moreover, let $g\colon V\to \Aa^{2l}$ be the map $(r,s,\uple{b})\mapsto
\uple{b}$ over $\Zz$ and
$$
\widetilde{\mcE}=R^4g_!(\mcK\otimes\mcK^{\vee})
$$
on $\Aa^{2l}_{\Fq}$. We claim that
\begin{enumerate}
\item[(a')] The restriction of $\widetilde{\mcE}$ to $X_j-X_{j-1}$ is
  lisse.
\item[(b')] We have an isomorphism
$$
\widetilde{\mcE}_{\bar{\eta}}\simeq \End_{V_{\bar{\eta}}}(\mcK_{\bar{\eta}})(-1).
$$
\item[(c')] We have an isomorphism
$$
\widetilde{\mcE}_{\uple{b}}\simeq \End_{V_{\uple{b}}}(\mcK_{\uple{b}})(-1).
$$
\end{enumerate}

Assuming these facts, we have
\begin{multline*}
 \dim \End_{U_{\uple{b}}}(\mcR^*_{\uple{b}}) = \dim \mcE_{\uple{b}} = \dim \mcE_{\bar{\eta}} = \dim \End_{U_{\bar{\eta}}}(\mcR^*_{\bar{\eta}})\\ = \dim  \End_{V_{\bar{\eta}}}(\mcK_{\bar{\eta}}) = \dim \widetilde{\mcE}_{\bar{\eta}} = \dim \widetilde{\mcE}_{\uple{b}} = \dim \End_{V_{\uple{b}}}(\mcK_{\uple{b}}) 
 \end{multline*}
 with the identities following from respectively (c), (a), (b), the
 assumption, (b'), (a'), and (c'). (In particular, when we apply
 assumption (a) and (a'), we use the fact that $\uple{b}$ is a
 specialization of $\bar{\eta}$, hence they lie on the same connected
 component of $X_j - X_{j-1}$, and so any lisse sheaf on
 $X_j - X_{j-1}$ has equal ranks at these two points.)


We now prove the claims. The assertions (b)/(b') and (c)/(c') follow
from the proper base change theorem, Poincaré duality, and semisimplicity.

Assertion (a) is a consequence of Deligne's semicontinuity theorem and
the tameness of $\mcR^*$. Specifically, by
Lemma~\ref{strata-finite-etale}, we know that $U$, over
$X_{\Fq}-(X_{\Fq} \cap X_{j-1})$, is the complement of a finite
\'{e}tale divisor inside a morphism smooth and proper of relative
dimension one, and $\mcR^* \otimes \mcR^{*,\vee}$ is a lisse sheaf on
it. By Lemma \ref{strata-tame}, the Swan conductor of
$\mcR^* \otimes \mcR^{*,\vee}$ at this divisor vanishes, and so by
Deligne's semicontinuity theorem \cite[Corollary 2.1.2]{LaumonSMF} the
cohomology sheaf is lisse.

Assertion (a'): Let $Y = X_{\Fq}-(X_{\Fq} \cap X_{j-1})$. Then
$\mcK \otimes \mcK^\vee$ is lisse on $V \times_{\Aa^{2l}} Y$. Let
$\left( \mcK \otimes \mcK^\vee\right)^{\pi_1 ( V \times_{\Aa^{2l}}
  Y)}$ be its (geometric) monodromy invariants. Then there is a
natural map
$$
\left( \mcK \otimes \mcK^\vee\right)^{\pi_1 ( V \times_{\Aa^{2l}} Y)}
\to \mcK \otimes \mcK^\vee
$$
over $V \times_{\Aa^{2l}} Y$, where we interpret
$\left( \mcK \otimes \mcK^\vee\right)^{\pi_1 ( V \times_{\Aa^{2l}}
  Y)}$ as a constant sheaf. This induces by functoriality a map
$$
R^4 g_! \left( \mcK \otimes \mcK^\vee\right)^{\pi_1 ( V
  \times_{\Aa^{2l}} Y)} \to R^4 g_! \mcK \otimes \mcK^\vee
$$ 
over $Y$. Because $V$ is an open subset of $\Aa^{2l+2}$ whose fibers
under $g$ are all nonempty, the top cohomology of a constant sheaf
along $g$ is a constant sheaf, so this gives a map
$$
\left( \mcK \otimes \mcK^\vee\right)^{\pi_1 ( V \times_{\Aa^{2l}} Y)}
\to R^4 g_! \mcK \otimes \mcK^\vee.
$$

We claim that this last map is an isomorphism. It is sufficient to
check this on the stalk at each point $\bfb$. To do this, first check
that the monodromy group of $\mcK \otimes \mcK^\vee$ over
$V \times_{\Aa^{2l}} Y$ is equal to the monodromy of the same sheaf on
$V_{\bfb}$. This can be done using Goursat-Kolchin-Ribet, since
$\uple{\chi}$ has \CGM\ and $p>2k+1$. We also use the fact that,
because $Z$ is finite etale over $Y$, and $Z$ includes
$\{-b_1,\dots,-b_{2l} \}$, no $b_i,b_j$ that are distinct generically
on the $Y$ stratum can become equal at any point of $Y$.

Next observe that this map is simply the natural map from the
monodromy invariants of $\mcK \otimes \mcK^{\vee}$ to the monodromy
coinvariants of $\mcK \otimes \mcK^\vee$. Because the monodromy is
semisimple, it is an isomorphism.
\end{proof}

\section{Diophantine preliminaries for the proof of the generic
  statement}
\label{sec-dio}

This section uses independent notation from the rest of the paper. In
particular, we will use the letter $k$ to denote finite fields.

We will use the following variant of the Diophantine Criterion for
irreducibility of Katz (compare~\cite[p. 25]{mmp} and~\cite[Lemma
4.14]{KMS}).

\begin{lemma}\label{lm-dio}
  Let $w$ be an integer.  Let $X$ be a geometrically irreducible
  separated scheme of finite type over a finite field $k$, and let $U$
  be a normal open dense subset of $X$. Let $\ell$ be a prime
  different from the characteristic of $k$. Let $\sheaf{F}$ be an
  $\ell$-adic sheaf on $X$, mixed of weights $\leq w$ on $X$, and
  lisse and pure of weight $w$ on $U$. We have then
\begin{equation}\label{eq-dio-irred}
\dim \End_{\pi_1(U\times\bFq)}(\sheaf{F}|U)= \limsup_{\nu\to+\infty}
\frac{1}{|k|^{\nu(\dim(X_{\bFq})+w)}}\sum_{x\in X(k_{\nu})}
|t_{\sheaf{F}}(x;k_{\nu})|^2,
\end{equation}
where $k_{\nu}$ is the extension of $k$ of degree $\nu$ in a fixed
algebraic closure.
\par
In particular, if the right-hand side of the formula above is equal to
$1$, then $\mcF|U$ is geometrically irreducible.
\end{lemma}

\begin{proof}
  Let $n=\dim(X_{\bFq})$. Up to performing a Tate twist on $\mcF$, we
  may assume that $w=0$.  For any $x\in X(k_\nu)$ we have then
$$
|t_{\sheaf{F}}(x;k_{\nu})|^2\leq \rk(\mcF)^2
$$
hence by trivial counting we get
\begin{align*}
\frac{1}{|k|^{n\nu}}\sum_{x\in X(k_{\nu})}
|t_{\sheaf{F}}(x;k_{\nu})|^2&=\frac{1}{|k|^{n\nu}}\sum_{x\in U(k_{\nu})}
|t_{\sheaf{F}}(x;k_{\nu})|^2+\frac{1}{|k|^{n\nu}}\sum_{x\in (X-U)(k_{\nu})}
|t_{\sheaf{F}}(x;k_{\nu})|^2	\\
&=\frac{1}{|k|^{n\nu}}\sum_{x\in U(k_{\nu})}
|t_{\sheaf{F}}(x;k_{\nu})|^2+O_\mcF(|k|^{-\nu}).
\end{align*}
This shows that we may restrict the sum on the right-hand side
of~(\ref{eq-dio-irred}) to $U(k_\nu)$.

Since $\mcF$ and its dual $\mcF^{\vee}$ are lisse and pointwise pure
of weight $0$ on $U$, the sheaf $\End(\mcF)=\mcF\otimes \mcF^{\vee}$
is also lisse and pointwise pure of weight $0$ on $U$. Moreover, for
all $x\in
U(k_\nu)$, we have
$$
t_{\End(\mcF)}(x;k_\nu)=|t_{\sheaf{F}}(x;k_{\nu})|^2.
$$ 
\par
By the Grothendieck--Lefschetz trace formula, we have
\begin{multline*}
  \frac{1}{|k|^{n\nu}}\sum_{x\in U(k_{\nu})}
  |t_{\sheaf{F}}(x;k_{\nu})|^2=
  \frac{1}{|k|^{n\nu}}\tr(\Frob_{k_\nu}|H_c^{2n}(U\times\bFq,\End(\mcF)))\\
  +\frac{1}{|k|^{n\nu}}\sum_{i=0}^{2n-1}
  (-1)^i\tr(\Frob_{k_\nu}|H_c^{i}(U\times\bFq,\End(\mcF))).
\end{multline*}
By Deligne's Riemann Hypothesis~\cite{WeilII}, all eigenvalues of the
Frobenius of $k_{\nu}$ acting on the cohomology group
$H_c^{i}(U\times\bFq,\End(\mcF_{\ov k}))$ have modulus
$\leq |k|^{i/2}$, and therefore
$$
|\tr(\Frob_{k^\nu}|H_c^{i}(U\times\bFq,\End(\mcF)))|\leq
\dim(H_c^{i}(U\times\bFq,\End(\mcF)))|k|^{i\nu/2},
$$ 
so that we derive
$$
\frac{1}{|k|^{n\nu}}\sum_{x\in U(k_{\nu})}
|t_{\sheaf{F}}(x;k_{\nu})|^2 =\frac{1}{|k|^{n\nu}}
\tr(\Frob_{k_\nu}|H_c^{2n}(U\times\bFq,\End(\mcF)))+O(|k|^{-\nu/2}).
$$
On the other hand, we have a Frobenius-equivariant isomorphism 
$$
H_c^{2n}(U\times\bFq,\End(\mcF))\simeq \End(\mcF)_{\pi_1(U\times\bFq)}(-n).
$$
The eigenvalues of Frobenius on $\End(\mcF)_{\pi_1(U\times\bFq)}(-n)$
have modulus $q^n$. Therefore
$$
|k|^{-n\nu}\tr(\Frob_{k^\nu}|H_c^{2n}(U\times\bFq,\End(\mcF)))
$$
is the sum of the $\nu$-th power of
$\dim H_c^{2n}(U\times\bFq,\End(\mcF))$ complex numbers, each of of
modulus $1$, and by a standard lemma, we have
therefore
\begin{align*}
  \limsup_{\nu\to+\infty}
  \frac{1}{|k|^{n\nu}}\tr(\Frob_{k^\nu}|H_c^{2n}(U\times\bFq,\End(\mcF)))
  &=\dim H_c^{2n}(U\times\bFq,\End(\mcF))\\
  &=\dim\End_{\pi_1(U\times\bFq)}(\mcF),
\end{align*}
by the geometric semi-simplicity of $\mcF|U$.
\end{proof}

This result, combined with the injectivity statement, reduces the
desired isomorphism to a bound on exponential sums, where $\uple{b}$
are summed over a stratum of the stratification. The technique we will
use to obtain cancellation is a form of separation of variables, where
we essentially obtain cancellation in the sum over each individual
coordinate $b_i$.

We now describe a general geometric form of the type of separation of
variables that we will use.

\begin{itemize}
\item[-] Let $m$ and $N$ be natural numbers. Let $S$ be a finite set.

\item[-] Let $\mathcal O_K$ be the ring of integers of a number field,
  and $B$ a separated scheme of finite type over $\mathcal O_K [1/N]$.

\item[-]Let $C_i$ for $i \in S$ be curves over $B$. Let $A$ be a
  smooth geometrically irreducible curve over $\Zz[1/N]$. We will use
  $s$ as a variable for points of $A$ and $x_i$ for points of $C_i$.

\item[-]We denote $\mathcal{C}=C_1\times_B\cdots\times_B C_n$. We view functions
on $C_i$ as functions on $\mathcal{C}$ by composing with the $i$-th
projection.

\item[-]For $1\leq j\leq m$, let $\uple{f}_j=(f_{i,j})_{1\leq i\leq
  n}\in\Gamma$ be a tuple of functions on the curves $C_i$, and let $g_j$ be a function on $B$.

\item[-]Let $Y\subseteq \mathcal{C}$ be the common zero locus of the $m$
functions
$$
\Sigma_j:=g_j + \sum_{ i \in S} f_{i,j}\in\Gamma(\mathcal{C},\mathcal{O}_{\mathcal{C}}),\ j=1,\cdots,m.
$$

\item[-]Let $\pi \colon Y\times A\to Y$ be the obvious projection, and
$g_i\colon Y\times A\to C_i\times A$ the obvious morphisms.

\item[-]Let $\ell$ be a prime number dividing $N$. For $i \in S$, and
  $q$ some prime ideal of $\mathcal O_K$ coprime to $N$, we assume
  given a lisse $\ell$-adic sheaf $\sheaf{F}_i$, pointwise pure of
  weight $0$, on $C_i\times A_\Fq$. We denote by
  $(\rho,x_i,s)\mapsto t_i(\rho,x_i,s;k)$ the trace function of
  $\sheaf{F}_i$ over some finite extension $k/\Fq$.

\item[-]For $s\in A(k)$ and $\rho \in B(k)$ we set 
$$\mcF_{i, \rho, s}:=\mcF_i|C_i\times_B \{ \rho\} \times\{s\}$$
the sheaf on $C_i\times k$ obtained by restricting to the fiber of
$\rho$ and ``freezing'' the $s$-variable.  We assume that for any $q$,
any $k/\Fq$ and any point $s\in A(k)$ the conductor of
$\mcF_{i,\rho, s}$ is bounded by some constant $C\geq 1$

\item For $q$ some prime of $\mathcal O_K$ coprime with $n$, we are given a lisse $\ell$-adic sheaf $\sheaf{G}$, pointwise pure of weight $0$, on $B \times A_{\Fq}$. We denote by $(\rho,s) \mapsto t_* (\rho,s;k)$ its trace function.

\end{itemize}

We make the following ``twist-independence'' assumption:

\begin{TI}For all $i$, for all $\rho \in B$ and for all $s_1\not=s_2$
  in $A$, the lisse sheaf
  $\sheaf{F}_{i,\rho, s_1}\otimes\sheaf{F}_{i,\rho, s_2}^{\vee}$ on
  each geometrically irreducible component of $C_{i, \rho} $ has no
  geometrically irreducible component that is of rank $1$.
\end{TI}

The implicit constants associated with the symbols $O(\cdots)$ or
$\ll$ are assumed to depend on $\mcC,A$, the maps
$(\uple{f}_j)_{j=1,\cdots,m}$, and the conductors of the sheaves
involved.

The  main estimate on exponential sums we will need is the following

\begin{proposition}\label{lm-dio-new}
  Assume that Assumption \emph{(TI)} holds. We have
\begin{multline}\label{th-diophantine-cor-eq}
\sum_{(\rho,\uple{x})\in Y(k)}\Bigl| \sum_{s\in A(k)} t_*(\rho,s;k) \prod_{i=1}^n
t_i(\rho,x_i,s;k)\Bigr|^2=\\ \sum_{(\rho,\uple{x})\in Y(k)}\sum_{s\in
  A(k)}|t_* (\rho,s,k)|^2\prod_{i=1}^n|
t_i(\rho,x_i,s;k)|^2+O\Bigl(|k|^{\dim B+ |S|/2+2}\Bigr).
\end{multline}
\end{proposition}

\begin{rem} One can often show (by fibering by curves) that as
  $|k|\ra\infty$ the first term on the righthand side of
  \eqref{th-diophantine-cor-eq} satisfies
$$
\sum_{(\rho,\uple{x})\in Y(k)}\sum_{s\in A(k)}|t_* (\rho,s,k)|^2
\prod_{i=1}^n| t_i(\rho,b_i,s;k)|^2\gg |k|^{\dim(Y\times A)_{\bFq}}
$$
while the error term is
$$
\ll |k|^{(n-m+1)-1/2}\ll |k|^{\dim(Y\times A)_{\bFq}-1/2}
$$ as soon
as 
$$
m\leq \frac{|S|-3}{2}.
$$
\end{rem}

\begin{example}  
  Take $B$ a point, $C_i=A=\Gg_m$,
  $\sheaf{F}_i=[(b_i,s)\mapsto b_is]^*\HYPK_2$ on $\Gg_m^2$, and
  $\sheaf{G}= \bQl$. Define $f_{i,j}(b_i)=b_i^j$ and $Y$ be the
  subvariety of $\Gm^n$ defined by the equations
  $$\sum b_i=\cdots=\sum b_i^m=0;$$
  One has $\dim V_{\bFq}=n-m$ for $q$ large enough. Then (TI) is
  satisfied and Proposition \ref{lm-dio-new} states that
$$
\multsum_{\substack{b_1,\ldots, b_n\in \Fqt\\
    \sum b_i=\cdots=\sum b_i^m=0}} \Bigl|\sum_{s\in\Fqt}\prod_{i=1}^n
\hypk_2(b_is;q)\Bigr|^2=
\multsum_{\substack{b_1,\ldots, b_n\in \Fqt\\
    \sum b_i=\cdots=\sum b_i^m=0}} \sum_{s\in\Fqt}\prod_{i=1}^n\Bigl|
\hypk_2(b_is;q)\Bigr|^2 +O(q^{(n-m+1)-1/2}),
$$
provided $m\leq (n-3)/2$.
\end{example}

\begin{proof}
  We will omit the indication of the finite field, which is always
  $k$, in the notation for trace functions.  Opening the square, we
  have
\begin{multline}\label{eq-develop}
  \sum_{(\rho,\uple{x})\in Y(k)}\Bigl| \sum_{s\in A(k)}t_* (\rho,s)
  \prod_{i=1}^n t_i(\rho, x_i, s)\Bigr|^2= \sum_{(\rho,\uple{x})\in
    Y(k)}\sum_{s\in A(k)}|t_* (\rho,s)|^2 \prod_{i=1}^n
  |t_i(\rho, x_i ,s)|^2\\
  +\sumsum_{\substack{s_1,s_2\in A(k)\\s_1\not= s_2}}
  \sum_{(\rho,\uple{x})\in Y(k)}t_* (\rho,s_1) \ov{t_* (\rho,s_2)}
  \prod_{i=1}^n t_i(\rho,x_i,s_1)\ov{t_i(\rho, x_i,s_2)}.
\end{multline}

We detect the condition $(\rho,\uple{x} )\in Y(k)$ through additive
characters. Thus, let $\psi$ a non-trivial character of $k$. For
$\uple{x}=(x_i)_{i\in S}\in \mcC(k)$, we have
\begin{align*}
  \delta_{(\rho,\uple{x})\in
  Y(k)}=\prod_{j=1}^m\frac{1}{|k|}\sum_{\lambda_j\in
  k}\psi(\lambda_j\Sigma_j(\rho, \uple{x}))
  &=\frac{1}{|k|^m}
    \sum_{\uple{\lambda}\in k^m}
    \psi\Bigl(g_j(\rho) + \sum_{j=1}^m\sum_{i\in S} \lambda_j f_{i,j}(x_i)\Bigr)\\
  &=\frac{1}{|k|^m}\sum_{\uple{\lambda}\in k^m} \psi(g_{\uple{\lambda}}
    (\rho)) \prod_{i=1}^n\psi(f_{i,\uple{\lambda}}(x_i)),
\end{align*}
where $\uple{\lambda}=(\lambda_j)_{j\leq m}$, and
$$
  g_{\uple{\lambda}} (\rho) = \sum_{j=1}^m \lambda_j g_j (\rho),\quad\quad
  f_{i,\uple{\lambda}}(x_i)=\sum_{j=1}^m\lambda_j f_{i,j}(x_i).	
$$
Thus the second sum on the right-hand side of~(\ref{eq-develop}) is
equal to
\begin{multline*} 
  \frac{1}{|k|^m} \sumsum_{\substack{s_1,s_2\in A(k)\\s_1\not= s_2}}
  \sum_{\uple{\lambda}\in k^m}\sum_{(\rho, \uple{x})\in
    \mcC(k)}\psi(g_{\uple{\lambda}}(\rho)) t_* (\rho,s_1) \ov{t_*
    (\rho,s_2)} \prod_{i \in S} t_i(\rho,x_i
  ,s_1)\ov{t_i(\rho,x_i,s_2)}\psi(f_{i,\uple{\lambda}}(x_i))
  \\
  =\frac{1}{|k|^m} \sumsum_{\substack{s_1,s_2\in A(k)\\s_1\not=
      s_2}}\sum_{\uple{\lambda}\in k^m} \sum_{\rho \in B(k)}
  \psi(g_{\uple{\lambda}}(\rho)) t_* (\rho,s_1) \ov{t_* (\rho,s_2)}
  \\
  \times \prod_{i\in S} \Bigl( \sum_{x_i\in C_{i,\rho}(k)}
  t_i(\rho,x_i,s_1)\ov{t_i(\rho,x_i,s_2)}
  \psi(f_{i,\uple{\lambda}}(x_i))\Bigr).
\end{multline*}
For $s_1\not=s_2$, it follows from the twist-independence assumption
and the Riemann Hypothesis (Proposition~\ref{pr-recall-rh}
and~(\ref{eq-cond-twist})) that for each $i\in S$, we have
$$
\sum_{x_i \in C_{i,\rho}(k) }
t_i(\rho,x_i,s_1)\ov{t_i(\rho,x_i,s_2)}\psi(f_{i,\uple{\lambda}}(x_i))
\ll |k|^{1/2}
$$ 
and $t_* (\rho,s_1) \ov{t_* (\rho,s_2)} \ll 1$ for all $\rho\in
B(k)$. Hence the sum above is $\ll |k|^{\dim B + |S|/2+2}$, which
concludes the proof.
\end{proof}

\section{Parameterization of strata}


The goal of this section is to give a convenient parameterization of
the irreducible components of the strata of the stratification $(X_j)$
(Definition~\ref{def-xj}).

Let $j$ be an integer with $X_j$ non-empty.  Let
$X\subset X_j\subset \Aa^{2l}$ be an irreducible component of $X_{j}$
over $\Zz$ which intersects the characteristic zero part. Let
$\overline{\eta}$ be a geometric generic point of $X$.

We will show that $X$ is the projection of a space defined by
equations of a certain explicit type; more precisely, these will be
exactly of the type that can be handled using Lemma \ref{lm-dio-new},
allowing us to evaluate the sums that appear in Lemma \ref{lm-dio}. To
describe these equations and to perform an inductive process, where we
express better and better approximations of $X$ as the image of such
space, we need to package certain data, which we do using the
following definitions.


\begin{definition} A \emph{perspective datum} $\ppersp$ on $X$ is a
  tuple
$$
\ppersp=(m,S,B,(C_i), (b_i), (f_{i,j}), (g_j))
$$
where
\begin{itemize}
\item  $m\geq 0$ is an integer.
\item $ S \subseteq \{1,\dots 2l\}$.
\item $B$ is a separated scheme of finite type over $\Qbar$.
\item $(C_i)_{i\in S}$ is a family of relative curves over $B$.
\item $(b_i)_{i\in S}$ is a family of functions $b_i\colon B\to \Aa^1$
  if $i\notin S$ and $b_i\colon C_i\to \Aa^1$ if $i\in S$, such that
  if $i\in S$, the function $b_i$ is not constant on any irreducible
  component of any geometric fiber of $C_i\to B$.
\item $(f_{i,j})_{\substack{i\in S\\1\leq j \leq m}}$ is a family of
  functions $f_{i,j}\colon C_i \to \Aa^1$.
\item $(g_j)_{1\leq j\leq m}$ is a family of functions
  $g_j\colon B \to \Aa^1$.
\end{itemize}
\end{definition}

To simplify the notation, we will sometimes write $\ppersp\cdot m$,
..., $\ppersp\cdot (g_j)$ for the corresponding data.

Let $\ppersp$ be a perspective datum over $X$.  We denote
$\mathcal{C}_{\ppersp}$ the fiber product over $B$ of the curves $C_i$
for $i\in S$, and $\mathcal{Y}_{\ppersp}$ the subvariety of
$\mathcal{C}_{\ppersp}$ defined as the zero locus of the functions
$$
g_j+ \sum_{i \in S}f_{i,j}
$$
for $1\leq j\leq m$, where we extend the functions $f_{i,j}$ and the
functions $g_j$ by pullback to $\mathcal{C}_{\ppersp}$.

A \emph{perspective} over $X$ is a triple $(\ppersp,Y,\bar{\gamma})$
where
\begin{itemize}
\item $\ppersp$ is a perspective datum on $X$,
\item $Y$ is an irreducible component of $\mathcal{Y}_{\ppersp}$ 
\item $\bar{\gamma}$ is a geometric point of $Y$,
\end{itemize}
such that the morphism $g\colon \mathcal{Y}_{\ppersp}\to \Aa^{2l}$
defined by $(b_1,\ldots,b_{2l})$ induces a quasi-finite morphism
$$
Y-g^{-1}(\mcV^{\Delta})\to \Aa^{2l}-\mcV^{\Delta}
$$
which maps $\bar{\gamma}$ to $\bar{\eta}$.


The goal of this section will be to construct a perspective on $X$
where $Y$ is irreducible and the image of the map $Y \to \Aa^{2l}$ is
$X$. More precisely, the main result is the following:

\begin{theorem}\label{thm-ag}
  There exists a perspective $(\ppersp,Y,\bar{\gamma})$ on $X$ such
  that $Y$ is irreducible, $\bar{\gamma}$ is a geometric generic point
  of $Y$, and 
$$ 
2l- |\ppersp\cdot S| + 2\ppersp\cdot m\leq 4(2l-\dim(X)).
$$
\end{theorem}

The reader is encouraged to first finish reading the proof of the main
theorems of this paper, assuming that this statement holds, since this
will illustrate how the perspective data is exploited in the final
steps.

The basic strategy is the following:
\begin{enumerate}
\item We start with a perspective with $S$ as large as possible, $m$
  as small as possible, but $\bar{\gamma}$ potentially a quite special
  point of $Y$ (Lemma \ref{lm-perspective-example}). We plan to reduce
  $\dim Y$ while keeping the growth of $m$ and the loss of $|S|$
  controlled by a step-by-step induction.
\item At each step, we find some equations that are satisfied at
  $\bar{\gamma}$ but not at the generic point of $Y$ (Lemmas
  \ref{lm-finite-algebraic-part} and
  \ref{lm-infinite-algebraic-part}).
\item We construct a new perspective by adding these new equations
  (which may require also adjoining some new variables to $B$ and $C_i$), lowering $\dim
  Y$ (Lemma~\ref{lm-dimension-lowering}). However, the solution set in $Y$ of these new equations might
  not contain any irreducible components of the solution set in
  $\mathcal Y_\ppersp$ of the new equations, since they may instead be absorbed
  into other irreducible components of $\mathcal Y_\ppersp$. To deal with
  this, we must assume $\mathcal Y_\ppersp = Y$.
\item We can ensure that this condition holds by a Diophantine
  argument, which requires increasing $|S|$
  (Lemma~\ref{lm-component-cleaning}). This requires certain
  irreducibility assumptions on $B$ and on the curves $C_i$, which we
  ensure in Lemma \ref{lm-irreducibility-cleaning} by a direct
  construction.
\item Finally, we prove Theorem~\ref{thm-ag} by showing that an
  induction involving all these steps terminates in a suitable
  perspective.
\end{enumerate}

We begin by the exhibiting trivial examples of perspectives that will
be used to start the induction process (or to terminate it in a
trivial case).

\begin{lemma}\label{lm-perspective-example} 
  \emph{(1)} The tuple
$$
\ppersp_0=(0,\{1,\ldots,2l\}, \Spec(\Qbar),(\Aa^1)_{1\leq i\leq
  2l},(\mathrm{Id}_{\Aa^1})_{1\leq i\leq 2l}, \emptyset,\emptyset)
$$
is a perspective datum, and $(\ppersp_0, X, \bar{\eta})$ is a
perspective. 
\par
\emph{(2)} The tuple
$$
\ppersp_1=(0,\emptyset,X,\emptyset, (b_i|X),\emptyset, \emptyset)
$$
is a perspective datum and $(\ppersp_1,X,\bar{\eta})$ is a
perspective.
\end{lemma}

\begin{proof} 
  This is an elementary check. In (1), we have
  $\mathcal{Y}_{\ppersp_0}=\Aa^{2l}$, and the morphism
  $X \to \Aa^{2l}$ is quasi-finite, while in (2) we have
  $\mathcal{C}_{\ppersp_1}=X$, with the same conclusion.
\end{proof}

In the next three lemmas, we begin the proof of the second step by
studying how the roots of the polynomial $P_{\uple{b}}$, which are the
$r$-coordinates of the points in the fiber $Z_{\uple{b}}$, can change
under specialization.

Let $F$ be an algebraically closed field. Let $r_0,b_1,\dots,b_{2l}$
be elements of $F$.  Formally, the polynomial $P_{\uple{b}}\in F[r]$
is the product
$$
P_{\uple{b}}=\prod_{i=1}^{2l}(r+b_i)\ \prod_{(\zeta_i)\in\mmu_k^{2l}}
\Bigl(\sum_{i=1}^{2l} \zeta_i (r+b_i)^{1/k}\Bigr).
$$
This expansion makes sense unambiguously in an algebraic closure $K$
of the complete local field $F((r-r_0))$, provided we fix a choice of
$k$-th roots of $r+b_i$ in $K$. In particular, the order of vanishing
of $P_{\uple{b}}$ at $r_0$ is the sum of the valuation of the
factors, where the valuation on $F((r-r_0))$ is extended uniquely to
$K$.
\par
For $1\leq i\leq 2l$, fix $k$-th roots $(r_0+b_i)^{1/k}$ of $r_0+b_i$
in $F$ consistent with the choice of $(r+b_i)^{1/k}$ in $K$.  Then the
multiplicity of the factor
$$
\sum_{i=1}^{2l} \zeta_i (r+b_i)^{1/k}
$$
at $r_0$ is 
$$
\begin{cases}
0&\text{ if }\quad
\displaystyle{\sum_{i=1}^{2l} \zeta_i (r_0+b_i)^{1/k}} \neq 0,
\\
1/k&\text{ if }\quad
\displaystyle{\sum_{i=1}^{2l} \zeta_i (r_0+b_i)^{1/k}} =0
\text{ but }
\displaystyle{\sum_{\substack{ 1\leq i\leq 2l\\
    r_0+b_i=0}} \zeta_i} \neq 0,
\end{cases}
$$
and otherwise it is equal to the multiplicity of the formal power
series
$$
\sum_{\substack{1\leq i \leq 2l\\
    r_0+b_i \neq 0}} \zeta_i (r+b_i)^{1/k}\in F[[r]]\subset K
$$ 
at $r_0$, when one choses the branch of $(r+b_i)^{1/k}$ with constant
coefficient $(r_0+b_i)^{1/k}$.






\begin{lemma}\label{lm-finite-algebraic-part} 
  Let $R$ be a local integral domain with algebraically closed residue
  field $F$, and let $K$ be an algebraic closure of the fraction field
  of $R$. Let $b_1,\dots,b_{2l}$ be elements of $R$, and
  $\bar{\uple{b}}\in F^{2l}$ their reductions modulo the maximal
  ideal.  Let $r_0$ be some root of $P_{\bar{\uple{b}}}\in F[r]$.
  Assume that there exist at least two roots of $P_{\uple{b}}$ in $K$
  that reduce to $r_0$.  For $1\leq i\leq 2k$, fix a $k-$th root of
  $r_0 + b_i$ in $F$.

  Consider an algebraic closure $\widetilde{K}$ of $K((u))$.  For
  $\uple{\zeta}\in\mmu_k^{2l}$, let $n(\uple{\zeta})\geq 0$ be the
  multiplicity of
$$
\sum_{i=1}^{2l} \zeta_i (r + b_i)^{1/k}
$$
at $r_0$, as defined above.

There is no solution $(u_0,v_1,\ldots,v_{2l})\in R^{1+2l}$ of the
system of equations
\begin{align}
  v_i^k&=u_0+b_i\\
  u_0 + b_i
       &=0,\text{ for all $i$ such that $r_0+b_i=0$}
         \label{eq-m1}\\
  \sum_{i=1}^{2l} \zeta_i v_i       
       &=0, \text{ for all $\uple{\zeta}$ such
         that }
         \sum_{i=1}^{2l} \zeta_i (r _0 + b_i)^{1/k}=0\in F
         \label{eq-m2}\\
  \sum_{\substack{1 \leq i \leq 2l\\ r_0 + b_i \neq 0}} \zeta_i
  v_i^{1-kt}
       &=0,  \text{ if $n(\uple{\zeta})\geq 2$ and $0\leq 
         t\leq n(\uple{\zeta})-1$}.
\label{eq-m3}
\end{align}
\end{lemma}

\begin{proof} 
  Suppose that there exists a solution $u_0\in R$. We estimate from
  below the multiplicity of $u_0$ as a root of $P_{\uple{b}}$.  For
  each factor of $P_{\uple{b}}$, the valuation at $u_0$ is at least
  the valuation of the corresponding factor of $P_{\bar{\uple{b}}}$ at
  $r_0$, hence by summing, the order of vanishing of $P_{\uple{b}}$ at
  $u_0$ is at least the order of vanishing of $P_{\bar{\uple{b}}}$ at
  $r_0$. But this contradicts the assumption that there exist two
  roots of $P_{\uple{b}}$ reducing to $r_0$. 
\end{proof}

\begin{lemma} \label{lm-infinite-algebraic-part} Let $R$ be a local
  integral domain with algebraically closed residue field $F$
  containing a primitive $k$-th root of unity. Let $b_1,\dots,b_{2l}$
  be elements of $R$ and $\bar{\uple{b}}$ the reduction of $\uple{b}$
  modulo the maximal ideal.  Assume that
  $\deg(P_{\uple{\bar{b}}})<\deg(P_{\uple{b}})$.

  \emph{(1)} If $\uple{b}\notin \mcV^{\Delta}$, then for any
  $\uple{\zeta}=(\zeta_i)\in\mmu_k^{2l}$ there exists an integer
  $n_{\uple{\zeta}}\geq 0$ such that
$$
\sum_{i=1}^{2l}\zeta_i \bar{b}_i^{n_{\uple{\zeta}}}\not=0\in F.
$$
\par
\emph{(2)} There exists some $\uple{\zeta}=(\zeta_i)\in\mmu_k^{2l}$
and some integer $\nu$ with $0\leq \nu\leq n_{\uple{\zeta}}-1$ such
that
$$
\sum_{i=1}^{2l}\zeta_i b_i^n\not=0\in R.
$$
\end{lemma}

\begin{proof} 
Writing
$$
\sum_{i=1}^{2l} \zeta_i (r+b_i)^{1/k} = r^{1/k} \sum_{i=1}^{2l}
\zeta_i (1+b_i/r)^{1/k} = r^{1/k} \sum_{t=0}^\infty
\Bigl(\prod_{j=0}^{t-1} \frac{1/k-j} {1+j} \Bigr)
\Bigl(\sum_{i=1}^{2l}\zeta_ib_i^t\Bigr) \frac{1}{r^t}
$$
for $(\zeta_i)\in\mmu_k^{2l}$, we first see that if (1) fails, then
the left-hand side is identically $0$, which implies that
$\uple{b}\in\mcV^{\Delta}$.  Then we obtain
$$
\deg(P_{\uple{b}})=2l+k^{2l-1}
-\sum_{(\zeta_i)\in\mmu_k^{2l}}m_{\uple{\zeta}}
$$ 
where $m_{\uple{\zeta}}\geq 0$ is the largest integer such that
$$
\sum_{i=1}^{2l}\zeta_ib_i^t=0
$$
for $0\leq t\leq m_{\uple{\zeta}}$.
If condition (2) does not hold, we therefore deduce that
$\deg(P_{\uple{b}})\leq \deg(P_{\bar{\uple{b}}})$, which contradicts
the assumption.
\end{proof}

The next lemma is one of the key ingredients of the proof of
Theorem~\ref{thm-ag}.

\begin{lemma}\label{lm-dimension-lowering}
  Let $\ppersp$ be a perspective datum on $X$ and
  $(\ppersp,Y,\bar{\gamma})$ a perspective. If $\mathcal{Y}_{\ppersp}$
  is irreducible, so that $Y=\mathcal{Y}_{\ppersp}$, and
  $\bar{\gamma}$ is not a geometric generic point of
  $\mathcal{Y}_{\ppersp}$, then there exists a perspective
  $(\ppersp',Y',\bar{\gamma}')$ with
$$
\ppersp'\cdot S=\ppersp\cdot S,\quad
\dim(\ppersp'\cdot B)\leq \dim(\ppersp\cdot B)+1\quad
\dim(Y')<\dim(Y).
$$
\end{lemma}


\begin{proof}
  Let $\bar{\alpha}$ be a geometric generic point of $Y$, and
  $\bar{\beta}$ its image in $\Aa^{2l}$. By definition of a
  perspective, the fiber of $Y\to \Aa^{2l}$ over $\bar{\eta}$ is
  finite, and since it contains $\bar{\gamma}$, it cannot contain the
  point
  $\bar{\alpha}$ that
  specializes to $\bar{\gamma}$. Hence $\bar{\beta}\not=\bar{\eta}$,
  and since $\bar{\alpha}$ specializes to $\bar{\gamma}$, it follows
  that $\bar{\beta}$ specializes to $\bar{\eta}$. In particular, we
  deduce that $\bar{\beta}\notin\mcV^{\Delta}$.
\par
By definition, $\bar{\gamma}$ is a geometric generic point of
$X\subset X_j$. If $\bar{\beta}$ was a point of $X_j$, it would follow
that they are equal, which is not the case. Hence the fiber of
$f\colon Z\to \Aa^{2l}-\mcV^{\Delta}$ over $\bar{\beta}$ has $\geq
j+1$ points, whereas the fiber over $\bar{\eta}$ has $j$ points.
\par
Consider now the local ring $R$ of the closure of $\bar{\beta}$ at the
point $\bar{\eta}$. It has algebraically closed residue field. The
polynomial $P_{\bar{\beta}}\in R[r]$ has $\geq j+1$ roots, and the
specialization $P_{\bar{\eta}}$ has $j$ roots. So either there exist
two roots of $P_{\bar{\eta}}$ that have the same image in the residue
field, or $\deg(P_{\bar{\beta}})>\deg(P_{\bar{\eta}})$.
\par
\smallskip
\par
\textbf{Case 1} (two roots coincide).
\par
Let $r_0$ be the common reduction of at least two roots of
$P_{\bar{\beta}}$. We will apply Lemma~\ref{lm-finite-algebraic-part}
to $R$ and to this $r_0$. We define the multiplicity $n(\uple{\zeta})$
for $\uple{\zeta}\in\mmu_k^{2l}$ as in that lemma.

We consider the covering $\widetilde{B}\to B\times \Aa^1$, with
coordinate $u$ on $\Aa^1$, obtained by adjoining $k$-th roots $v_i$ of
$u+b_i$ for all $i\notin S$. We then define $B'$ as the complement in
$\widetilde{B}$ of the zero locus of $u+b_i$ for all $i\notin S$ such
that $r_0+b_i\not=0$. For $i\notin S$, the functions $b_i$ define
functions $B'\to \Aa^1$ by composing with the projection $B'\to B$.

For $i\in S$, we consider the curve $\widetilde{C}_i\to \widetilde{B}$
obtained from the base change of $C_i\times \Aa^1\to B\times \Aa^1$ to
$B'$ by adjoining a $k$-th root $v_i$ of $u+b_i$, so we have a diagram
$$
\begin{array}{ccccc}
  C_i & \longleftarrow & C_i\times_B \widetilde{B} 
  & \longleftarrow & \widetilde{C}_i\\
  \downarrow & & \downarrow \\
  B & \longleftarrow &  \widetilde{B}
\end{array}
$$
If $r_0+b_i\not=0$, we define $C'_i$ as the complement in $C_i$ of the
zero locus of $u+b_i$, and otherwise we define $C'_i=\widetilde{C}_i$.
In all cases, the morphism $C'_i\to C_i$ allows us to define a
function $b_i\colon C'_i\to \Aa^1$. The fibers of this function over a
geometric point of $B'$ project to geometric fibers of $C_i\to B$,
hence irreducible components project to irreducible components, and so
$b_i$ is not constant on any irreducible component of any geometric
fiber, since $\ppersp$ is a perspective datum.
\par
We next define the scheme $\mathcal{C}'\to B'$ as the fiber product
for $i\in S$ of the curves $C'_i$ over $B'$.
\par
There exists a lift $\bar{\gamma}'$ of $\bar{\gamma}$ in
$\mathcal{C}'$ such that $u(\bar{\gamma}')=r_0$ (indeed, we can lift
$\bar{\gamma}$ to the fiber product of the $\widetilde{C}_i$ over
$\widetilde{B}$, and the resulting point lies in $\mathcal{C}'$ since
$r_0+b_i=0$ if $u+b_i=0$). We fix such a lift. This choice defines
canonical $k$-th roots of
$u(\bar{\gamma}') + b_i(\bar{\gamma}')=r_0+b_i$, and we will use these
later.
\par
The functions $g_j$, $1\leq j\leq m$ and $f_{i,j}$ of the perspective
datum $\ppersp$ extend to $B'$ and $C'_i$, respectively, by composing
with the projections $B'\to B$ and $C'_i\to C_i$. We will now add
additional functions (corresponding to a change of the value of the
parameter $m$).
\par
Precisely, let $m'=m+m_1+m_2+m_3$, where $m_1$ (resp. $m_2$, $m_3$) is
the number of equations~(\ref{eq-m1}) in
Lemma~\ref{lm-finite-algebraic-part} (resp. number of
equations~(\ref{eq-m2}) or~(\ref{eq-m3})). We define the additional
functions $g_j$ and $f_{i,j}$ for $m+1\leq j\leq m'$, making a
one-to-one correspondance between the values of $j$ and the equations
of those three types.
\par
If $j$ corresponds to an equation~(\ref{eq-m1}), i.e., to an integer
$i$ with $1\leq i\leq 2l$ such that $r_0+b_i=0$, then we define
$$
\begin{cases}
  f_{i',j}=u+b_i&\textit{ for } i'\in S \text{ if }i'=i\\
  f_{i',j} = 0 & \textit{ for } i'\in S \text{ if } i'\not=i \\
  g_j=0,
\end{cases}
$$
if $i\in S$, and otherwise we define
$$
\begin{cases}
  f_{i',j}=0  & \textit{ for } i'\in S  \\
  g_j=u+b_i.
\end{cases}
$$
\par
If $j$ corresponds to an equation~(\ref{eq-m2}), i.e., to some
$\uple{\zeta}\in \mmu_k^{2l}$ such that
$$
\sum_{i=1}^{2l} \zeta_i (r _0 + b_i)^{1/k}=0
$$
we define
$$
\begin{cases}
  f_{i,j}=\zeta_iv_i&\textit{ for } i\in S\\
  g_j=\sum_{i\notin S}\zeta_i v_i.
\end{cases}
$$
\par
Finally, if $j$ corresponds to an equation~(\ref{eq-m3}), i.e., to
$\uple{\zeta}\in \mmu_k^{2l}$ and $t$ such that
$n(\uple{\zeta})\geq 2$ and $0\leq t\leq n(\uple{\zeta})-1$, then we
define
$$
\begin{cases}
  f_{i,j}=\zeta_iv_i^{1-kt}&\text{ if } i\in S\text{ and } r_0+b_i\not=0\\
  g_j=\sum_{\substack{i\notin S\\r_0+b_i\not=0}}\zeta_i
  v_i^{1-kt}.
\end{cases}
$$
(note that by the definition of $C'_i$, the function $v_i$ is
non-vanishing). We now have defined the perspective datum
$$
\ppersp'=(m',S, B', (C'_i)_{i\in S}, (b_i), (f_{i,j})_{\substack{i\in
    S\\1\leq j\leq m'}}, (g_j)_{1\leq j\leq m'}).
$$
\par
The associated variety, i.e., the vanishing locus $\mathcal{Y}'$ of
$$
g_j+\sum_{i\in S} f_{i,j}
$$
for $1\leq j\leq m'$, contains $\bar{\gamma}'$ by construction (see
Lemma~\ref{lm-finite-algebraic-part} again). Let $Y'$ be an
irreducible component of $\mathcal{Y}'$ containing $\bar{\gamma}'$. We
claim that $(\ppersp',Y',\bar{\gamma}')$ is the required perspective.
\par
First, for $y\in\mathcal{Y}$, the points of the fiber of
$\mathcal{Y}'\to \mathcal{Y}$ over $y$ are determined by the value of
the function $u$ on $\mathcal{Y}'$, whose values lie in the set of
roots of the polynomial $P_{\uple{b}(y)}$. In particular, the fiber is
finite, and hence $\mathcal{Y}'$ is quasi-finite over
$\mathcal{Y}$. It follows on the one hand that $Y'$ has dimension
$\leq \dim(Y)$, and on the other hand that $Y'$ is quasi-finite over
$\Aa^{2l}-\mcV^{\Delta}$. So $(\ppersp',Y',\bar{\gamma}')$ is a
perspective. 
\par
We have $\dim(B')\leq \dim(B)+1$. It remains therefore to check that
$\dim(Y')<\dim(Y)$. We have already observed that
$\dim(Y')\leq \dim(Y)$. Suppose the dimensions were equal. Then, since
$Y'\to Y$ is quasi-finite, the geometric generic point $\bar{\gamma}'$
would map to $\bar{\alpha}$ in $Y$, and therefore to $\bar{\beta}$ in
$\Aa^{2l}$. By applying finally Lemma~\ref{lm-finite-algebraic-part},
we obtain a contradiction: since two roots of $P_{\bar{\beta}}$ reduce
to the same root of $P_{\bar{\eta}}$, there cannot be solutions in $R$
of the system of
equations~(\ref{eq-m1}),~(\ref{eq-m2}),~(\ref{eq-m3}), whereas this is
exactly what we obtain from the fact that $\bar{\beta}$ is the image
of $\bar{\gamma}'$.
\par
\smallskip
\par
\textbf{Case 2} (the degree drops).
\par
We now consider instead Lemma~\ref{lm-infinite-algebraic-part}, and
define integers $n_{\uple{\zeta}}$ for $\zeta\in\mmu_k^{2l}$ as the
least integer $\geq 0$ such that
$$
\sum_{i=1}^{2l}\zeta_i \bar{b}_i^{n_{\uple{\zeta}}}\not=0
$$
at $\bar{\eta}$ (this exists by statement (1) in the lemma). We define
$m'=m+m_1$, where $m_1$ is the number of pairs $(\uple{\zeta},\nu)$
with $\uple{\zeta}\in\mmu_k^{2l}$ and
$0\leq \nu\leq n_{\uple{\zeta}}$. For $m+1\leq j\leq m'$,
corresponding in one-to-one fashion to $(\uple{\zeta},\nu)$, we define
$$
\begin{cases}
  f_{i,j}=\zeta_i b_i^{\nu}&\text{ for } i\in S\\
  g_j=\sum_{i\notin S}\zeta_i b_i^{\nu}.&
\end{cases}
$$
Then
$\ppersp'=(m',S,B,(C_i),(b_i),(f_{i,j})_{\substack{i\in S\\1\leq j\leq
    m'}},(g_j)_{1\leq j\leq m'})$ is a perspective datum (since the
$b_i$ have not changed, the non-constancy condition is also
unchanged). The point $\bar{\gamma}$ belongs to the associated variety
$\mathcal{Y}'\subset \mathcal{Y}_{\ppersp}\subset
\mathcal{C}_{\ppersp}$ (by definition of $n_{\uple{\zeta}}$), so
$(\ppersp',Y',\bar{\gamma})$ is a perspective, where $Y'$ is the
irreducible component of $\mathcal{Y}'$ containing $\bar{\gamma}$. By
Lemma~\ref{lm-infinite-algebraic-part}, on the other hand,
$\bar{\alpha}$ does not lie in $\mathcal{Y}'$, so all its irreducible
components, including $Y'$, have dimension
$<\dim(\mathcal{Y}_{\ppersp})=\dim(Y)$.
\end{proof}

In the next lemma, we produce from a a perspective another one with a
specific value of the parameter $m$.

\begin{lemma}\label{lm-equation-cleaning}
  Let $(\ppersp,Y,\bar{\gamma})$ be a perspective on $X$. There exists
  a perspective $(\ppersp',Y',\bar{\gamma})$ such that
\begin{gather*}
  \ppersp'\cdot S=\ppersp\cdot S,\quad\quad \ppersp'\cdot
  B=\ppersp\cdot B,\quad\quad \ppersp'\cdot (C_i)=\ppersp\cdot (C_i)
  \quad\quad \ppersp'\cdot (b_i)=\ppersp\cdot (b_i)
  \\
  \ppersp'\cdot m=\dim(\ppersp\cdot B)+|\ppersp\cdot S|-\dim(Y)
  \\
  \text{$Y'$ is isomorphic to $Y$},\quad\quad
  \text{$\mathcal{Y}_{\ppersp}\subset \mathcal{Y}_{\ppersp}'$ as
    $B$-schemes.}
\end{gather*}
\end{lemma}


\begin{proof} 
  Let $m'=\dim(\ppersp\cdot B)+|\ppersp\cdot S|-\dim(Y)$. It is the
  codimension of $Y$ in $\mathcal{C}_{\ppersp}$. Let $X$ be the
  subspace of $\Gamma(\mathcal{C}_{\ppersp},\mathcal{O})$ generated by
  the functions
$$
h_j=g_j+\sum_{1\leq j\leq m}f_{i,j}
$$
for $1\leq j\leq m$. We claim that for any integer $\nu$ with
$0\leq \nu\leq m'$, there exist $(\varphi_1,\ldots,\varphi_{\nu})$ in
$X$ such that all irreducible components of the zero locus
$V(\varphi_1,\ldots,\varphi_{\nu})$ in $\mathcal{C}_{\ppersp}$ that
contain $Y$ have codimension $\nu$ in $\mathcal{C}_{\ppersp}$.
\par
We prove this by induction on $\nu$. The statement is true for
$\nu=0$. Assume that $\nu\leq m'$ and that the property holds for
$\nu-1$ and the functions $(\varphi_1,\ldots,\varphi_{\nu-1})$. Let
$W$ be an irreducible component of the zero locus
$V(\varphi_1,\ldots,\varphi_{\nu-1})$. It has codimension
$\nu-1<m'=\codim(Y)$ in $\mathcal{C}_{\ppersp}$ so $Y$ is a proper
closed irreducible subset of $W$. Hence there exists $j$ such that
$h_j$ does not vanish identically on $W$, and in particular the set of
$\varphi\in X$ such that $\varphi$ does not vanish on $W$ is a
non-empty Zariski-open subset of $X$. Taking intersection of these
open sets, there exists $\varphi_{\nu}\in X$ such that $\varphi_{\nu}$
is non-vanishing on all irreducible components $W$ containing $Y$. It
follows that $(\varphi_1,\ldots, \varphi_{\nu})$ satisfy the induction
assumption.
\par
For $\nu=m'$, this means that all irreducible components of
$V(\varphi_1,\ldots,\varphi_{m'})$ containing $Y$ have codimension
$m'=\codim(Y)$ in $\mathcal{C}_{\ppersp}$. Hence $Y$ is one of the
irreducible components of $V(\varphi_1,\ldots,\varphi_{m'})$.
\par
For $1\leq \nu\leq m'$, write
$$
\varphi_{\nu}=\sum_{1\leq j\leq m}\alpha_{\nu,j}h_j.
$$
We define
$$
g'_{\nu}=\sum_{1\leq j\leq m}\alpha_{\nu,j}g_j, \quad\quad
f'_{i,\nu}=\sum_{1\leq j\leq m}\alpha_{\nu,j}\sum_{i\in S}f_{i,j},
$$
for $i\in S$ and $1\leq \nu\leq m'$ so that
$$
g'_{\nu}+\sum_{i\in S}f'_{i,\nu}=\varphi_{\nu}.
$$
Then 
$$
\ppersp'=(m',S,B,(C_i)_{i\in S},(b_i),(f'_{i,j})_{\substack{i\in
    S\\1\leq j\leq m'}}, (g'_j)_{1\leq j\leq m'})
$$
is a perspective datum on $X$; by construction $Y$ is an irreducible
component of $\mathcal{Y}_{\ppersp'}$ and
$\mathcal{Y}_{\ppersp}\subset \mathcal{Y}_{\ppersp}'$ as $B$-schemes,
so $(\ppersp',Y,\bar{\gamma})$ is a perspective with the desired
properties.
\end{proof}

In the next lemma, we have a single perspective, so we don't use the
selector notation. 

\begin{lemma}\label{lm-dio-3} 
  Let $(\ppersp,Y,\bar{\gamma})$ be a perspective on $X$. For any
  $T\subset S$ and $b \in B$ , we put
$$
\widetilde{\Gamma}_{T,b }=\prod_{i\in
  T}\Gamma(C_{i,b},\mathcal{O}_{C_{i,b}}),\quad\quad \Gamma_{T,b}=\prod_{i\in
  T}(\Gamma(C_{i,b},\mathcal{O}_{C_{i,b}})/\kappa_b),
$$
where the $\kappa_b$ is the residue field at $b$. The spaces $\widetilde{\Gamma}_{T,b}$ and $\Gamma_{T,b}$
are $\kappa_b$-vector spaces.  For $1\leq j\leq m$, we denote
$\uple{f}_{T,j,b}=(f_{i,j})_{i\in T}\in\widetilde{\Gamma}_{T,b}$.
  \par
  Assume that $S$ is not empty, that $B$ is irreducible, and that the generic fiber of $C_i\to B$ is
  geometrically irreducible for all $i\in S$.
\par
One of the following properties holds:
\par
\emph{(a)} The scheme $\mathcal{Y}_{\ppersp}$ has a unique
geometrically irreducible component whose projection to $B$ is
dominant.
\par
\emph{(b)} There exists a proper subset $T\subset S$ such that the images of
$(\uple{f}_{T,1,\eta},\ldots,\uple{f}_{T,m,\eta})$ 
span a subspace of $\Gamma_{T,\eta}$
of dimension $\leq m-(|S| - |T|)/2$, for $\eta$ the generic point of
$B$.
\end{lemma}

\begin{proof}
  There exists a number field and an open dense subset $\Oc$ of its
  ring of integers in a number field such that the persective datum is
  defined over $\Oc$. We fix one model of $\ppersp$ over $\Oc$, and
  we will use the same notation for its components as for the
  original objects over $\Qbar$.
  We assume that property (b) does \emph{not} hold and we will show
  that (a) holds. We will do this by studying fibers of
  $\mathcal{Y}_{\ppersp}\to B$ over finite-valued field points of a
  suitable dense open subset of $B$, using the point-counting
  criterion for irreducibility over finite fields.


For $b \in B$, the condition that the all curves $C_{i,b}$ are geometrically irreducible is a constructible condition.  So is the condition $(\uple{f}_{T,1,b},\ldots,\uple{f}_{T,m,b})$ generate a
  subspace of $\Gamma_{T,b}$ of dimension $> m - (|S|- |T|)/2$ for all proper subsets $T$ of $S$.
  
  By assumption, including the negation of (b), these properties both hold at the generic point, hence we can find a dense open subset $B^\circ$ where both properties hold.

  Let $\Spec(\kappa)\to \Spec(\Oc)$ be a finite-field valued point of
  $\Spec(\Oc)$.  Fix $b\in B^{\circ}(\kappa)$.  Let $\psi$ be a fixed
  non-trivial additive character of $\kappa$.  We denote
  $V=\mathcal{Y}_{\ppersp,b,\kappa}$. We compute $|V(\kappa)|$ using
  additive characters (as in the proof of Proposition
  \ref{lm-dio-new}). For $\uple{\lambda}\in\kappa^m$ and
  $x\in \mathcal{C}_{\ppersp}(\kappa)$, we denote
$$
f_{\uple{\lambda}}(x)=\sum_{j=1}^m\lambda_j\sum_{i\in S}f_{i,j}(x).
$$
and
$$
\xi(\uple{\lambda})=\psi\Bigl(\sum_{j=1}^m\lambda_jg_j(b)\Bigr).
$$
We have
\begin{align*}
  |V(\kappa)|&=
               \frac{1}{|\kappa|^m}
               \sum_{x\in\mathcal{C}_{\ppersp,b}(\kappa)}
               \prod_{j=1}^m\sum_{\lambda\in\kappa}\psi
               \Bigl(
               \lambda\Bigl(g_j(x)+\sum_{i\in S}f_{i,j}(x)\Bigr)
               \Bigr)\\
             &=
               \frac{1}{|\kappa|^m}
               \sum_{x\in\mathcal{C}_{\ppersp,b}(\kappa)}
               \prod_{j=1}^m\sum_{\lambda\in\kappa
               }\psi(\lambda g_j(b))\psi
               \Bigl(
               \lambda\sum_{i\in S}f_{i,j}(x)
               \Bigr)
  \\
             &
               =\frac{1}{|\kappa|^m} 
               \sum_{\uple{\lambda}\in
               \kappa^m}\xi(\uple{\lambda})E(b;\uple{\lambda}),
\end{align*}
where
$$
E(b;\uple{\lambda})=\sum_{x\in\mathcal{C}_{\ppersp,b}(\kappa)}\psi(f_{\uple{\lambda}}(x)).
$$
By definition of
$\mathcal{C}_{\ppersp}$ as a fiber product, we have the separation of
variable formula
$$
E(b;\uple{\lambda})= \prod_{i\in S} \sum_{x\in
  C_{i,b}(\kappa)}\psi\Bigl(\sum_{j=1}^m\lambda_jf_{i,j}(x)\Bigr).
$$
Let
$$
S_{\uple{\lambda}}=\Bigl\{i\in S\,\mid\, \sum_{j=1}^m
\lambda_jf_{i,j}\text{ is constant on } C_{i,b}\Bigr\}\subset S.
$$
Applying the Weil bound for the exponential sums over
$C_{i,b}(\kappa)$ (assuming the characteristic is larger than the degree of the functions $f_{i,j}$), it follows that
$$
E(b;\uple{\lambda})\ll
|\kappa|^{|S_{\uple{\lambda}}|+(|S|-|S_{\uple{\lambda}}|)/2}=
|\kappa|^{(|S|+|S_{\uple{\lambda}}|)/2}.
$$
\par
We now split the expression for
$|V(\kappa)|$ above according to the value of
$S_{\uple{\lambda}}$, and isolate the term corresponding to
$S_{\uple{\lambda}}=S$ from the others. This gives
$|V(\kappa)|=N_1+N_2$, where
$$
N_1=\frac{1}{|\kappa|^m} \sum_{\substack{\uple{\lambda}\in
    \kappa^m\\S_{\uple{\lambda}}=S}}\xi(\uple{\lambda})E(b;\uple{\lambda}),
\quad\quad 
N_2= \frac{1}{|\kappa|^m} \sum_{\substack{\uple{\lambda}\in
    \kappa^m\\S_{\uple{\lambda}}\not=S}}\xi(\uple{\lambda})E(b;\uple{\lambda}).
$$

Taking $T = S - \{i\}$ for a fixed $i \in S$ in the defining property
of $B^\circ$, we observe that the tuple
$(\uple{f}_{T,1,b},\ldots,\uple{f}_{T,m,b})$ generates a subspace of
$\Gamma_{T,b}$ of dimension $> m - (|S|- |T|)/2 > m-1/2$, hence are
linearly independent in $\Gamma_{T,b}$, and thus are linearly
independent in $\Gamma_{S,b}$.
The condition $S_{\uple{\lambda}}=S$ arises then only when
$\uple{\lambda}=0$.
Hence
$$
N_1=\frac{1}{|\kappa|^m} \prod_{i\in S}|C_{i,b}(\kappa)|.
$$
Since
$C_{i,b}$ is a geometrically irreducible curve (by the choice of
$B^{\circ})$, we have
$|C_{i,b}(\kappa)|=|\kappa|+O(|\kappa|^{1/2})$ for all $i$. Hence
\begin{align*}
  N_1&= |\kappa|^{|S|-m}(1+O(|\kappa|^{-1/2}))^{|S|}+ O(|\kappa|^{-m+|S|-1/2})\\
     &=|\kappa|^{|S|-m}+O(|\kappa|^{|S|-m-1/2}).
\end{align*}
\par
On the other hand, we have 
$$
N_2\ll \frac{1}{|\kappa|^m}\sum_{\substack{T\subset S\\T\not=S}}
|\kappa|^{n(T)}|\kappa|^{(|S|+|T|)/2}
$$
where $n(T)$ is the dimension of the
$\kappa$-vector subspace of
$\kappa^m$ whose elements are all
$\uple{\lambda}$ such that $S_{\uple{\lambda}}\subset
T$. We have $n(T)=\ker(\varphi_T)$, where $\varphi_T\colon \kappa^m\to
\Gamma_{T,b,\kappa}/\kappa$ is the linear map
$$
\uple{\lambda}\mapsto \sum_{j=1}^m\lambda_j\uple{f}_{T,j}\mods \kappa.
$$
Since $T$ is a proper subset of $S$, by the definition of $B^\circ$,
we must have $\dim
\Imag(\varphi_T)>m-\frac{|S|-|T|}{2}$, so that
$n(T)<(|S|-|T|)/2$, which implies $n(T)\leq
(|S|-|T|)/2-1/2$, so we derive
$$
N_2\ll |\kappa|^{-m+(|S|-|T|)/2+(|S|+|T|)/2-1/2}=|\kappa|^{|S|-m-1/2}.
$$
We conclude that
$$
|V(\kappa)|=|\kappa|^{|S|-m}+O(|\kappa|^{|S|-m-1/2}).
$$ 
Applying this to finite extensions of $\kappa$ and applying the
Lang-Weil estimates, we conclude that $V$ is geometrically
irreducible.

Recalling that $V$ was the fiber of $\mathcal{Y}_{\ppersp}$ over an
arbitrary point $b \in B^\circ(\kappa)$, we see that all the fibers of
$\mathcal{Y}_{\ppersp}$ over finite-field valued points of $B^\circ$
with sufficiently large characteristic are geometrically irreducible,
so all the fibers of $\mathcal{Y}_{\ppersp}$ over points of $B^\circ$
are geometrically irreducible. Therefore $\mathcal{Y}_{\ppersp}$ has a
unique geometrically irreducible component that is dominant over $B$,
concluding the proof that condition (a) holds.
\end{proof}

\begin{lemma}\label{lm-component-cleaning}
  Let $(\ppersp,Y,\bar{\gamma})$ be a perspective on $X$ defined over
  an open subscheme $\Spec(\Oc)$ of the ring of integers in a number
  field.  Assume that $S$ is not empty, that $B$ is geometrically
  irreducible, that each $C_i$ is irreducible and that the generic
  fiber of $C_i\to B$ is geometrically irreducible for all $i\in S$.
\par
If $\mathcal{Y}_{\ppersp}$ is reducible and all irreducible components
of $\mathcal{Y}_{\ppersp}$ are dominant over $B$, then there exists a
perspective $(\ppersp',Y',\bar{\gamma})$ on $X$ such that
$\dim Y' = \dim Y$ and
$$
1\leq |\ppersp\cdot S|-|\ppersp'\cdot S| \leq 2(\ppersp\cdot
m-\ppersp'\cdot m).
$$
\end{lemma}


\begin{proof} 
  We apply Lemma \ref{lm-dio-3} to $(\ppersp,Y,\bar{\gamma})$, and use
  the same notation. Since $\mathcal{Y}_{\ppersp}$ is reducible
  and all its irreducible components are dominant over $B$, there are
  at least two irreducible components that are dominant over $B$.
  By Lemma~\ref{lm-dio-3}, we conclude that 
  there exists a proper subset $T\subset \ppersp\cdot S$ such that the
  span of $(\uple{f}_{T,1,\eta},\ldots,\uple{f}_{T,m,\eta})$ in $\Gamma_T$ has
  dimension $\leq m-(|S| - |T|)/2$. 
  
  For $\uple{\lambda} \in \ker(\varphi_T)$ and $i \in T$, $\sum_{j=1}^m\lambda_j f_{i,j}$ is equal to an element of $\kappa_\eta$ and hence a rational function on $B$. Let $B^*$ be an open subset of $B$ on which all these functions are defined.  Because $C_i$ is irreducible, $\sum_{j=1}^m\lambda_j f_{i,j}$  is equal to this function on $B^*$ not just at the generic point, but everywhere.

  Let $m'$ be the dimension of the span $X$ of
  $(\uple{f}_{T,1},\ldots,\uple{f}_{T,m})$ in $\Gamma_T$. We have then
$$
1\leq |\ppersp\cdot S|-|T|\leq 2(\ppersp\cdot m-m').
$$
\par
Let $\widetilde{\mathcal{C}}$ be the fibre product of $C_i$ for
$i\in S-T$ with $B^*$ over $B$. We have an evaluation map
$$
\varphi_T\colon \Aa^{m}\to \Gamma_T
$$
sending $(\lambda_i)_{i\in T}$ to 
$$
\sum_{j=1}^m\lambda_j\uple{f}_{T,j}.
$$
We define $B'\subset \widetilde{\mathcal{C}}$ to be the common zero locus
of the functions
$$
\sum_{j=1}^m\lambda_j\Bigl(g_j+\sum_{i\in S}f_{i,j}\Bigr)
$$
for all $\uple{\lambda}$ in $\ker(\varphi_T)$. These
expressions are indeed well-defined functions on
$\widetilde{\mathcal{C}}$ because, as we saw earlier
$$
\sum_{j=1}^m\lambda_j f_{i,j}
$$
is equal to a function on $B^*$ for $i\in T$ if $\uple{\lambda}\in\ker(\varphi_T)$.
\par
Furthermore, we choose $f'_{i,j}$ in $X$ for $i\in T$ and
$1\leq j\leq m'$ so that $f'_{i,j} = \sum_{\nu=1}^m \beta_{j,\nu} \uple{f}_{i,\nu}$ for $(\beta_{j,\nu})_{1\leq j \leq m'}$  a set of elements of $\Aa^m$ that span its image $X$ under $\varphi_T$. Define
$$
g'_j=\sum_{\nu=1}^m \beta_{\nu,j}\Bigl(g_\nu+\sum_{i\in S-T}f_{i,\nu}\Bigr).
$$
Then the tuple
$$
\ppersp'=(m',T, B', (C_i\times_B B')_{i\in T}, (b'_i)_{i\in T},
(f'_{i,j}), (g'_j))
$$
is a perspective datum on $X$, where $b'_i$ is the extension of $b_i$
to $C'_i=C_i\times_B B'$ by pullback for $i\in T$, the composition
$B'\to \tilde{C}\to C_i=\Aa^1$ if $i\in S-T$, and the projection
$B'\to B\to \Aa^1$ otherwise.
\par
By construction, the fiber product $\mathcal{C}_{\ppersp'}$ is a locally closed subset of
contained in $\mathcal{C}_{\ppersp}$.  The subscheme
$\mathcal{Y}_{\ppersp'}$ is an open subset of $\mathcal{Y}_{\ppersp}$, because it has the same set of defining equations after restricting to an open subset $B^*$ of $B$. Because the irreducible component $Y$ was dominant over $B$, its restriction to this open subset has the same dimension, and because $\bar{\gamma}$ was dominant over $B$, it remains in this open subset as well. Hence
$(\ppersp',Y',\bar{\gamma})$ is the desired perspective on $X$.






\end{proof}

The last preparatory lemma constructs a perspective where the base $B$
satisfies the assumptions of the last lemma.

\begin{lemma}\label{lm-irreducibility-cleaning}
Let $(\ppersp,Y,\bar{\gamma})$ be a perspective on $X$. Then there
exists a perspective $(\ppersp', Y',\bar{\gamma}')$ such that 
$$
\ppersp'\cdot m=\ppersp\cdot m,\quad\quad
\ppersp'\cdot S=\ppersp\cdot S,
$$
and $\dim(Y')\leq \dim(Y)$, and moreover
\par
\emph{(a)} $\ppersp'\cdot B$ is irreducible.
\par
\emph{(b)} For all $i\in S$, the curve $\ppersp'\cdot C_i$ are
irreducible and the fiber of $\ppersp'\cdot C_i$ over the geometric
generic point of $\ppersp'\cdot B$ is irreducible.
\par
\emph{(c)} All irreducible components of $\mathcal{Y}_{\ppersp'}$ are
dominant over $B$, as is $\bar{\gamma}$.
\end{lemma}


The strategy of the proof is to make several modifications to the
given perspective datum to ensure that these three conditions hold. We
will first replace $B$ by an irreducible scheme, ensuring condition
(a). We then pass to a finite cover of $B$ over which generic
geometrically irreducible components of $C_i$ are defined and choose
one for each $i$, ensuring condition (b). Finally we remove a closed
subset from $B$, containing all the irreducible components that are
not dominant over $B$, ensuring condition (c).

\begin{proof}
  Let $\mathcal{A}\subset \mathcal{C}_{\ppersp}$ be an irreducible
  component containing $Y$. Let $B_0$ be the schematic closure of the
  image of $\bar{\gamma}$ under the projection $Y \to B$. It is
  closed and irreducible. Let $\beta$ be its generic point. Let
  $\beta'\to \beta$ be a finite extension such that all irreducible
  components of the generic fibers of the curves $C_i$ for $i\in S$
  are defined over $\beta'$. Let then $B' \to B_0$ be a finite flat
  morphism whose generic fiber is $\beta'\to \beta$ (we can construct
  such a morphism by taking a generator of the field extension
  $\beta'/\beta$, and multiplying it by a regular function on $B_0$ so
  that its minimal polynomial $P$ becomes monic; then the cover $B'$
  of $B_0$ obtained by adjoining a root of $P$ has the required
  property).

  Fix a lift $\bar{\gamma}'$ of $\bar{\gamma}$ to $Y \times_B B'$. Let
  $Y'$ be an irreducible component of $Y \times_B B'$ containing
  $\bar{\gamma}'$, and let $\mathcal{A}'$ be an irreducible componet
  of $\mathcal{C}_{\ppersp} \times_B B'$ containing $Y'$. Because
  $\bar{\gamma}$ maps to the generic point $\beta$ of $B_0$,
  $\bar{\gamma}'$ must map to the generic point $\beta'$ of $B'$ (the
  only point lying in the fiber), and so $Y' \to B'$ and
  $\mathcal{A}' \to B'$ are dominant maps. Because $\mathcal{A}'$ is
  an irreducible component of $\mathcal{C}_{\ppersp} \times _B B'$,
  and maps dominantly to $B'$, it follows that $\mathcal{A}'_{\beta'}$
  is an irreducible component of the pullback $\mathcal C_{\beta'}$ of
  the product of the curves $C_i$ to $\beta'$. Hence there are
  irreducible components $\widetilde{C}_{i,\beta'}$ of $C_{i,\beta'}$
  for $i\in S$, such that $\mathcal{A}'_{\beta'}$ is contained in the
  product of the $\widetilde{C}_{i,\beta'}$. Let $C'_i$ be the closure
  of $\widetilde{C}_{i,\beta'}$. This is an irreducible curve over
  $B'$.  We can pullback the functions $b_i$, $g_j$ and $f_{i,j}$ to
  $B'$ and $C'_i$, respectively. We have then constructed a
  perspective datum
$$
\ppersp'=(m,S,B', (C'_i),(b'_i),(f'_{i,j}), (g'_j)).
$$

The irreducible component $Y'$ is contained in
$\mathcal{C}_{\ppersp'}$, hence in $\mathcal{Y}_{\ppersp'}$. Since the
morphism $\mathcal{Y}_{\ppersp'}\to \mathcal{Y}_{\ppersp}$ is finite,
it is an irreducible component of $\mathcal{Y}_{\ppersp'}$. It
contains $\bar{\gamma}$' and so maps dominantly onto $B'$.
  
Let $B''$ be the complement in $B'$ of the closure of the images of
all irreducible components of $\mathcal{Y}_{\ppersp'}$ that are not
dominant over $B'$. We can pullback the data
$C'_i, b'_i, g'_j, f_{i,j}, Y''$ further to $B''$. This defines a
perspective datum
  $$ 
\ppersp''=(m,S,B'', (C''_i),(b''_i),(f''_{i,j}), (g''_j)).
$$
and a perspective $(\ppersp'', Y'', \bar{\gamma}')$.
  
By construction, $B'$ and $B''$ are geometrically irreducible. Since
the curves $C''_i$ are generically irreducible, and their geometric
generic fibers are defined over $\beta'$, they are generically
geometrically irreducible. Because $\bar{\gamma}$ maps dominantly to
$B$, $\bar{\gamma}'$ maps dominantly to $B'$ and $B''$. Finally, all
irreducible components of $\mathcal Y_{\ppersp''}$ map dominantly to
$B''$ by construction.
\end{proof}

We can now conclude this section.

\begin{proof}[Proof of Theorem~\ref{thm-ag}]
  Consider the set $\mathcal{P}$ of perspectives
  $(\ppersp,Y,\bar{\gamma})$ on $X$
  such that 
\begin{equation}\label{eq.1}
2 \dim(\ppersp\cdot B) + 2 \dim(Y) +|\ppersp\cdot S| \leq 6 l.
\end{equation}
This set is nonempty by Lemma~\ref{lm-perspective-example} (1), hence
it contains some element where
$$
\dim(Y)+|\ppersp\cdot S|
$$
is minimal.

Using Lemma~\ref{lm-irreducibility-cleaning}, we obtain a perspective
$(\ppersp,Y,\bar{\gamma})\in\mathcal{P}$ such that $\ppersp\cdot B$ is
geometrically irreducible, the curves $\ppersp\cdot C_i$ are
irreducible, the geometric generic fibers of $\ppersp\cdot C_i$ are
irreducible, and all irreducible components of $\mathcal{Y}_{\ppersp}$
as well as $\bar{\gamma}$ are dominant over $B$. By Lemma
\ref{lm-equation-cleaning}, we may assume that
$$
\ppersp\cdot m = \dim(\ppersp\cdot B) + |\ppersp\cdot S| - \dim(Y)
$$
(note that the last condition in Lemma~\ref{lm-equation-cleaning}
implies that all irreducible components of $\mathcal{Y}_{\ppersp'}$ as
well as $\bar{\gamma}'$ are dominant over $B$ for the new perspective
given by that lemma with input $(\ppersp,Y,\bar{\gamma})$.)

We will then see that, except in a trivial case, a perspective with
these properties satisfies the desired conclusion that
$\ppersp \cdot Y$ is irreducible, $\bar{\gamma}$ is the generic point
of $Y$, and
$$
2l- |\ppersp\cdot S| + 2\ppersp\cdot m\leq 4(2l-\dim(X)).
$$

First, if $Y$ is irreducible and $\bar{\gamma}$ is the generic point
of $Y$, then because $Y$ is quasi-finite over $\Aa^{2l}$, we have
$\dim(Y)=\dim(X)$, hence
\begin{align*}
  2l- |\ppersp\cdot S| + 2\ppersp\cdot m= 2\dim(\ppersp\cdot B) + 2l +
  |\ppersp\cdot S| - 2\dim(Y)&\leq 8l - 4 \dim(Y)\\
  =4(2l-\dim(X)).
\end{align*}

Next assume that $Y$ is irreducible and $\bar{\gamma}$ is not the
generic point of $Y$. Then Lemma \ref {lm-dimension-lowering}
provides a perspective $(\ppersp',Y',\bar{\gamma}')$ with
$$
|\ppersp'\cdot S|=|\ppersp\cdot S|,\quad\quad |\ppersp'\cdot B|\leq
|\ppersp\cdot B|+1,\quad\quad \dim(Y')<\dim(Y)
$$
so
$$ 
2\dim(\ppersp'\cdot B) + \dim (Y') + |\ppersp'\cdot S| \leq 6l
$$
but satisfying
$$
\dim(Y')+|\ppersp'\cdot S|<\dim(Y)+ |\ppersp\cdot S|,
$$
which contradicts the minimality of $\ppersp$.
 
Suppose now that $Y$ is reducible and $\ppersp\cdot S$ is
nonempty. Then Lemma \ref{lm-component-cleaning} provides a
perspective $(\ppersp',Y',\bar{\gamma}')$ which satisfies
$|\ppersp'\cdot S|<|\ppersp\cdot S|$, and moreover
\begin{multline}
  \dim(Y)=\dim(Y') \geq \dim(\ppersp'\cdot B) + |\ppersp'\cdot S| -
  \ppersp'\cdot m
  \\
  \geq \dim(\ppersp'\cdot B) -\ppersp\cdot m + \demi(|\ppersp'\cdot S|
  + |\ppersp\cdot S|)
  \\
  = \dim(\ppersp'\cdot B) - \dim(\ppersp\cdot B)+ \demi(|\ppersp'\cdot
  S|- |\ppersp\cdot S|) + \dim(Y)
\end{multline}
hence
$$
2\dim(\ppersp'\cdot B) -2 \dim(\ppersp\cdot B)  \leq |\ppersp\cdot S|
- |\ppersp'\cdot S|,
$$
which because of~(\ref{eq.1}) implies
$$
2 \dim(\ppersp'\cdot B) + 2 \dim(Y)+|\ppersp'\cdot S| \leq
6l.
$$
On the other hand, we have
$$
\dim(Y')+|\ppersp'\cdot S|<\dim(Y)+ |\ppersp\cdot S|,
$$
again contradicting the assumption of minimality.
 
Finally, the remaining case when $\ppersp\cdot S$ is empty is trivial:
in that case, $Y$ is a closed subscheme of $\ppersp\cdot B$ so that
$$
4 \dim(X) \leq 4 \dim(Y) \leq 2 \dim(\ppersp\cdot B) + 2\dim(Y)\leq 6l
$$
and we may simply take the trivial perspective
$(\ppersp_1,X,\bar{\eta})$ of Lemma~\ref{lm-perspective-example} (2),
for which
$$
2l-|\ppersp_1\cdot S|+2\ppersp_1\cdot m=2l\leq 4(2l-\dim(X)).
$$
\end{proof}


\section{The generic statement}
\label{sec-generic}




We continue with the previous notation.  Fix $j\geq 0$. Let
$X\subset X_j\subset \Aa^{2l}- \mathcal V^\Delta$ be an irreducible
component of $X_{j}$ over $\Zz$ which intersects the characteristic
zero part. Let $\overline{X}$ be the closure of $X$ in $\Aa^{2l}$.

Fix a perspective $(\ppersp, Y, \overline{\gamma})$ on $X$ such that
$\mathcal Y_\ppersp$ is irreducible, $\overline{\gamma}$ is a
geometric generic point of $Y$, and
$2l - |S| + 2m\leq 4\codim_{\Aa^{2l}}(X)$, which exists by
Theorem~\ref{thm-ag}. By definition, all of the perspective data is
defined over $\Qbar$. However, by standard finiteness arguments,
everything is necessarily defined over a finitely generated subring of
$\Qbar$, i.e. over a ring $\Oc_K[1/N]$, where $\Oc_K$ is the ring of
integers of a number field $K$ and $N\geq 1$ is some integer. We will
use the same notation $Y, C_i, b_i$, etc. to refer to the objects over
this ring. Since, by assumption, $\mathcal{Y}_{\ppersp,\Qbar}$ is
irreducible, and equal to $Y_{\Qbar}$, we deduce that
$\mathcal{Y}_{\ppersp}$ is geometrically irreducible and equal to $Y$.

Because the geometric generic point of $Y$ is a lift of the geometric
generic point of $X$, the image of $Y$ in $\Aa^{2l}$ is a dense subset
of $\overline{X}$. For all but finitely many prime ideals $\pi$ of
$\mathcal O_K[1/N]$, with residue field denoted $\Fq$, the variety
$Y_{\bFq}$ is irreducible and nonempty, $X_{\bFq}$ is irreducible and
nonempty, and the map $Y_{\bFq} \to \overline{X}_{\bFq}$ is dominant.
In the remainder of this section, we only consider finite fields $\Fq$
arising in this manner, and we also always assume that the
characteristic of $\Fq$ is $>2k+1$.

\begin{lemma}\label{lm-final-dio}
  Assume that $\uple{\chi}$ has \CGM.  If $p$ is large enough with
  respect to $(k,l,X)$ and $\dim(X_{\Qq})\geq (3l+1)/2$, then we have
\begin{multline}\label{eq-compare-dio}
  \sum_{y\in Y(\Fq)} \sum_{r\in\Fqt} \Bigl| \sum_{s\in\Fqt}
  \prod_{i=1}^l \hypk_k(r(s+b_i(y));\uple{\chi},q)
  \overline{\hypk_k(r(s+b_{i+l}(y));\uple{\chi},q)}
  \Bigr|^2=\\
  \sum_{y\in Y(\Fq)} \sum_{r\in\Fqt} \sum_{s\in\Fqt} \Bigl|
  \prod_{i=1}^l \hypk_k(r(s+b_i(y));\uple{\chi},q)
  \overline{\hypk_k(r(s+b_{i+l}(y));\uple{\chi},q)}
  \Bigr|^2+O(q^{\dim(X_{\Qq})+3/2}),
\end{multline}
where the implied constant depends only on $(\ppersp,k,l)$.
\end{lemma}

\begin{proof}
  We first fix $r\in\Fqt$. We apply Proposition~\ref{lm-dio-new} with
  data $(m,B, S, (C_i), \uple{(f_j)}, (g_j))$ coming from the
  perspective datum $\ppersp$, $A= \Gg_m$, and the sheaf $\sheaf{F}_i$
  is
  $$
  [(b_i,s)\mapsto s(r+b_i)]^*\HYPK_{k,\psi}(\uple{\chi}).
  $$

  Assumption (TI) holds by a Goursat-Kolchin-Ribet argument
  (see~\cite{ESDE} and~\cite{FKMSP}). Indeed, each irreducible
  component of $C_{i,\rho}$ is a geometrically irreducible curve on
  which $b_i$ is a nonconstant function. The sheaf
  $\sheaf{F}_{i,\rho,s_1} \otimes \sheaf{F}^\vee_{i, \rho,s_2}$ is the
  pullback along $b_i$ of the sheaf
  $$
  \sheaf{H}= [b_i \mapsto (s_1 (r+b_i))]^* \HYPK_{k,\psi}(\uple{\chi})
  \otimes [ b_i \mapsto (s_1 (r+b_i))]^*
  \HYPK_{k,\psi}(\uple{\chi})^\vee.
$$
The monodromy group after pulling back along the map $b_i$ is a finite
index subgroup, so it suffices to show that no finite-index subgroup
of the geometric monodromy group of $\sheaf{H}$ admits a
one-dimensional irreducible component. However, by Goursat's lemma,
the geometric monodromy group of $\sheaf{H}$ is a product of two
copies of the monodromy group of $\HYPK_{k,\psi}(\uple{\chi})$, acting
by the tensor product of the standard representation with its
dual. This group is connected, so has no proper finite-index
subgroups, and does not admit a one-dimensional representation, which
proves the claim.

The conductor of all the sheaves $\mcF_{i,\rho,s}$, which are
pullbacks of (shifted and translated) generalized Kloosterman sheaves
are bounded by constants depending only on $\ppersp$.

Applying Proposition~\ref{lm-dio-new} we obtain
\begin{multline*}
  \sum_{y\in Y(\Fq)} \Bigl| \sum_{s\in\Fqt} \prod_{i=1}^l
  \hypk_k(r(s+b_i(y));\uple{\chi},q)
  \overline{\hypk_k(r(s+b_{i+l}(y));\uple{\chi},q)}
  \Bigr|^2=\\
  \sum_{y\in Y(\Fq)} \sum_{s\in\Fqt} \Bigl| \prod_{i=1}^l
  \hypk_k(r(s+b_i(y));\uple{\chi},q)
  \overline{\hypk_k(r(s+b_{i+l}(y));\uple{\chi},q)} \Bigr|^2+O(q^{
    \dim B_{\Qq} + |S|/2 + 2}),
\end{multline*}
where the implied constant depends only on $(\ppersp, k,l)$.

Summing over $r$, we get the formula~(\ref{eq-compare-dio}), except
that the error term is $O(q^{ \dim B_{\Qq} + |S|/2 + 3})$. However,
since $X$ is the vanishing set of $m$ equations in a fiber product of
$|S|$ curves over $B$, we have
\begin{align*}
  \dim X_{\Qq} \geq \dim B_{\Qq} + |S| - m 
  &= \dim B_{\Qq} + \frac{|S|}{2} + l
    - \Bigl(l + m - \frac{|S|}{2}\Bigr)\\
  & \geq \dim B_{\Qq} + \frac{|S|}{2} + l - 2 (2l- \dim X)
    \geq \dim B_{\Qq} +\frac{|S|}{2} + 1/2,
\end{align*}
where the last two inequalities holds by the assumption on the
perspective and the assumption on $\dim X$, respectively.
\end{proof} 

Let $\eta$ be the generic point of $X_{\bFq}$ and let $\bar{\eta}$ be
a geometric generic point over $\eta$.  Let $\eta'$ be the the generic
point of $Y_{\Fq}$. We fix a $k$-tuple $\uple{\chi}$ of characters of
$\Fqt$.

\begin{lemma}\label{lm-equal-3}
  Assume that $\uple{\chi}$ has \CGM. We have
\begin{equation}
  \dim \End_{V_{\eta'\times\bFq}}(\mcK_{\eta'\times\bFq})=
  \dim \End_{U_{\eta'\times\bFq}}(\mcR^*_{\eta'\times\bFq}).
\end{equation} 
\end{lemma}

\begin{proof}  
  Let $Y^\circ$ be the smooth locus of $Y$.  The endomorphisms
  $ \End_{V_{\eta'\times\bFq}}(\mcK_{\eta'\times\bFq})$ are the same
  as the endomorphisms of the pullback of $\mcK$ to
  $Y_{\bFq}^\circ \times_{\Aa^{2l}} V$, because the monodromy
  representations of both sheaves are the same (as they are normal,
  with the same generic point). We calculate the endomorphisms by
  applying Lemma~\ref{lm-dio}, obtaining the lim-sup of
$$ q^{- \dim X -2}
\sum_{y\in Y(\Fq)} \sum_{r\in\Fqt} \Bigl| \sum_{s\in\Fqt}
\prod_{i=1}^l \hypk_k(r(s+b_i(y));\uple{\chi},q)
\overline{\hypk_k(r(s+b_{i+l}(y));\uple{\chi},q)}
\Bigr|^2.
$$
\par
We do the same for
$\End_{U_{\eta'\times\bFq}}(\mcR^*_{\eta'\times\bFq})$, obtaining the lim-sup of
$$ q^{- \dim X -2}
\sum_{y\in Y(\Fq)} \sum_{r\in\Fqt} \sum_{s\in\Fqt} \Bigl|
\prod_{i=1}^l \hypk_k(r(s+b_i(y));\uple{\chi},q)
\overline{\hypk_k(r(s+b_{i+l}(y));\uple{\chi},q)} \Bigr|^2.
$$
By Lemma~\ref{lm-final-dio}, these two quantities are equal up to
$O(q^{-1/2})$, and therefore their limsups are equal.
\end{proof}

In the remainder of this section, we will prove an analogous statement
with $\overline{\eta}$ instead of $\eta'$.  The method is to prove
that 
\begin{equation}\label{eq-equal-1}
  \dim \End_{V_{\bar{\eta}}}(\mcK_{\bar{\eta}})=
  \dim \End_{V_{\eta'\times\bFq}}(\mcK_{\eta'\times\bFq})
\end{equation}
and 
\begin{equation}\label{eq-equal-3}
  \dim \End_{U_{\eta'\times\bFq}}(\mcR^*_{\eta'\times\bFq})=
  \dim \End_{U_{\bar{\eta}}}(\mcR^*_{\bar{\eta}}).
\end{equation}

We will prove~(\ref{eq-equal-1}) immediately. The
formula~(\ref{eq-equal-3}) is more difficult, and its proof will use
vanishing cycles.

\begin{proposition}\label{pr-equal-1}
  Assume that $\uple{\chi}$ has \CGM.  For any extension $\eta'$ of
  $\eta$ we have
$$
\dim \End_{V_{\bar{\eta}}}(\mcK_{\bar{\eta}})=
\dim \End_{V_{\eta'\times\bFq}}(\mcK_{\eta'\times\bFq}).
$$
\end{proposition}

\begin{proof}
  Let $G$ be the geometric monodromy group of $\mcK$, and let $B$ be
  the set of distinct values of $b_1,\dots,b_{2l}$ at $\eta$. Then
  certainly the arithmetic monodromy group of $\mcK_{\eta\times \bFq}$
  is contained in $G^{|B|}$. By Goursat-Kolchin-Ribet, the geometric
  monodromy group of $\mcK_{\eta \times \bFq}$ is $G^{|B|}$, so the
  arithmetic and geometric monodromy groups are equal.  Therefore
  $\Gal(\overline{\eta}/\eta \times \bFq)$ acts trivially on
  $\End_{V_{\eta \times\bFq}}(\mcK_{\eta'\times\bFq})$ as this action
  factors through the quotient of the arithmetic monodromy group by
  the geometric monodromy group. It follows that
  $\Gal(\overline{\eta}/\eta' \times \bFq)$ acts trivially and so
  $ \End_{V_{\eta'\times\bFq}}(\mcK_{\eta'\times\bFq})$, which is the
  space of invariants of that action, is equal to the whole space.
\end{proof}

In order to prove~(\ref{eq-equal-3}), we first introduce some
notation. We write $\widetilde{\eta}=\eta\times\bFq$. We consider the
projective line $\Pp^1_{\widetilde{\eta}}$ with coordinate $r$. We denote
by $\Oc^{et}$ the étale local ring of $\Pp^1_{\widetilde{\eta}}$ at $\infty$ and
by $K$ its field of fractions. We will often identify $K$ (resp. a
separable closure $K^{sep}$ of $K$) with the corresponding spectra.

What follows is the key lemma.

\begin{lemma}\label{lm-key-generic}
  With assumptions as above, the action of $\Gal(K^{sep}/K)$ on
  $\mcR^*_{K^{sep}}$ is unipotent.
\end{lemma}

Note that to make sense of this action, we use the fact that the image
of the natural morphism $\Spec(K)\to \Aa^{1+2l}$ has image in $U$,
which follows from Lemma~\ref{z-polynomial}.

\begin{proof}
  We denote by $\sigma$ the special point of $\Spec(\Oc^{et})$. We
  consider the projective line $\Pp^1_{\Oc^{et}}$, with coordinate
  $t$, and denote by $j$ (resp. by $g$) the open immersion
  $\Gg_{m,\Oc^{et}}\to \Pp^1_{\Oc^{et}}$ (resp. the open immersion
  $\Aa^1_{\Oc^{et}}\to \Pp^1_{\Oc^{et}}$).

We consider the lisse sheaf 
$$
\widetilde{\mcK}=\bigotimes_{1\leq i\leq l}
\HYPK_{k,\psi}(\uple{\chi})(t(1+b_i/r))\otimes
\HYPK_{k,\psi}(\uple{\chi})(t(1+b_{i+l}/r))^{\vee}
$$
on $\Gg_{m,\Oc^{et}}$. 

By the change of variable $t=rs$ and the proper base change theorem,
the $\Gal(K^{sep}/K)$-action on $\mcR_{K^{sep}}$ is isomorphic to the
action on $H^1(\Pp^1_{K^{sep}},j_!\widetilde{\mcK})$. Since
$\mcR^*_{K^{sep}}$ is a quotient of $\mcR_{K^{sep}}$, the lemma will
follow if we prove that the action of $\Gal(K^{sep}/K)$ on
$H^1(\Pp^1_{K^{sep}},j_!\widetilde{\mcK})$ is unipotent.

By the long exact sequence for vanishing cycles, we have a long exact
sequence
\begin{equation}\label{eq-exact}
\cdots \to H^i(\Pp^1_{\sigma},j_!\widetilde{\mcK})\to
H^i(\Pp^1_{K^{sep}},j_!\widetilde{\mcK})\to H^i(\Pp^1_{\sigma},R\Phi
j_!\widetilde{\mcK})\to \cdots
\end{equation}

For each $i$, we have an isomorphism
$$
H^i(\Pp^1_{\sigma},j_!\widetilde{\mcK})=
H^i\Bigl(\Pp^1_{\sigma},j_!\Bigl(\HYPK_{k,\psi}(\uple{\chi})^{\otimes
  l}\otimes (\HYPK_{k,\psi}(\uple{\chi})^{\vee})^{\otimes
  l}\Bigr)\Bigr),
$$
hence the $\Gal(K^{sep}/K)$-action on these spaces is trivial.

On the other hand, the vanishing cycle complex
$R\Phi j_!\widetilde{\mcK}$ is zero away from the point at $\infty$ of
$\Pp^1_{\sigma}$ (local acyclicity of smooth morphisms and lisseness
of $j_!\widetilde{\mcK}$) and is zero at $0$ (because of tame
ramification and Deligne's semicontinuity theorem).

We therefore only need to understand $R\Phi j_!\widetilde{\mcK}$ at
$t=\infty$. By the second part of Lemma~\ref{lm-tilde-psi}, the local
monodromy at infinity of $j_!\widetilde{\mcK}$ is isomorphic to that
of a direct sum of sheaves of the form
$$
\mcL_{\widetilde{\psi}}\Bigl( 
    (t(1+b_1/r))^{1/k}
    +\sum_{i=2}^{2k}\eps_i \zeta_i (t(1+b_i/r))^{1/k}
\Bigr).
$$
Since $(1+b_i/r)^{1/k}$ belongs to the étale local ring $\Oc^{et}$,
this is isomorphic to the local monodromy of a direct sum of sheaves
of the form $\mcL_{\widetilde{\psi}}(\gamma(r)t^{1/k})$. We have
$\mcL_{\widetilde{\psi}}(\gamma(r)t^{1/k})=\varpi_*\mcL_{\widetilde{\psi}}(\gamma(r)u)$
where $\varpi$ is the finite covering $u\mapsto u^k$. We compute the
local monodromy at $\infty$ of this sheaf, which we
denote~$\mcG$. This is a standard computation. We use the long exact
sequence
$$
\cdots \to H^i(\Pp^1_{\sigma},g_!\sheaf{G})\to
H^i(\Pp^1_{K^{sep}},g_!\sheaf{G})\to H^i(\Pp^1_{\sigma},R\Phi
g_!\sheaf{G})\to \cdots
$$
and distinguish three cases:
\begin{enumerate}
\item If $\gamma(r)=0$ in $\Oc^{et}$, then $\sheaf{G}$ is tamely ramified at
  $\infty$, so the vanishing cycles vanish.
\item If $\gamma(r)\not=0$ in $\Oc^{et}$ but $\gamma(r)=0$ at the special
  point, then all $H^i$'s with coefficients in $g_!\sheaf{G}$ in the
  above exact sequence vanish except
$$
H^2(\Pp^1_{\sigma},g_!\sheaf{G}),
$$
which is one-dimensional with a trivial action of $\Gal(K^{sep}/K)$;
this implies that the action on
$H^i(\Pp^1_{\sigma},R\Phi g_!\sheaf{G})$ is trivial.
\item If $\gamma(r)\not=0$ at $\sigma$, then all cohomology groups in
  the sequence vanish by properties of the Artin-Schreier sheaves.
\end{enumerate}

In any of the three cases, by local acyclicity of smooth morphisms we
see that $R\Phi g_! \sheaf{G}$ vanishes outside the point at $\infty$,
so knowing that $H^i(\Pp^1_{\sigma},R\Phi g_!\sheaf{G})$ has trivial
Galois action implies that the Galois action on the stalk at $\infty$
vanishes.

Since the
vanishing cycle functor is additive and commutes with finite
pushforward, we conclude that $\Gal(K^{sep}/K)$ acts trivially on 
$H^i(\Pp^1_{\sigma},R\Phi j_!\widetilde{\mcK})$ for all $i$, hence by the
exact sequence~(\ref{eq-exact}), this group acts unipotently on
$H^i(\Pp^1_{K^{sep}},j_!\widetilde{\mcK})$, as desired.
\end{proof}

\begin{proposition}\label{pr-equal-3}
  Assume that $\uple{\chi}$ has \CGM. We have
$$
\dim \End_{U_{\eta'\times\bFq}}(\mcR^*_{\eta'\times\bFq})=
\dim \End_{U_{\bar{\eta}}}(\mcR^*_{\bar{\eta}}).
$$
\end{proposition}

\begin{proof}  We first note that we have an inclusion
$$
\End_{U_{\bar{\eta}}}(\mcR^*_{\bar{\eta}})
\subset 
\mcR^*_{\bar{\eta}}\otimes (\mcR^*_{\bar{\eta}})^{\vee}.
$$
Moreover, we have a commutative triangle
$$
\begin{tikzcd}
  \Gal(K^{sep}/K)\arrow{r}\arrow{rd}{\alpha} &
  \arrow{d}\pi_1(U_{\eta}))\\
  & \Gal(\bar{\eta}/\widetilde{\eta})
\end{tikzcd}
$$
where $\alpha$ is surjective because $K$ does not contain a finite
extension of $\widetilde{\eta}$.

The fundamental group $\pi_1(U_{\eta})$ acts on
$\mcR^*_{\bar{\eta}}\otimes (\mcR^*_{\bar{\eta}})^{\vee}$ and the
Galois group $\Gal(\bar{\eta}/\widetilde{\eta})$ acts on 
$\End_{U_{\bar{\eta}}}(\mcR^*_{\bar{\eta}})$, and these actions are
compatible with the inclusion above.

By Lemma~\ref{lm-key-generic}, the action of $\Gal(K^{sep}/K)$ on
$\mcR^*_{\bar{\eta}}\otimes (\mcR^*_{\bar{\eta}})^{\vee}$ is
unipotent, hence the action of $\Gal(\bar{\eta}/\widetilde{\eta})$ on 
$\End_{U_{\bar{\eta}}}(\mcR^*_{\bar{\eta}})$ is also unipotent since
$\alpha$ is surjective. But we know, by purity, that this action is
semisimple, and it follows that the action
$\Gal(\bar{\eta}/\widetilde{\eta})$ on
$\End_{U_{\bar{\eta}}}(\mcR^*_{\bar{\eta}})$ is in fact trivial. In
particular, we have
$$
\dim \End_{U_{\eta'\times\bFq}}(\mcR^*_{\eta'\times\bFq})=
\dim \End_{U_{\bar{\eta}}}(\mcR^*_{\bar{\eta}}).
$$ \end{proof} 

Finally, we can deduce:

\begin{theorem}\label{th-generic}
  Let $X$ be an irreducible component of $X_j$ which intersects the
  characteristic zero part. Assume that $p$ is a prime sufficiently
  large with respect to $(k,l,X)$. Let $\Fq$ be a finite field of
  characteristic $p$, and let $\overline{\eta}$ be the geometric
  generic point of $X_{\bFq}$. Suppose that $X$ has dimension at least
  $(3l+1)/2$.  Let $\uple{\chi}$ be a $k$-tuple of characters of
  $\Fqt$ with Property \CGM. Then we have
$$
\dim \End_{V_{\bar{\eta}}}(\mcK_{\bar{\eta}})=
\dim \End_{U_{\bar{\eta}}}(\mcR^*_{\bar{\eta}}).
$$
\end{theorem}

\begin{proof} 
  Since the assertion is geometric, we may replace $\Fq$ by a finite
  extension that is a residue field of the base $\Oc_K[1/N]$ of the
  ``spread-out'' perspective. The equality then follows, when the
  characteristic of $\Fq$ is sufficiently large in terms of $(k,l,X)$,
  by combining Proposition \ref{pr-equal-1}, Lemma~\ref{lm-equal-3}
  and Proposition \ref{pr-equal-3}.
\end{proof}

\section{Conclusion of the proof}\label{sec-end}

We recall that we want to prove Theorem~\ref{th-thetab}, which we
restate for convenience:

\begin{theorem}\label{th-end} 
  Assume that $\uple{\chi}$ has \CGMT.  If $p$ is large enough,
  depending only on $k,l$, then for any
  $\uple{b}\in \Aa^{2l}(\Fq)-\mcW(\Fq)$, the natural morphism
  $\theta_{\uple{b}}$ is an isomorphism.
\par
Furthermore, each irreducible component of $\mcR^*_{\uple{b}}$ has
rank greater than one.
\end{theorem}

\begin{proof} 
  Since $\uple{\chi}$ has \CGMT, by Lemma \ref{how-to-twist} there
  exists a character $\xi$, possibly over a finite extension
  $\Ff_{q^{\nu}}$ of $\Fq$, such that $\uple{\chi}'=\xi\uple{\chi}$
  has \CGM\ over $\Ff_{q^{\nu}}$. Consider $\uple{\chi}$ as a tuple of
  characters of $\Ff_{q^{\nu}}^{\times}$. Then
  $\HYPK_{k,\psi}(\uple{\chi}')=\mcL_{\xi}\otimes
  \HYPK_{k,\psi}(\uple{\chi})$, and it follows that the auxiliary
  sheaves $\mcK$ and $\mcR^*$ for $\uple{\chi}$ are obtained from
  those associated to $\uple{\chi}'$ by twisting by a rank $1$ sheaf
  $\mcL_{\xi} ( (r+b_1) \dots (r+b_l) (r+b_{l+1}^{-1}) \dots
  (r+b_{2l})^{-1})$. Then the corresponding endomorphism rings (and
  the morphism $\theta_{\uple{b}}$) are the same for $\uple{\chi}$ and
  $\uple{\chi}'$. Up to renaming the field, this implies that we may
  as well assume that $\uple{\chi}$ has \CGM\ over $\Fq$.

  Let $\uple{b}\in \Aa^{2l}(\Fq)-\mcW(\Fq)$ be a point.  Let $j$ be
  the minimum $j$ such that $\uple{b} \in X_j$. Let $X$ be an
  irreducible component of $X_j$ containing $\uple{j}$. By taking $q$
  sufficiently large, we may assume that $X$ intersects the
  characteristic zero part. As the set of irreducible components is
  finite and depends only on $k,l$, the minimum value for $q$ depends
  only on $k,l$.

  If the dimension of $X$ is less than $(3l+1)/2$, then
  $\uple{b} \in X \subseteq \mcW$.

  Otherwise, let $\eta$ be the generic point of $X$.  Then by
  Theorem~\ref{th-generic}, taking $q$ sufficiently large,
  $$ \dim \End_{V_{\bar{\eta}}}(\mcK_{\bar{\eta}})=
  \dim \End_{U_{\bar{\eta}}}(\mcR^*_{\bar{\eta}}).
$$

Because $\mcW_1$ has dimension $\leq l+1$, and
$\dim X \geq (3l+1)/2 > l+1$ as $l>1$, $\eta$ is not contained in
$\mcW_1$. By Lemma~\ref{strata-finite-etale}, $Z$ is finite \'{e}tale
over $X_j - X_{j-1}$. Because the $b_i$ are sections of $Z$, and
$\uple{b}$ is a specialization of $\eta$ inside $X_j - X_{j-1}$, any
two of the $b_i$ which are unequal over $\eta$ must remain unequal
over $\uple{b}$, so $\uple{b}\not \in \mcW_1$.

So by Theorem~\ref{th-inj}, the natural map
$$
\theta_{\uple{b}}: \End_{V_{\uple{b}}}(\mcK_{\uple{b}})
\to \End_{U_{\uple{b}}}(\mcR^*_{\uple{b}})
$$
is injective, hence by Proposition \ref{pr-specialization},
$\theta_{\uple{b}}$ is an isomorphism.

Each irreducible component of $\mcR^*_{\uple{b}}$ is the image of an
idempotent element of $ \End_{U_{\uple{b}}}(\mcR^*_{\uple{b}})$, which
because $\theta_{\uple{b}}$ is an isomorphism is induced by an
idempotent element of $ \End_{V_{\uple{b}}}(\mcK_{\uple{b}})$, and
thus is equal to the weight one part of the cohomology of the image of
that idempotent element of $ \End_{V_{\uple{b}}}(\mcK_{\uple{b}})$. In
other words, it is the weight one part of the cohomology of an
irreducible component of $\mcK_{\uple{b}}$. Hence by
Lemma~\ref{lm-injectivity-plus-epsilon}, its rank is at least two.
\end{proof}

We finally can conclude the proof by showing how Theorem~\ref{th-end}
allows us to give the estimates for complete sums used in the proof of
our main theorems. In both cases, we use the fact (as remarked before
the statements of Theorem~\ref{th-complete-1} and~\ref{th-complete-2})
that we may assume that the function $K$ is
$\Kl_k(x;\uple{\chi},q)$. By Lemma \ref{lm-top-vanishing} and the
Grothendieck--Lefschetz trace formula, for any
$\uple{b} \not\in \mcV^\Delta$, the function $\bfR$ is equal to minus
the trace function of the sheaf $\mcR$, if the additive character
$\psi$ is chosen so that $\psi(x)=e(x/q)$ for $x\in\Fq$.

\begin{proof}[Proof of Theorem~\ref{th-complete-1}]\label{pg-th-c1}
  We have defined $\mcV^{\Delta}$ and $\mcW$, and they satisfy the
  codimension bounds stated in the theorem (see~(\ref{eq-codim-w})).
 \par
 We need to estimate the complete sums
$$
\Sigma_{II}(\bfb)=
\sum_{r\in\Fq}|\bfR(r,\bfb)|^2-\sum_{s\in\Fqt}\sum_{r\in\Fq}
|\bfK(sr,s\bfb)|^2
$$
for $\bfb\in\Ff_q^{2l}$.  Since $\Kl_k$ is bounded, we have
$\Sigma_{II}(\bfb)\ll q^{3}$ for all $\bfb$, which is the trivial
bound~(\ref{eq-complete-11}).
\par
If $\bfb\in \mcW(\Fq)$ and $\bfb\notin \mcV^{\Delta}(\Fq)$, then we
obtain $\Sigma_{II}(\bfb)\ll q^2$ by estimating the two terms in
$\Sigma_{II}$ separately, and using the Riemann Hypothesis together
with the fact that the $\mcR$-sheaf is mixed of weights $\leq 1$ on
$\Aa^{2l}-\mcV^{\Delta}$, and the $\mcK$-sheaf is pure of weight
$0$. This proves~(\ref{eq-complete-12}).
\par
Now assume that $\bfb\notin \mcW(\Fq)$. By Theorem~\ref{th-end}, the
Frobenius-equivariant map
$$
\theta_{\uple{b}}\colon \End_{V_{\uple{b}}}(\mcK_{\uple{b}})\to 
\End_{U_{\uple{b}}}(\mcR^*_{\uple{b}})
$$ 
is an isomorphism. In particular the Frobenius automorphism of $\Fq$
has the same trace on both spaces. The trace on
$\End_{U_{\uple{b}}}(\mcR^*_{\uple{b}})$ is, by the
Grothendieck--Lefschetz trace formula, equal to
$$
\sum_{r\in\Fq}|\bfR(r,\bfb)|^2+O(q^{3/2})
$$
where the error term arises from the contribution of the
$H^1_c$-cohomology and of the weight $<1$ part of $\mcR$. Similarly,
the trace of Frobenius on $\End_{V_{\uple{b}}}(\mcK_{\uple{b}})$ is
equal to
$$
\sum_{s\in\Fqt}\sum_{r\in\Fq}|\bfK(sr,s\bfb)|^2+O(q^{3/2})
$$
where the error term arises from the contribution of the
$H^1_c$-cohomology. Comparing, we obtain~(\ref{eq-complete-13}).
\par
It remains to observe that, in all these estimates, the implied
constant depends only on the sum of the Betti numbers of the relevant
sheaves. These are estimated in the usual way by reducing to
expressions as exponential sums and applying the Betti number bounds
of Bombieri--Katz (see~\cite[Th. 12]{KatzBetti}
and~\cite[Prop. 4.24]{KMS} for the analogue argument in our previous
paper).
\end{proof}

\begin{proof}[Proof of Theorem~\ref{th-complete-2}]\label{pg-th-c2}
We recall that we need to estimate 
$$
\Sigma_I(\bfb)= \sum_{r\in\Fq}\bfR(r,\bfb)
$$
(see~(\ref{eq-def-sigma})).  Since $\Kl_k$ is bounded, we have
$\Sigma_{I}(\bfb)\ll q^{2}$ for all $\bfb$, which is the trivial
bound~(\ref{eq-complete-21}).
\par
If $\bfb\in \mcW(\Fq)$ and $\bfb\notin \mcV^{\Delta}(\Fq)$, then we
obtain $\Sigma_{I}(\bfb)\ll q^{3/2}$ because the $\mcR$-sheaf is of
weights $\leq 1$ on $\Aa^{2l}-\mcV^{\Delta}$. This
proves~(\ref{eq-complete-22}).
\par
Finally, if $\bfb\notin \mcW(\Fq)$, then we obtain
$\Sigma_I(\bfb)\ll q$ straightforwardly from Deligne's Riemann
Hypothesis, since $\mcR^*$ is of weight $1$ and has no geometrically
trivial irreducible component (by Theorem~\ref{th-end} it doesn't even
have rank $1$ components), proving~(\ref{eq-complete-23}).
\par
Again, the implied constants in these estimates depend only on the sum
of the Betti numbers of the relevant sheaves, and are estimated by
reducing to expressions as exponential sums and applying the Betti
number bounds of Bombieri--Katz~\cite{KatzBetti}.
\end{proof}

\end{document}